\documentclass[11pt,reqno]{amsart}
\usepackage{amsmath,amsfonts,amsthm,color,amssymb, mathrsfs}

\usepackage{pxfonts}
\makeatletter
\@tfor\@tempa:=
\alpha \beta \gamma \delta \epsilon \zeta \eta \theta \iota \kappa \lambda \mu
\nu \xi \pi \rho \sigma \tau \upsilon \phi \chi \psi \omega \varepsilon
\vartheta \varpi \varrho \varsigma \varphi  
\do {% for some reason, must proceed in two steps
	\count255 \numexpr\@tempa+"7000\relax
	\expandafter\mathchardef\@tempa \count255 }
\makeatother

\usepackage{xcolor}
\usepackage[T1]{fontenc}
\usepackage{wasysym}
\usepackage[normalem]{ulem}
\usepackage{stmaryrd} 

\usepackage[left=1in, right=1in, top=1.1in,bottom=1.1in]{geometry}
\usepackage{enumitem}
\usepackage{hyperref}
\hypersetup{
	colorlinks   = true,
	citecolor    = blue,
	linkcolor=blue
}
\allowdisplaybreaks
\usepackage{csquotes}
\usepackage{graphicx}
\usepackage{pstricks}
\usepackage{lmodern}
\newtheorem{thm}{Theorem}[section]

\newtheorem{lem}[thm]{Lemma}

\newtheorem{defn}[thm]{Definition}

\newtheorem{rem}[thm]{Remark}

\def \a{{\alpha}}

\def \G{{\Gamma}}
\def \d{{\mathrm{d}}}

\newcommand{\Om}{\Omega}

\newcommand{\cF}{\mathcal{F}}

\newcommand{\RR}{\mathbb{R}}
\newcommand{\PP}{\mathbb{P}}

\newcommand{\E}{\mathbb{E}}

\newcommand{\bean}{\begin{eqnarray*}}
	\newcommand{\eean}{\end{eqnarray*}}

\newcommand{\EE}{\mathbb{E}}

\newcounter{bean}
\newcommand{\benuma}{\setlength{\labelwidth}{.25in}
	
	\begin{list}
		{(\alph{bean})}{\usecounter{bean}}}
	\newcommand{\eenuma}{\end{list}}

\begin{document}
	\title[Volterra type McKean-Vlasov stochastic integral equations]{Asymptotic behaviors for Volterra type McKean-Vlasov stochastic integral equations with small noise }
	\author[S. Liu]{Shanqi Liu}
	\address{School of Mathematical Science, Nanjing Normal University, Nanjing 210023, China}
	\email{shanqiliumath@126.com}
	
	\author[Y. Hu]{Yaozhong Hu}
	\address{Department of Mathematical and Statistical Sciences, University of Alberta, Edmonton T6G 2G1, Canada}
	\email{yaozhong@ualberta.ca}
	
	\author[H. Gao]{Hongjun Gao}
	\address{School of Mathematics, Southeast University, Nanjing 211189, China}
	\email{hjgao@seu.edu.cn (Corresponding author)}
	\vspace{-2cm}
	\maketitle
	\vspace{-1cm}
	\begin{abstract}
		This  work is devoted to studying asymptotic behaviors for Volterra type McKean-Vlasov stochastic differential equations with small noise. By applying the weak convergence approach, we establish the large and moderate deviation principles. In addition,  we  obtain the central limit theorem and find the  
		Volterra integral equation  satisfied by the limiting process, which  involves  the Lions derivative of the drift coefficient.
		
		\medskip\noindent\textbf{Keywords.}Stochastic Volterra integral equations, McKean-Vlasov SDEs, Large deviation principles, Central limit Theorem, Moderate deviation principles.
		\smallskip
		
		\noindent\textbf{AMS 2020 Subject Classifications.} 60H20; 60F10; 60F05.
	\end{abstract}
	%\end{adjustwidth}
	\section{Introduction}
		Consider the following Volterra type McKean-Vlasov stochastic differential equations (SDEs) on $\mathbb{R}^{d}$:
	\begin{align}\label{Volterra MV eq}
		X^\varepsilon_{t,\xi}=\xi+\int_{0}^{t}K_1(t,s)b(s,X^\varepsilon_{s,\xi},\mathcal{L}_{X^\varepsilon_{s,\xi}})\d s+\sqrt{\varepsilon}\int_{0}^{t}K_2\left(t,s\right)\sigma(s,X^\varepsilon_{s,\xi},\mathcal{L}_{X^\varepsilon_{s,\xi}})\d W_s,\quad 0\le t\le T\,, 
	\end{align}
	where $\mathcal{L}_{X^{\varepsilon}_{t,\xi}}
	\in \mathcal{P}_2(\mathbb{R}^d)$ denotes the probability law of $X^{\varepsilon}_{t,\xi}$, $\varepsilon\in(0, 1]$ is a small parameter, $K_i(t,s)$, $i = 1, 2$ are two positive functions defined on the simplex $\Delta_T=\left\{0\le s<t\le T\right\}$ which may be
	singular, $W_t$
	is an  $\mathrm{m}$-dimensional Brownian motion defined on the classical Wiener space
	$(\Om, \cF, \PP)$, $\xi$ is a $\RR^d$ valued random variable independent of the Brownian motion $B$,  and the coefficients $\sigma:[0,T]\times\mathbb{R}^d\times \mathcal{P}_2(\mathbb{R}^d)\to\mathbb{R}^d\times \mathbb{R}^m$, $b:[0,T]\times\mathbb{R}^d\times \mathcal{P}_2(\mathbb{R}^d)\to\mathbb{R}^d$ are Borel measurable.  Here we use $\mathcal{P}_2(\mathbb{R}^d)$ to denote the set of all square integrable probability measures on $\RR^d$.  
	
	This paper aims to investigate the asymptotic behavior  for \eqref{Volterra MV eq}. More precisely, we will study the asymptotic behaviors of   
	$$\frac{X^{\varepsilon}_{t,\xi}-X^0_{t,\xi}}{\sqrt{\varepsilon}h(\varepsilon)}, \quad t\in [0,T],$$
	as $\varepsilon\to 0$, where $X^0$ satisfies the following deterministic Volterra integral equation
	(formally setting $\varepsilon=0$ in \eqref{Volterra MV eq}):
	\begin{align}\label{Volterra eq-limit eq}
		X^0_{t,\xi}=\xi+\int_{0}^{t}K_1(t,s)b(s,X^0_{s,\xi},
		\delta_{X^0_{s,\xi}})\d s\,, 
	\end{align}
	where $\delta$ is the Dirac point measure.  We focus particularly 	on the following three tasks. 
	%three scale asymptotics.
	\begin{enumerate} 
		\item[(i)]  The choice  $h(\varepsilon)=\frac{1}{\sqrt{\varepsilon}}$     provides the large deviation principle (LDP) with speed $\varepsilon^{-1}$ (e.g.  Theorem \ref{thm-LDP} in Section 3). 
		\item[(ii)] The choice  $h(\varepsilon)=1$ yields  the central limit theorem (CLT), i.e. $\frac{X^{\varepsilon}_{t,\xi}-X^0_{t,\xi}}{\sqrt{\varepsilon}}$  converges to a stochastic process which
		solves a Volterra integral equation involving the Lions derivative of the drift coefficient (e.g.  Theorems  \ref{CLT} in Section 4).
		\item[(iii)]   The choice  $h(\varepsilon)$ satisfying 
		$$h(\varepsilon)\to \infty, \ \sqrt{\varepsilon}h(\varepsilon)\to 0 \ \text{as}\ \varepsilon\to 0,$$
		fills  the gap between the CLT  and  LDP scales, that is, it 
		 provides the moderate deviation estimate with speed $h^2(\varepsilon)$ (e.g. Theorem \ref{MDP}
		 in Section 5). 
	\end{enumerate}
	%The aim of this paper is to study the asymptotic behaviors of \eqref{Volterra MV eq} as $\varepsilon\to 0$. It is worthwhile to mention that our model is very general, which covers the McKean-Vlasov SDEs driven by standard Brownian motion or fractional Brownian motion, Caputo fractional SDEs \cite{AMN}, and classical stochastic Volterra integral equations without distribution measure.   
	
	%$\bullet$ Following
	To complete the first two tasks we shall use  the weak convergence criteria presented in \cite{JP} 
	%, the first aim of this paper is to obtain LDP and MDP for Volterra-type McKean-Vlasov integral equations. 
	and the key point is to identify the corresponding controlled
	Volterra equations.
	% for Volterra-type McKean-Vlasov integral equations. 
	%
	%$\bullet$ The second aim of this paper is to obtain CLT for Volterra-type McKean-Vlasov integral equations. Contrary to the CLT for McKean-Vlasov SDEs, the 
	The main difficulty to complete the third task is to obtain  estimate of the form  $\E\Big(\sup_{t\in [0,T]} |Z^\varepsilon_{t,\xi}-Z_{t,\xi}|^p\Big)$ (see Section \ref{app to CLT} for the definition of $Z^\varepsilon_{t,\xi}$ and $ Z_{t,\xi}$)  in particular due to the singularity of kernels.
	%, as far as we know there is no direct method to deal with it. Inspired by the study of Euler schemes in Zhang \cite{ZX1}, Qiao first used extended Kolmogorov's continuity criterium to overcome this difficulty. 
	We shall apply extended Kolmogorov's continuity criteria (see e.g. \cite{ZX1}) to overcome this difficulty.  This approach extends  the approach of  Qiao \cite{Qiao} to distribution depepdent case. 
	%. We provide these two different approache s to prove CLT in Section \ref{app to CLT}.
	
In the absence of the  Volterra kernels, that is the kernels $K_1=K_2=1$, the equation \eqref{Volterra MV eq} is reduced to classical McKean-Vlasov SDEs (also called distribution dependent SDEs \cite{RW} or mean-field SDEs). 
	%The model of McKean-Vlasov SDEs, 
	In this case the equation was first  studied  by Kac in  \cite{Kac56, Kac59}, which describe limiting behaviors of individual particles in an interacting particle system of mean-field type when the number of particles goes to infinity (propagation of chaos) and has widely applied to    various fields.  We refer to \cite{HLL, LSZZ} and the references therein for more details.  For this type of  classical McKean-Vlasov SDEs,  the  small noise asymptotic behaviors have been extensively investigated in recent years.  Herrmann et al. \cite{HIP} obtained the LDP in path space equipped with the uniform norm, when the drift is   superlinear  but coercive,     and the  diffusion coefficient is a constant. Dos Reis et al. \cite{dST} obtained the LDP in both uniform
	and H\"older topologies under the assumption that the drift and diffusion coefficients satisfy some extra H$\mathrm{\ddot{o}}$lder continuity with respect to time.

	When the coefficients $b$ and $\sigma$ do not depend on the distribution $\mathcal{L}_{X^{\varepsilon}_{t,\xi}}$ of the solution process  \eqref{Volterra MV eq} is reduced to stochastic Volterra integral equations or Volterra type  SDEs.  
%	On the other hand, possibly singular kernels need to be controlled carefully, which causes fundamental technical difficulties. 
There are now a large amount of work dedicated to this type of equations. They can be driven by general semimartingales in addition to Brownian motion. When the kernel is singular, usually the coefficients needs to be nicer and when the kernel is more regular, the coefficients can be allowed to be more general.   Among the large  body of literature  let us mention   \cite{ALP, AKO, BM1,BM2, LHH1, PP, PS,Pro}.   
%Protter \cite{Pro} studied stochastic Volterra integral equations driven by a general semimartingale.
% Pardoux and Protter \cite{PP} further discussed anticipating stochastic Volterra equations by applying the Malliavin Calculus. In the past few years, stochastic Volterra integral equations have been intensively investigated, the authors in first introduced and studied affine Volterra processes with convolutional kernel. Ackermann et al \cite{AKO} further studied inhomogeneous adffine Volterra process. Pr\"{o}mel and Scheffels \cite{PS} studied one-dimensional
%	stochastic Volterra equations with locally Hölder continuous diffusion coefficients and sufficiently regular
%	kernels. 
%Li et al in \cite{LHH1} studied the asymptotic separation for stochastic Volterra integral equations with doubly singular kernels and in 
Let us also point out that the  Euler and Milstein schemes for these  equations with singular kernels
was  studied in \cite{LHH2}.  For these  Volterra-type SDEs with regular kernel,   a Freidlin–Wentzell’s
type LDP was established in   \cite{NR}   in continuous functions space  and  was later extended  in   \cite{Lak} to the Besov-Orlicz space.  The case of  singular kernels was studied in   \cite{ZX1}  and then in  \cite{ZX2}.  The MDP for one dimensional stochastic Volterra equations with regular
kernels the MDP   was studied in  \cite{LWYZ}.  
 More recently, Jacquier and Pannier \cite{JP} established LDP and MDP for stochastic Volterra systems with singular kernels under weaker conditions. Qiao \cite{Qiao} first proved CLT for stochastic Volterra-type SDEs with singular kernels. To the best of our knowledge, there are no LDP, CLT, and MDP results for Volterra-type McKean-Vlasov stochastic integral equations so far.

		Compared with classical SDEs,  Volterra type  SDEs  are more complex. On the one hand, these equations are in general neither Markovian nor semi-martingales, and include fractional Brownian motion with Hurst index $H\in(0,1)$ as a special case when  $b=0$,  $\sigma=1$,  and 
	\begin{align}
		K_2(t,s)=K^H(t,s)=c_H(t-s)^{H-\frac{1}{2}}+c_H(\frac{1}{2}-H)\int_{s}^{t}(\theta-s)^{H-\frac{3}{2}}\Big(1-(\frac{s}{\theta})^{\frac{1}{2}-H}\Big)\mathrm{d}\theta,\nonumber
	\end{align}
	with
	$$c_H=\Big(\frac{2H\G(\frac{3}{2}-H)}{\G(H+\frac{1}{2})\G(2-2H)}\Big)^{\frac{1}{2}}.$$
	
%	It is then  natural to study  Volterra-type McKean-Vlasov SDEs, which can be seen as the combination of McKean-Vlasov SDEs and Volterra type SDEs. 
Volterra-type McKean-Vlasov SDEs first appeared in  \cite{SWY},  in which the authors studied the well-posedness   for  a class of mean-field backward stochastic {V}olterra integral equations associated with   the  Pontryagin’s type maximum principle  for a controlled mean field Volterra system.  
%Recently, there has been a growing interest in the 
%	study of mean-field forward (backward) SDEs. In 
This result is then extended in \cite{WH23}
%to mean-field forward (backward) SDEs.
%, the authors extended 
% 	the results in \cite{SWY} 
%by considering 
to study the well-posedness and Pontryagin’s type 
 	maximum principle of  mean-field backward doubly stochastic Volterra integral 
	equation. Very recently,   %Pr\"{o}mel and Scheffels  consider 
	the well-posedness for Volterra type McKean-Vlasov SDEs is considered in (\cite[Theorem $2.3$]{PS1}) when the kernel  is singular  but the coefficients $b$ and $\sigma$ are    Lipschitzian, or when   kernel is regular  with  H$\mathrm{\ddot{o}}$lder  continuous coefficients  (\cite[Theorem $2.10$ or Theorem $1.3$]{JLZ}).  The result      \cite[Theorem $2.3$]{PS1} is extended in   \cite{LG}    to a more general case by applying Volterra-type Gronwall's inequality.

%	 
%	 There are several approaches toward the small $\varepsilon$ asymptotics.
%	 {\red  The idea in  \cite{ADRST,dST,HIP,RW-2} is to   replace the distribution $\mathcal{L}_{X^{\varepsilon}_t}$ of $X^{\varepsilon}_t$ in the coefficients with the Dirac measure $\delta_{X^0_{t}}$, and then to use time discretization approximation and exponent equivalence techniques. 
%	Another approach   is the weak convergence approach.  Suo and Yuan \cite{SY} established the CLT and MDP for McKean-Vlasov SDEs. Liu et al. \cite{LSZZ} established the LDP and MDP for McKean-Vlasov SDEs with jumps, where they originally established the weak convergence criteria with respect to distribution dependent case. And then,  Hong et al. \cite{HLL} obtained the LDP for distribution dependent SPDEs. Gu et al. \cite{GS} obtained the LDP for path-distribution-dependent SDEs. Fang et al. \cite{FLQZ} studied asymptotic behaviors of small perturbation for multivalued McKean-Vlasov SDEs. Fan et al. \cite{FYY} studied small noise asymptotic behaviors for McKean-Vlasov SDEs driven by fractional Brownian motions, where they established the weak convergence criteria
%	in the factional Brownian motion setting.} 
%	
%	 

	The rest of the paper is organized as follows. In Section \ref{The large deviation principle for Volterra type McKean-Valsov SDEs}, we establish the large deviation principle for Volterra type McKean-Vlasov SDEs. The CLT and MDP are presented in Section \ref{app to CLT} and Section \ref{MDP-1} respectively. To facilitate the reading of the paper,  we postpone some auxiliary lemmas to  Appendix \ref{Appendix}.
	
	Throughout this paper $C_p$ denotes a positive constant which may changes from line to line, where the subscript $p$ is used to emphasize that the constant depends on certain parameter(s)  $p$.
\section{Preliminaries}
	\subsection{Notations}
	Throughout this paper, we denote by  $\mathbb{R}^{d}$ the $n$-dimensional Euclidean space. Let $(\Omega,\mathcal{F},(\mathcal{F}_t)_{t\ge0},\mathbb{P})$ to be a given complete filtered probability space $(\Omega,\mathcal{F},(\mathcal{F}_t)_{t\ge0},\mathbb{P})$, where $\mathcal{F}_t$ is a nondecreasing family of sub-$\sigma$ fields of $\mathcal{F}$ satisfying the usual conditions. Let $\mathcal{P}$ be the space of all probability measures $\mu$ on $\mathbb{R}^{d}$, and let
	$$\mathcal{P}_{2}(\mathbb{R}^{d})=\left\{\mu \in \mathcal{P}\left(\mathbb{R}^{d}\right): \mu\left(|\cdot|^{2}\right):=\int_{\mathbb{R}^{d}}|x|^{2} \mu(\mathrm{d} x)<\infty\right\}.$$
	It is well known that $\mathcal{P}_{2}$ is a Polish space under the Wasserstein distance
	$$
	\mathbb{W}_{2}(\mu, \nu):=\inf _{\pi \in \mathscr{C}(\mu, \nu)}\left(\int_{\mathbb{R}^{m} \times \mathbb{R}^{m}}|x-y|^{2} \pi(\mathrm{d} x, \mathrm{d} y)\right)^{\frac{1}{2}}, \ \mu, \nu \in \mathcal{P}_{2}\left(\mathbb{R}^{d}\right),
	$$
	where $\mathscr{C}(\mu, \nu)$ is the set of couplings for $\mu$ and $\nu$; that is, $\pi \in \mathscr{C}(\mu, \nu)$ is a probability measure on $\mathbb{R}^{d} \times \mathbb{R}^{d}$ such that $\pi\left(\cdot \times \mathbb{R}^{d}\right)=\mu$ and $\pi\left(\mathbb{R}^{d} \times \cdot\right)=\nu$.
  
	For measurable functions $K$  and $L$  on $\mathbb{R}_{+}\times \mathbb{R}_{+}$  % and measurable function $L$ on $\mathbb{R}_{+}\times \mathbb{R}_{+}$, 
	the convolutions $K*L$ %and $L*K$ are 
	is defined by
	$$(K*L)(t,s)=\int_{s}^{t}K(t,u)L(u,s)\d u,$$
%	$$(L*K)(t,s)=\int_{s}^{t}L(t,u)K(u,s)\d u,$$
	for almost all $(t,s)\in \Delta$, where $\Delta:=\{t,s\in\mathbb{R}_{+}\times \mathbb{R}_{+}: s\leq t\}.$ 
	\subsection{$L$-derivatives for the functions on $\mathcal{P}_2(\mathbb{R}^d)$}For any $\mu\in\mathcal{P}_2(\mathbb{R}^{d})$, the tangent space at $\mu$ is given by
	$$T_{\mu,2}:=L^2(\mathbb{R}^d\mapsto\mathbb{R}^d;\mu):=\Big\{\phi:\mathbb{R}^d\mapsto\mathbb{R}^d; \phi \ \text{is measurable with}\ \mu(|\phi|^2):=\int_{\mathbb{R}^d}|\phi(x)|^2\mu(\d x)<\infty \Big\}.$$
	For $\phi\in T_{\mu,2}$, we set $\|\phi\|_{T_{\mu,2}}:=\int_{\mathbb{R}^d}|\phi(x)|^2\mu(\d x).$
	\begin{defn}
		Let $f:\mathcal{P}_2(\mathbb{R}^d)\mapsto\mathbb{R}$ be a continuous function, and let $I$ be the identity map on $\mathbb{R}^d$.
		
		$(i)$ If for any $\mu\in\mathcal{P}_2(\mathbb{R}^d)$
		$$T_{\mu,2}\ni\phi\mapsto D_{\phi}^Lf(\mu):=\lim_{\varepsilon\to 0}\frac{f(\mu\circ(I+\varepsilon\phi)^{-1})-f(\mu)}{\varepsilon}\in \mathbb{R}$$
		is a well-defined bounded linear functional, then we say that $f$ is intrinsically differentiable at $\mu$ and we denote the intrinsic derivative of $f$ at $\mu$ by $D^L_{\phi}f(\mu)$.
		
		$(ii)$ If for any $\mu\in \mathcal{P}_2(\mathbb{R}^d)$
		$$\lim _{\|\phi\|_{T_{\mu,2}} \rightarrow 0} \frac{|f\left(\mu \circ(\mathrm{Id}+\phi)^{-1}\right)-f(\mu)-D^L_{\phi}f(\mu)|}{\|\phi\|_{T_{\mu,2}}}=0,$$
		$f$ is called $L$-differentiable at $\mu$ with the $L$-derivative (i.e. Lions derivative) $D^L f(\mu)$. By the Riesz representation theorem, there exists a  unique element $D^Lf(\mu)\in T_{\mu,2}$ satisfying
		$$\big\langle D^Lf(\mu ), \phi\big\rangle_{T_{\mu,2}}:=\int_{\mathbb{R}^{d}}\big\langle D^Lf(\mu)(x),\phi(x)\big\rangle\mu(\d x)=D^L_{\phi}f(\mu), \quad \forall \ \phi\in T_{\mu,2}.$$
		
		$(iii)$ Denote by $C^1(\mathcal{P}_2(\mathbb{R}^d))$ the set of all continuous functionals  $f$ on $\mathcal{P}_2(\mathbb{R}^d)$ satisfying that $f$ is $L$-differentiable at any point $\mu\in\mathcal{P}_2(\mathbb{R}^d)$, and the $L$-derivative $D^Lf(\mu)(x)$ has a version jointly continuous in $(\mu,x)\in\mathcal{P}_2(\mathbb{R}^d)\times \mathbb{R}^d$. Moreover, let $C^1_b(\mathcal{P}_2(\mathbb{R}^d))$ be the subset of $C^1(\mathcal{P}_2(\mathbb{R}^d))$ such that  $D^L f(\mu)(x)$ is 
		bounded:  there is a $C_f$, depending only on $f$ so that  $|D^L f(\mu)(x)|\le C_f$ for all $\mu\in 
		\mathcal{P}_2$ and $x\in \RR^d$.

		For a vector-valued function $f=(f_i)$, or a matrix-valued function $f=(f_{i,j})$ with $L$-differentiable components, we write 
		$$D^L_{\phi}f(\mu)=(D^L_{\phi} f_{i}(\mu)), \ \text{or}\ D^L_{\phi}f(\mu)=(D^L_{\phi} f_{i,j}(\mu)), \ \mu\in\mathcal{P}_2(\mathbb{R}^d).$$
	\end{defn}
	\subsection{Volterra type Growall's inequality    }
	We first recall the following lemma due to Girpenberg \cite{GG} (Theorem $1$).
	\begin{lem}\label{lemma 2.1}
		Let $K:\Delta\to\mathbb{R}_{+}$ be a measurable function. Assume that for any $T>0$,
		$$t\mapsto \int_{0}^{t}K(t,s)\d s\in L^{\infty}(0,T)$$
		and
		$$\limsup_{\varepsilon\downarrow 0}\Big\|\int_{\cdot+\varepsilon}^{\cdot}K(\cdot+\varepsilon,s)\d s\Big\|_{L^{\infty}(0,T)}<1.$$
		Define
		\begin{align}\label{definition of R_n}
			R_1(t,s):=K(t,s),\quad R_{n+1}(t,s):=\int_{s}^{t}K(t,u)R_{n}(u,s)\d u,\ \  n\in\mathbb{N}.	
		\end{align}
		Then for any $T>0$, there exist constants $C_{T}>0$ and $\gamma\in (0,1)$ such that
		$$\Big\|\int_{0}^{\cdot}R_{n}(\cdot,s)\d s\Big\|_{L^{\infty}(0,T)}\leq C_{T}n\gamma^{n}, \quad \forall \ n\in \mathbb{N}.$$
		In particular, the series
		\begin{align}\label{Resolvent}
			R(t,s):=\sum_{n=1}^{\infty}R_{n}(t,s)
		\end{align}
		converges for almost all $(t,s)\in \Delta$.  Moreover, we have 
		\begin{align}\label{definition of resolvent}
			R(t,s)-K(t,s)=(K*R)(t,s)-(R*K)(t,s)\,, 
		\end{align}
		and for any $T>0$,
		\begin{align}\label{finite of resolvent}
			t\mapsto \int_{0}^{t}R(t,s)\d s\in L^{\infty}(0,T).
		\end{align}
	\end{lem}
	The function $R$ defined by \eqref{Resolvent} is called the resolvent of the kernel $K$. The set of all   functions $K$ in Lemma \ref{lemma 2.1} will be denoted by $\mathcal{K}$. 
	We now present Gronwall's lemma of Volterra type due to Zhang \cite{ZX2} (Lemma $2.2$).
	\begin{lem}\label{Gronwall}
		Let $K\in\mathcal{K}$ and $R_n$ and $R$ be defined by \eqref{definition of R_n} and \eqref{Resolvent}, respectively. Let $f,g:\mathbb{R}_{+}\to\mathbb{R}_{+}$ be two measurable functions satisfying that for any $T>0$ and some $n\in\mathbb{N}$,
		$$t\mapsto \int_{0}^{t}R_n(t,s)f(s)\d s\in L^{\infty}(0,T)$$
		and for almost all $t\in(0,\infty)$,
		$$\int_{0}^{t}R(t,s)g(s)\d s<\infty.$$
		If for almost all $t\in (0,\infty)$,
		$$f(t)\leq g(t)+\int_{0}^{t}K(t,s)f(s)\d s,$$
		then for almost all $t\in (0,\infty)$,
		$$f(t)\leq g(t)+\int_{0}^{t}R(t,s)g(s)\d s.$$
	\end{lem}
	\subsection{A criterion of large deviation principles}\label{A criterion of large deviation principles}
	Let $E$ be a Polish space. For each $\varepsilon>0$, let $X^{\varepsilon}$ be an $E$-valued random variable given on $(\Omega,\mathcal{F},(\mathcal{F}_t)_{t\ge0},\mathbb{P})$.
	\begin{defn}(rate function) A function $I:E\rightarrow [0,+\infty)$ is called a rate function  if $I$ is lower semicontinuous. Moreover, a rate function $I$ is called a good rate function if the level set $\{x\in E: I(x)\leq N\}$ is compact for each constant $N<\infty$.
	\end{defn}
	\begin{defn}(Large deviation principle)
	A sequence of  random variables  $\{ X^{\varepsilon}\}$ is said to satisfy the LDP on $E$ with rate function $I$ if the following lower and upper bound conditions hold:
		
		$(i)$ (Lower bound) For any open set $G\subset E$,
		$$\liminf_{\varepsilon\rightarrow 0}\varepsilon\log \mathbb{P}(X^{\varepsilon}\in G)\geq -\inf_{x\in G}I(x).$$
		
		$(ii)$ (Upper bound) For any closed set $F\subset E$,
		$$\limsup_{\varepsilon\rightarrow 0}\varepsilon\log \mathbb{P}(X^{\varepsilon}\in E)\leq -\sup_{x\in F}I(x).$$
	\end{defn}
	\begin{defn}(Laplace principle)
		We say that ${X^{\varepsilon}}$ satisfies the Laplace principle with the rate function $I$, if for any real bounded continuous function $g$ on $E$
		$$\lim_{\varepsilon\to 0}\varepsilon \log \EE \Big\{\exp\Big[-\frac{g(X^{\varepsilon})}{\varepsilon}\Big]\Big\}=-\inf_{x\in E}\big(g(x)+I(x)\big).$$
	\end{defn}
	\begin{lem}(Varadhan's lemma \cite{Var})\label{Varadhan's lemma}
		Let  the random sequence $\{ X^{\varepsilon} \}$  with values in   a Polish space $E$   fulfill  the LDP with a good rate function $I$.  Then $\{X^{\varepsilon}\}$ fulfills the Laplace principle on $E$ with the same rate function $I$.
	\end{lem}
	\begin{lem}(Bryc's converse \cite{DZ})\label{Bryc's converse}
		The Laplace principle implies the LDP with the same good rate function.
	\end{lem}
	Combination of  Lemmas \ref{Varadhan's lemma} and \ref{Bryc's converse} yields that if $E$ is a Polish space and $I$ is a good rate function, then the LDP and Laplace principle are equivalent.
	
	Denote $C^d_T:=C([0,T],\mathbb{R}^d)$ the space of continuous functions from $[0,T]$ to $\mathbb{R}^d$ equipped with the supremum norm $\|\phi\|_T:=\sup_{t\in [0,T]}|\phi_t|$ and let
	$$\mathcal{A}:=\Big\{v:\Omega\times [0,T]\to \mathbb{R}^{m}\ \text{progressively measurable}, E\Big[\int_{0}^{T}|v_t|^2\d t\Big]<+\infty\Big\},$$
	$$\mathcal{S}_N:=\Big\{v\in L^2: \int_{0}^{T}|v_s|^2\d s\leq N\Big\},$$
	and
	$$\mathcal{A}_N:=\Big\{v\in \mathcal{A}: v\in\mathcal{S}_N\Big\}.$$
	\begin{defn}\cite{JP}\label{v}
		For all $N\in\mathbb{N}$ and $v\in\mathcal{A}_N$ we define
		\begin{align}
			\mathcal{G}^0_{v,N}&:=\Big\{\phi:\Omega\to C^d_T\,; \  \text{so that there exist} \ \{\varepsilon_{n}\}_{n\in\mathbb{N}}\subset[0,\infty)\ \text{with}\ \lim_{n\to \infty}\varepsilon_{n}=0\nonumber\\&\text{and a sequence}\ \{v^{\varepsilon_{n}}\}_{n\in\mathbb{N}}\subset\mathcal{A}_N\text{with}\ \lim_{n\to \infty}v^{\varepsilon_{n}}=v\ \text{in distribution such that}\nonumber\\&\phi=\lim_{n\to \infty}\mathcal{G}^{\varepsilon_{n}}\Big(W+\frac{1}{\sqrt{\varepsilon_{n}}}\int_{0}^{\cdot}v_s^{\varepsilon_{n}}\d s\Big)\ \text{in distribution}\Big\},\nonumber
		\end{align}
		and $\mathcal{G}^0_{v,N}$ is empty if $v$ is not in $\mathcal{A}_N$. For all $v\in\mathcal{A}$, we also denote $\mathcal{G}_v^0:=\bigcup_{n\in\mathbb{N}}\mathcal{G}^0_{v,N}$.
	\end{defn}
	Then we define the functional $I: C^d_T\to [0,+\infty]$ given by
	\begin{align}\label{def of I}
		I(\phi):=\inf\Big\{\frac{1}{2}\int_{0}^{T}|v_s|^2\d s: v\in L^2 \ \text{such that}\ \phi\in\mathcal{G}^0_v\Big\}.
	\end{align}
	\begin{defn}\cite{JP}
		We say that $\phi\in C^d_T$ is uniquely characterized if there exists a sequence $\{v^n\}_{n\in\mathbb{N}}\subset L^2$ such that 
		$$\mathcal{G}^0_{v^n}=\{\phi\}\quad \text{and}\quad \frac{1}{2}\int_{0}^{T}|v^n_s|^2\d s\leq I(\phi)+\frac{1}{n},$$
		for all $n\in\mathbb{N}$. In particular, if there exists $\tilde{v}\in L^2$ which attains the infimum in \eqref{def of I} and $\mathcal{G}^0_{v^n}=\{\phi\}$ then $\phi$ is
		uniquely characterized, because one can choose $v^n=\tilde{v}$ for all $n\in\mathbb{N}$.
	\end{defn}
	The above functions $\phi,v,v^n,\tilde{v}$ are deterministic.
	
	In order to formulate the sufficient condition for Laplace  principle (equivalently, Large
	Deviations Principle). We also need the following basic assumption from \cite{JP}.
	
	\textbf{(U)} For any $\delta>0$ and any $\phi\in C^d_T$ such that $I(\phi)<\infty$, there exists $\phi^{\delta}$ uniquely characterized such that $\|\phi-\phi^{\delta}\|_T\leq \delta$ and $\|I(\phi)-I(\phi^{\delta})\|_T\leq\delta$.
	\begin{lem}\cite{JP}\label{general LDP}
		Assume that
		
		Condition (i):  For all $N > 0$, all $v\in\mathcal{A}_N$ and all families $\{v^{\varepsilon}\}_{\varepsilon>0}$ in $\mathcal{A}_N$ converging in distribution
		to $v$ as $\varepsilon$ tends to zero, $\{\mathcal{G}^{\varepsilon}\big(W+\frac{1}{\sqrt{\varepsilon}}\int_{0}^{\cdot}v_s^{\varepsilon}\d s\big)\}$
		is tight.
		
		Condition (ii):  The functional $I$ defined by \eqref{def of I} has compact level sets, i.e. $I$ is a good rate function.
		
		Condition (iii):  Assumption \textbf{(U)} holds.
		
		Then the family $\{\mathcal{G}^{\varepsilon}(W)\}_{\varepsilon>0}$ satisfies the Laplace principle and, by equivalence,  satisfies  the Large
		Deviations Principle with rate function $I$ and speed $\varepsilon^{-1}$.
	\end{lem}
\subsection{Technical lemmas}
In order to prove our  main results, we first state some assumptions used later.
	
	\textbf{(H1)} There exist  positive constant $\gamma>0$ and $C_T>0$ such that for any $0\le t<t'\le T$
$$\int_{0}^{t}|K_i(t',s)-K_i(t,s)|^2\d s\leq C_T|t'-t |^{2\gamma}, \quad t\in [0,T],\quad i=1,2,$$
%	$$\int_{0}^{t}|K_2(t',s)-K_2(t,s)|^2\d s\leq C_T|t-t'|^{2\gamma}, \quad t\in [0,T],$$
%	$$\int_{t}^{t'}K^2_1(t,s)\d s\leq C_T|t-t'|^{2\gamma}, \quad t\in [0,T],$$
$$\int_{t}^{t'}K^2_i(t',s)\d s\leq C_T|t-t'|^{2\gamma}, \quad t\in [0,T]\,,\quad i=1,2\,.$$

\textbf{(H2)} For any $x,y\in\mathbb{R}^d$, $\mu,\nu\in \mathcal{P}_{2}(\mathbb{R}^d)$ and $t\in [0,T]$
\begin{align}
	|b(t,x,\mu)-b(t,y,\nu)\|+\|\sigma(t,x,\mu)-\sigma(t,y,\nu)\|&\leq L_1\Big(|x-y|+\mathbb{W}_{2}(\mu,\nu)\Big),
	\nonumber			
	% 		\\\|\sigma(t,x,\mu)-\sigma(t,y,\nu)\|^2&\leq L_1\Big(|x-y|^2+\mathbb{W}_{2}^2(\mu,\nu)\Big),\nonumber
\end{align}
where $L_1>0$ is a positive constant.

\textbf{(H3)} For any $x\in\mathbb{R}^d$, $\mu\in \mathcal{P}_{2}(\mathbb{R}^d)$ and $t\in [0,T]$
\begin{align}
	|b(t,x,\mu)|+\|\sigma(t,x,\mu)\| &\leq L_2\Big(1+|x|+\mathbb{W}_{2}(\mu,\delta_0)\Big),\nonumber
	%		\\\|\sigma(t,x,\mu)\|^2&\leq L_2\Big(1+|x|^2+\mathbb{W}_{2}^2(\mu,\delta_0)\Big),\nonumber
\end{align}
where $L_2>0$ is a positive constant.

\begin{rem}\label{condition of Volterra}
	(i) Note that by \textbf{(H1)} and H$\mathrm{\ddot{o}}$lder's inequality, the kernel $K(t,s):=C_1K_1(t,s)+C_2K_2(t,s)+C_3K_1^2(t,s)+C_4K_2^2(t,s)$ with positive constants $C_i\geq 0, i=1,2,3,4$ satisfies the following properties:
	$$t\mapsto \int_{0}^{t}K(t,s)\d s\in L^{\infty}(0,T)$$
	and
	$$\limsup_{\varepsilon\downarrow 0}\Big\|\int_{\cdot}^{\cdot+\varepsilon}K(\cdot+\varepsilon,s)\d s\Big\|_{L^{\infty}(0,T)}<1.$$
	
	 	Let $R_n^{K}$ and $R^{K}$ be defined by \eqref{definition of R_n}  and \eqref{Resolvent} in terms of K respectively. By \eqref{definition of R_n}, \eqref{definition of resolvent}, \eqref{finite of resolvent} and  \textbf{(H1)}, when $n=1$ we have %for positive constants $C_5,C_6$
		\begin{align}
			\int_{0}^{T}R_1^{K}(t,s)  \d s=\int_{0}^{T}K(t,s) \d s<\infty,\nonumber
		\end{align}
		and
		\begin{align}
			&\int_{0}^{t}R^{K}(t,s) \d s<\infty\nonumber.
		\end{align}
		This enables  us to use Lemma \ref{Gronwall}.

(ii) It is clear that under the assumptions \textbf{(H1)}-\textbf{(H3)}, there exists a unique solution $X^{\varepsilon}_{t,\xi}$ of \eqref{Volterra MV eq} ($X^{0}_{t,\xi}$ of \eqref{Volterra eq-limit eq}) (Theorem $3.1$ of \cite{LG}).
\end{rem}
\begin{lem}\label{basic lemma}
	Under the assumption \textbf{(H1)}, the following inequalities hold for any adapted $\mathbb{R}^{d}$-valued process $Y$ and $\mathbb{R}^{d\times m}$-valued process $Z$:
	
	(i) for $p\geq 2$ and $t\in [0,T]$,
	$$\E\Big[\Big|\int_{0}^{t}K_1(t,s)Y_s\d s\Big|^p\Big]\leq C_{p,T}\int_{0}^{t}K_1(t,s)\cdot \E|Y_s|^p\d s,$$
	
	(ii) for $p\geq 2$ and $t\in [0,T]$,
	$$\E\Big[\Big|\int_{0}^{t}K_2(t,s)Z_s\d W_s\Big|^p\Big]\leq C_{p,T}\int_{0}^{t}K_2^2(t,s)\cdot \E\|Z_s\|^p\d s,$$
	
	(iii) for $p\geq 2,v\in\mathcal{A}_{N}$ and $t\in [0,T]$,
	$$\E\Big[\Big|\int_{0}^{t}K_1(t,s)Z_s v_s\d s\Big|^p\Big]\leq C_{p,T,N}\int_{0}^{t}K^2_1(t,s)\cdot \E\|Z_s\|^p\d s,$$
	
	(iv) for $p\geq 1, t\in [0,T]$ and $h\geq 0$ with $t+h\leq T$,
	$$\E\Big[\Big|\int_{0}^{t}(K_1(t+h,s)-K_1(t,s))Y_s\d s\Big|^p\Big]+\E\Big[\Big|\int_{t}^{t+h}K_1(t+h,s)Y_s\d s\Big|^p\Big]\leq C_{p,T} h^{\gamma p}\sup_{t\in[0,T]}\E[|Y_t|^p],$$
	
	(v) for $p\geq 2, t\in [0,T]$ and $h\geq 0$ with $t+h\leq T$,
	$$\E\Big[\Big|\int_{0}^{t}(K_2(t+h,s)-K_2(t,s))Z_s\d W_s\Big|^p\Big]+\E\Big[\Big|\int_{t}^{t+h}K_2(t+h,s)Z_s\d W_s\Big|^p\Big]\leq C_{p,T} h^{\gamma p}\sup_{t\in[0,T]}\E[\|Z_t\|^p],$$
	
	(vi) for $p\geq 2, v\in\mathcal{A}_{N}, t\in [0,T]$ and $h\geq 0$ with $t+h\leq T$,
	$$\E\Big[\Big|\int_{0}^{t}(K_1(t+h,s)-K_1(t,s))Z_s v_s\d s\Big|^p\Big]+\E\Big[\Big|\int_{t}^{t+h}K_1(t+h,s)Z_s v_s\d s\Big|^p\Big]\leq C_{T,N} h^{\gamma p}\sup_{t\in[0,T]}\E[\|Z_t\|^p].$$
\end{lem}
\begin{proof}
	Let us show $(i)$, $(iii)$ and $(iv)$ first. Take $p\geq 2$. H$\mathrm{\ddot{o}}$lder's inequality show that
	\begin{align}
		\E\Big[\Big|\int_{0}^{t}K_1(t,s)Y_s\d s\Big|^p\Big]&\leq \int_{0}^{t}K_1(t,s)\cdot \E[|Y_s|^p]\d s\cdot \Big|\int_{0}^{t}K_1(t,s)\d s\Big|^{p-1}\nonumber\\&\leq C_{p,T}\int_{0}^{t}K_1(t,s)\cdot \E[|Y_s|^p]\d s.\nonumber
	\end{align}
For $(iii)$, by H$\mathrm{\ddot{o}}$lder's inequality and recall $v\in \mathcal{A}_{N}$ we have
\begin{align}
	\E\Big[\Big|\int_{0}^{t}K_1(t,s)Z_s v_s\d s\Big|^p\Big]&\leq \big(\int_{0}^{t}v_s^2\d s\big)
^{\frac{p}{2}}\cdot \E\Big|\int_{0}^{t}K^2_1(t,s)\cdot\|Z_s\|^2\d s\Big|^{\frac{p}{2}}\nonumber\\&\leq N^{\frac{p}{2}}\int_{0}^{t}K_1^2(t,s)\cdot \E[\|Z_s\|^p]\d s
\cdot \Big|\int_{0}^{t}K_1^2(t,s)\d s\Big|^{\frac{p}{2}-1}\nonumber\\&\leq C_{p,T,N}\int_{0}^{t}K_1^2(t,s)\cdot \E[\|Z_s\|^p]\d s.\nonumber
\end{align}
	For $(iv)$, by Minkowski's integral inequality we observe 
	\begin{align}
		\E\Big[&\Big|\int_{0}^{t}(K_1(t+h,s)-K_1(t,s))Y_s\d s\Big|^p\Big]+\E\Big[\Big|\int_{t}^{t+h}K_1(t+h,s)Y_s\d s\Big|^p\Big]\nonumber\\&\leq \Big(\int_{0}^{t}\Big|K_1(t+h,s)-K_1(t,s)\Big|\cdot[\E|Y_s|^p]^{\frac{1}{p}}\d s\Big)^{p}+\Big(\int_{t}^{t+h}K_1(t+h,s)\cdot[\E|Y_s|^p]^{\frac{1}{p}}\d s\Big)^{p}\nonumber\\&\leq \sup_{t\in[0,T]}E[|Y_s|^p]\Big(\Big(\int_{0}^{t}\Big|K_1(t+h,s)-K_1(t,s)\Big|\d s\Big)^{p}+\Big(\int_{t}^{t+h}K_1(t+h,s)\d s\Big)^{p}\Big)\nonumber\\&\leq C_{p,T} h^{\gamma p}\sup_{t\in[0,T]}\E[|Y_s|^p].\nonumber
	\end{align}
	We next show $(ii)$ and $(v)$. For $(ii)$, let $p\geq 2$,  using BDG's inequality, Minkowski's integral inequality and
	H$\mathrm{\ddot{o}}$lder's inequality   we   have 
	\begin{align}
		\E \Big[\Big|\int_{0}^{t}K_2(t,s)Z_s\d W_s\Big|^p\Big]&\leq C_{p}\E\Big[ \Big(\int_{0}^{t}K_2^2(t,s)\|Z_s\|^2\d s\Big)^{\frac{p}{2}}\Big] \nonumber\\&\leq C_{p} \int_{0}^{t}K_2^2(t,s)\cdot \E[\|Z_s\|^p]\d s\cdot \Big|\int_{0}^{t}K_2^2(t,s)\d s\Big|^{\frac{p}{2}-1}\nonumber \\&\leq C_{p,T}\int_{0}^{t}K_2^2(t,s)\cdot \E[\|Z_s\|^p]\d s.\nonumber                                                                   \end{align}
	For $(v)$, take $p\geq 2$. By the BDG's inequality and Minkowski's inequality, we obtain
	\begin{align}
		\E\Big[&\Big|\int_{0}^{t}(K_2(t+h,s)-K_2(t,s))Z_s\d W_s\Big|^p\Big]+\E\Big[\Big|\int_{t}^{t+h}K_2(t+h,s)Z_s\d W_s\Big|^p\Big]\nonumber\\&\leq C_p\E\Big(\int_{0}^{t}\big|K_2(t+h,s)-K_2(t,s)\big|^2\cdot\|Z_s\|^2\d s\Big)^{\frac{p}{2}}+C_p\E\Big(\int_{t}^{t+h}K^2_2(t+h,s)\cdot\|Z_s\|^2\d s\Big)^{\frac{p}{2}}\nonumber\\&\leq C_p\Big(\int_{0}^{t}\big|K_2(t+h,s)-K_2(t,s)\big|^2\cdot \E[\|Z_s\|^p]^{\frac{2}{p}}\d s\Big)^{\frac{p}{2}}+C_p\Big(\int_{t}^{t+h}K_2^2(t+h,s)\E[\|Z_s\|^p]^{\frac{2}{p}}\d s\Big)^{\frac{p}{2}}\nonumber\\&\leq C_{p} \sup_{t\in[0,T]}\E[\|Z_t\|^p]\Big[\Big(\int_{0}^{t}|K_2(t+h,s)-K_2(t,s)|^2\d s\Big)^{\frac{p}{2}}+\Big(\int_{t}^{t+h}K^2_2(t+h,s)\d s\Big)^{\frac{p}{2}}\Big]\nonumber\\&\leq C_{p,T} h^{\gamma p}\sup_{t\in[0,T]}\E[\|Z_t\|^p].\nonumber
	\end{align}
Finally, we show $(vi)$. By H$\mathrm{\ddot{o}}$lder's inequality and recall $v\in \mathcal{A}_{N}$ we have
\begin{align}
	&\E\Big[\Big|\int_{0}^{t}\big(K_2(t+h,s)-K_2(t,s)\big)Z_s v_s\d s\Big|^p\Big]+\E\Big[\Big|\int_{t}^{t+h}K_2(t+h,s)Z_s v_s\d s\Big|^p\Big]\nonumber\\&\leq \big(\int_{0}^{t}v_s^2\d s\big)
	^{\frac{p}{2}}\cdot \E\Big(\int_{0}^{t}\big|K_2(t+h,s)-K_2(t,s)\big|^2\cdot\|Z_s\|^2\d s\Big)^{\frac{p}{2}}\nonumber\\&\quad+\big(\int_{0}^{t}v_s^2\d s\big)
	^{\frac{p}{2}}\cdot \E\Big(\int_{t}^{t+h}K^2_2(t+h,s)\cdot\|Z_s\|^2\d s\Big)^{\frac{p}{2}}\nonumber\\&\leq C_N\Big(\int_{0}^{t}\big|K_2(t+h,s)-K_2(t,s)\big|^2\cdot \E[\|Z_s\|^p]^{\frac{2}{p}}\d s\Big)^{\frac{p}{2}}+\Big(\int_{t}^{t+h}K_2^2(t+h,s)\E[\|Z_s\|^p]^{\frac{2}{p}}\d s\Big)^{\frac{p}{2}}\nonumber\\&\leq C_{N} \sup_{t\in[0,T]}\E[\|Z_t\|^p]\Big(\Big(\int_{0}^{t}|K_2(t+h,s)-K_2(t,s)|^2\d s\Big)^{\frac{p}{2}}+\Big(\int_{t}^{t+h}K^2_2(t+h,s)\d s\Big)^{\frac{p}{2}}\Big)\nonumber\\&\leq C_{T,N} h^{\gamma p}\sup_{t\in[0,T]}\E[\|Z_t\|^p],\nonumber
\end{align}

	 which completes the proof.
\end{proof}
	\begin{lem}\label{xxi}
	Under the assumptions \textbf{(H1)}-\textbf{(H3)}, it follows that for $p\geq 1,\xi,\eta\in \mathbb{R}^d$
	$$\sup_{t\in [0,T]}\E|X^{\varepsilon}_{t,\xi}|^{p}\leq C_{p,T,L_2}(1+|\xi|^{p}),$$
	and 
	$$\sup_{t\in [0,T]}\E|X^{\varepsilon}_{t,\xi}-X^{\varepsilon}_{t,\eta}|^{p}\leq C_{p,T,L_1}|\xi-\eta|^{p},$$
	where the constants $C_{p,T,L_1}$ and $C_{p,T,L_2}$ are independent of $\varepsilon$.
\end{lem}
\begin{proof}
		We  can assume  $p\geq 2$,  which is sufficient for $p\geq1$ by Lyapunov inequality. By Lemma \ref{basic lemma} (i)-(ii), \textbf{(H1)}, \textbf{(H3)} and Lemma $2.6$ in \cite{dST}, for $0\leq \varepsilon<1$ we have
	\begin{align}
		&\E|X^{\varepsilon}_{t,\xi}|^{p}\nonumber\\&=\E\Big|\xi+\int_{0}^{t}K_1(t,s)b(s,X^\varepsilon_{s,\xi},\mathcal{L}_{X^\varepsilon_{s,\xi}})\d s+\sqrt{\varepsilon}\int_{0}^{t}K_2(t,s)\sigma(s,X^\varepsilon_{s,\xi},\mathcal{L}_{X^\varepsilon_{s,\xi}})\d W_s\Big|^{p}\nonumber\\&\leq 3^{p-1}|\xi|^{p}+3^{p-1}\E\Big|\int_{0}^{t}K_1(t,s)b(s,X^\varepsilon_{s,\xi},\mathcal{L}_{X^\varepsilon_{s,\xi}})\d s\Big|^{p}+3^{p-1}\E\Big|\sqrt{\varepsilon}\int_{0}^{t}K_2(t,s)\sigma(s,X^\varepsilon_{s,\xi},\mathcal{L}_{X^\varepsilon_{s,\xi}})\d W_s\Big|^{p}\nonumber\\&\leq 3^{p-1}|\xi|^{p}+C_{p,T}\int_{0}^{t}K_1(t,s)\cdot \E|b(s,X^\varepsilon_{s,\xi},\mathcal{L}_{X^\varepsilon_{s,\xi}})|^{p}\d s+C_{p,T}\int_{0}^{t}K^2_2(t,s)\cdot \E\big\|\sigma(s,X^\varepsilon_{s,\xi},\mathcal{L}_{X^\varepsilon_{s,\xi}})\big\|^{p}\d s\nonumber\\&\leq 3^{p-1}|\xi|^{p}+C_{p,T,L_2}\int_{0}^{t}K_1(t,s)\cdot \E\Big|1+|X^\varepsilon_{s,\xi}|+\mathbb{W}_{2}(\mathcal{L}_{X^\varepsilon_{s,\xi}},\delta_0)\Big|^{p}\d s\nonumber\\&\quad+C_{p,T,L_2}\int_{0}^{t}K_2^2(t,s)\cdot \E\Big|1+|X^\varepsilon_{s,\xi}|+\mathbb{W}_{2}(\mathcal{L}_{X^\varepsilon_{s,\xi}},\delta_0)\Big|^{p}\d s\nonumber\\&\leq 3^{p-1}|\xi|^{p}+C_{p,T,L_2}\int_{0}^{t}K_1(t,s)\big(1+\E|X^{\varepsilon}_{s,\xi}|^{p}\big)\d s+C_{p,T,L_2}\int_{0}^{t}K^2_2(t,s)\big(1+\E|X^{\varepsilon}_{s,\xi}|^{p}\big)\d s\nonumber\\&\leq 3^{p-1}|\xi|^{p}+C_{p,T,L_2}+C_{p,T,L_2}\int_{0}^{t}K_{1,2}(t,s)\cdot \E|X^{\varepsilon}_{s,\xi}|^{p}\d s,\nonumber
	\end{align}
	where $K^{*}_{1,2}(t,s):=C_{p,T,L_2}\cdot (K_1(t,s)+K_2^2(t,s))$ and the constant $C_{p,T,L_2}$ is independent of $\varepsilon$.
	
	Let $R^{K^{*}_{1,2}}$ be defined by \eqref{Resolvent} in terms of $K_{1,2}$. By Remark \ref{condition of Volterra}, Lemma \ref{Gronwall} and \eqref{finite of resolvent}, we obtain that for almost all $t\in [0,T]$,
	\begin{align}
		\E|X^{\varepsilon}_{t,\xi}|^{p}&\leq 3^{p-1}|\xi|^{p}+C_{p,T,L_2}+\int_{0}^{t}R^{K^{*}_{1,2}}(t,s)\cdot\big(3^{p-1}|\xi|^{p}+C_{p,T,L_2}\big)\d s\nonumber\\&\leq C_{p,T,L_2}(1+|\xi|^{p}).\nonumber
	\end{align}
	Similarly, we can prove 
	$$\E|X^{\varepsilon}_{t,\xi}-X^{\varepsilon}_{t,\eta}|^{p}\leq C_{p,T,L_1}|\xi-\eta|^{p}.$$
\end{proof}
Proceeding as in the proof of Lemma \ref{xxi}, we have the following Lemma associated with $X^{0}_{t,\xi}$.
\begin{lem}\label{x0xi}
	Under the assumptions \textbf{(H1)}-\textbf{(H3)}, it follows that for $p\geq 1,\xi,\eta\in \mathbb{R}^d$
	$$\sup_{t\in [0,T]}|X^{0}_{t,\xi}|^{p}\leq C_{p,T,L_2}(1+|\xi|^{p}),$$
	and 
	$$\sup_{t\in [0,T]}|X^{0}_{t,\xi}-X^{0}_{t,\eta}|^{p}\leq C_{p,T,L_1}|\xi-\eta|^{p},$$
	where the constants $C_{p,T,L_1}$ and $C_{p,T,L_2}$ are independent of $\varepsilon$.
\end{lem}
 
\begin{lem}\label{esi-Q}  Denote  $Q^{\varepsilon}_{t}:=\frac{X^{\varepsilon}_{t,\xi}-X^0_{t,\xi}}{\sqrt{\varepsilon}h(\varepsilon)}$. 
		Assume %that the solution $Q^\varepsilon_{t}$ of \eqref{Volterra eq-differ} is well-posedness and 
		that \textbf{(H1)}-\textbf{(H3)} hold.   Then  for $p>2,\xi\in \mathbb{R}^{d}$
		$$\sup_{t\in[0,T]}\E|Q^{\varepsilon}_{t}|^p\leq \frac{C_{p,T,L_1,L_2}}{|h(\varepsilon)|^p}(1+|\xi|^{p}),$$
		where the constant $C_{p,T,L_1,L_2}$ is independent of $\varepsilon$. 
\end{lem}
\begin{proof}
	Write 
	\begin{align}\label{Volterra eq-differ}
		Q^\varepsilon_{t}=\int_{0}^{t}K_1(t,s)\cdot\frac{b(s,X^\varepsilon_{s,\xi},\mathcal{L}_{X^\varepsilon_{s,\xi}})-b(s,X^0_{s,\xi},\delta_{X^0_{s,\xi}})}{\sqrt{\varepsilon}h(\varepsilon)}\d s+\frac{1}{h(\varepsilon)}\int_{0}^{t}K_2(t,s)\sigma(s,X^\varepsilon_{s,\xi},\mathcal{L}_{X^\varepsilon_{s,\xi}})\d W_s.
	\end{align}
By Lemma \ref{basic lemma} (i)-(ii) and                                                                                                                                                                                                                                                                                                                                                                                                                                                                                                                                                                                                                                                                                                                                                                                                                                                                                                                                                                                                                                               \textbf{(H1)}-\textbf{(H3)}, we have
\begin{align}
	&\E|Q^{\varepsilon}_{t}|^{p}\nonumber\\&=\E\Big|\int_{0}^{t}K_1(t,s)\cdot\frac{b(s,X^\varepsilon_{s,\xi},\mathcal{L}_{X^\varepsilon_{s,\xi}})-b(s,X^0_{s,\xi},\delta_{X^0_{s,\xi}})}{\sqrt{\varepsilon}h(\varepsilon)}\d s+\frac{1}{h(\varepsilon)}\int_{0}^{t}K_2(t,s)\sigma(s,X^\varepsilon_{s,\xi},\mathcal{L}_{X^\varepsilon_{s,\xi}})\d W_s\Big|^{p}\nonumber\\
	&\leq 2^{p-1}\E\Big|\int_{0}^{t}K_1(t,s)\cdot\frac{b(s,X^\varepsilon_{s,\xi},\mathcal{L}_{X^\varepsilon_{s,\xi}})-b(s,X^0_{s,\xi},\delta_{X^0_{s,\xi}})}{\sqrt{\varepsilon}h(\varepsilon)}\d s\Big|^{p}+2^{p-1}\E\Big|\frac{1}{h(\varepsilon)}\int_{0}^{t}K_2(t,s)\sigma(s,X^\varepsilon_{s,\xi},\mathcal{L}_{X^\varepsilon_{s,\xi}})\d W_s\Big|^{p}\nonumber\\
	&\leq C_{p,T}\int_{0}^{t}K_1(t,s)\cdot \E	\Big|\frac{b(s,X^\varepsilon_{s,\xi},\mathcal{L}_{X^\varepsilon_{s,\xi}})-b(s,X^0_{s,\xi},\delta_{X^0_{s,\xi}})}{\sqrt{\varepsilon}h(\varepsilon)}\Big|^{p}\d s\nonumber\\
	&\qquad+\frac{C_{p,T}}{|h(\varepsilon)|^p}\int_{0}^{t}K^2_2(t,s)\cdot \E\big\|\sigma(s,X^\varepsilon_{s,\xi},\mathcal{L}_{X^\varepsilon_{s,\xi}})\big\|^{p}\d s\nonumber\\&\leq C_{p,T,L_1}\int_{0}^{t}K_1(t,s)\cdot \E\Big|\frac{|X^\varepsilon_{s,\xi}-X^{0}_{s,\xi}|+\mathbb{W}_{2}(\mathcal{L}_{X^\varepsilon_{s,\xi}},\delta_{X^0_{s,\xi}})}{\sqrt{\varepsilon}h(\varepsilon)}\Big|^{p}\d s\nonumber\\
	&\qquad+\frac{C_{p,T,L_2}}{|h(\varepsilon)|^p}\int_{0}^{t}K_2^2(t,s)\cdot \E\Big| \big(1+|X^{\varepsilon}_{s,\xi}|+\mathbb{W}_2(\mathcal{L}_{X^{\varepsilon}_{s,\xi}},\delta_{0})\big)\Big|^{p}\d s\nonumber\\
	&\leq C_{p,T,L_1}\int_{0}^{t}K_1(t,s)\cdot \E|Q^{\varepsilon}_{s,\xi}|^{p}\d s+\frac{C_{p,T,L_2}}{|h(\varepsilon)|^p}\int_{0}^{t}K^2_2(t,s)\cdot\big(1+\E|X^{\varepsilon}_{s,\xi}|^{p}\big)\d s\nonumber\\
	&\leq \frac{C_{p,T,L_2}}{|h(\varepsilon)|^p}\big(1+|\xi|^p\big)+C_{p,T,L_1}\int_{0}^{t}K_{1}(t,s)\cdot \E|Q^{\varepsilon}_{s,\xi}|^{p}\d s\nonumber,
\end{align}
where in the last step Lemma \ref{xxi} is used.

Let $R^{\hat{K}_1}$ be defined by \eqref{Resolvent} in terms of $\hat{K}_1:=C_{p,T,L_1}K_1(t,s)$. By Remark \ref{condition of Volterra}, Lemma \ref{Gronwall} and \eqref{finite of resolvent}, we obtain that for almost all $t\in [0,T]$,
\begin{align}
	\E|Q^{\varepsilon}_{t}|^{p}&\leq \frac{C_{p,T,L_2}}{|h(\varepsilon)|^p}\big(1+|\xi|^{p}\big)+\int_{0}^{t}R^{\hat{K}_{1}}(t,s)\cdot  \frac{C_{p,T,L_2}}{|h(\varepsilon)|^p}\big(1+|\xi|^{p}\big)\d s\nonumber\\&\leq \frac{C_{p,T,L_1,L_2}}{|h(\varepsilon)|^p}(1+|\xi|^{p}).\nonumber
\end{align}  
This proves the lemma.                                                              
\end{proof}

 	\begin{lem}\label{estimate of Lambda}
%	Assume  \textbf{(H1)}-\textbf{(H3)} hold and   that  $\Lambda^{\varepsilon,v}_{t,x}$ satisfies  
%	\eqref{eq of Lambda}. 
	Let $M: [0, T]\times \RR^d\times  \mathcal{P}_2(\mathbb{R}^d)\times \mathcal{P}_2(\mathbb{R}^d)\times 
	\mathbb{R}\to \mathbb{R}^d$, $N, G:  [0, T]\times \RR^d\times    \mathcal{P}_2(\mathbb{R}^d)\times 
	\mathbb{R}\to \mathbb{R}^d$  satisfy  
	$$M(s,X^{0}_{s,\xi},\mathcal{L}_{X_{s,\xi}^{\varepsilon}},\delta_{X_{s,\xi}^{0}},\Lambda^{\varepsilon,v}_{s,x},h(\varepsilon))\leq C\big(1+|\Lambda_{s,x}^{\varepsilon,v}|+\mathbb{W}_2(\mathcal{L}_{X_{s,\xi}^{\varepsilon}},\delta_{0})+\frac{1}{\sqrt{\varepsilon}h(\varepsilon)}\mathbb{W}_2(\mathcal{L}_{X_{s,\xi}^{\varepsilon}},\delta_{X^0_{s,\xi}})\big),$$
	$$N(s,X^{0}_{s,\xi},\mathcal{L}_{X_{s,\xi}^{\varepsilon}},\Lambda^{\varepsilon,v}_{s,x},h(\varepsilon))\vee G(s,X^{0}_{s,\xi},\mathcal{L}_{X_{s,\xi}^{\varepsilon}},\Lambda^{\varepsilon,v}_{s,x},h(\varepsilon)) \leq C\big(1+|\Lambda_{s,x}^{\varepsilon,v}|+|X^{0}_{s,\xi}|+\mathbb{W}_2(\mathcal{L}_{X_{s,\xi}^{\varepsilon}},\delta_{0})\big) $$
	for some finite  constant $C$.  Let $\Lambda^{\varepsilon,v}_{t,x}$ satisfy 
	\begin{align}\label{eq of Lambda}
		\Lambda^{\varepsilon,v}_{t,x}&=x+\int_{0}^{t}K_1(t,s)M(s,X^{0}_{s,\xi},\mathcal{L}_{X_{s,\xi}^{\varepsilon}},\delta_{X_{s,\xi}^{0}},\Lambda^{\varepsilon,v}_{s,x},h(\varepsilon))\d s\nonumber\\&+\int_{0}^{t}K_1(t,s)N(s,X^{0}_{s,\xi},\mathcal{L}_{X_{s,\xi}^{\varepsilon}},\Lambda^{\varepsilon,v}_{s,x},h(\varepsilon))v_s\d s+\int_{0}^{t}K_2(t,s)G(s,X^{0}_{s,\xi},\mathcal{L}_{X_{s,\xi}^{\varepsilon}},\Lambda^{\varepsilon,v}_{s,x},h(\varepsilon))\d W_s\,,
	\end{align} 
	where $X_{s, \xi}^0$ and $X_{s, \xi}^\varepsilon$  are defined as before.  
	Then,   for $p> 2, N>0,v\in\mathcal{A}_N,  \xi\in \mathbb{R}^d$
	$$\sup_{t\in  [0,T]}\E|\Lambda^{\varepsilon,v}_{t,x}|^{p}\leq C_{p,T,N,L_2}(1+|\xi|^{p}),$$
	where the constant $C_{p,T,N,L_2}$ is independent of $\varepsilon\in (0, 1],$ and $v\in\mathcal{A}_N$ (but it may depend  on $\mathcal{A}_N$ or $N$).
\end{lem} 
\begin{proof}
Since $\Lambda_{t,\xi}^{\varepsilon,v}$ satisfies   \eqref{eq of Lambda}, we have 
	\begin{align}
		&\E[|\Lambda_{t,x}^{\varepsilon,v}|^p]\nonumber\\&=\E\Big[\Big|x+\int_{0}^{t}K_1(t,s)M(s,X^{0}_{s,\xi},\mathcal{L}_{X_{s,\xi}^{\varepsilon}},\delta_{X_{s,\xi}^{0}},\Lambda^{\varepsilon,v}_{s,x},h(\varepsilon))\d s\nonumber\\
		&\qquad+\int_{0}^{t}K_1(t,s)N(s,X^{0}_{s,\xi},\mathcal{L}_{X_{s,\xi}^{\varepsilon}},\Lambda^{\varepsilon,v}_{s,x},h(\varepsilon))v_s\d s+\int_{0}^{t}K_2(t,s)G(s,X^{0}_{s,\xi},\mathcal{L}_{X_{s,\xi}^{\varepsilon}},\Lambda^{\varepsilon,v}_{s,x},h(\varepsilon))\d W_s\Big|^p\Big]\nonumber\\
		&\leq 4^{p-1}|x|^p+4^{p-1}\E\Big|\int_{0}^{t}K_1(t,s)M(s,X^{0}_{s,\xi},\mathcal{L}_{X_{s,\xi}^{\varepsilon}},\delta_{X_{s,\xi}^{0}},\Lambda^{\varepsilon,v}_{s,x},h(\varepsilon))\d s\Big|^p\nonumber\\
		&\qquad+ 4^{p-1}\E\Big|\int_{0}^{t}K_1(t,s)N(s,X^{0}_{s,\xi},\mathcal{L}_{X_{s,\xi}^{\varepsilon}},\Lambda^{\varepsilon,v}_{s,x},h(\varepsilon))v_s\d s\Big|^p\nonumber\\
		&\qquad +4^{p-1}\E\Big|\int_{0}^{t}K_2(t,s)G(s,X^{0}_{s,\xi},\mathcal{L}_{X_{s,\xi}^{\varepsilon}},\Lambda^{\varepsilon,v}_{s,x},h(\varepsilon))\d W_s\Big|^p\nonumber\\
		&:=4^{p-1}|x|^p+I_1+I_2+I_3.\nonumber
	\end{align}
	For the term $I_1$, by Lemma \ref{basic lemma}-(i), \textbf{(H1)}-\textbf{(H3)}, Lemmas \ref{xxi} and   \ref{esi-Q} we have
	\begin{align}
		I_1&\leq C_{p,T} \int_{0}^{t}K_1(t,s)\cdot \E\Big| M(s,X^{0}_{s,\xi},\mathcal{L}_{X_{s,\xi}^{\varepsilon}},\delta_{X_{s,\xi}^{0}},\Lambda^{\varepsilon,v}_{s,x},h(\varepsilon))\Big|^p\d s\nonumber\\&\leq C_{p,T} \int_{0}^{t}K_1(t,s)\cdot \E\big|1+|\Lambda_{s,x}^{\varepsilon,v}|+\mathbb{W}_2(\mathcal{L}_{X_{s,\xi}^{\varepsilon}},\delta_{0})+\frac{1}{\sqrt{\varepsilon}h(\varepsilon)}\mathbb{W}_2(\mathcal{L}_{X_{s,\xi}^{\varepsilon}},\delta_{X^0_{s,\xi}})\big|^p\d s\nonumber\\&\leq C_{p,T}\int_{0}^{t}K_1(t,s)\cdot\big(1+\E|\Lambda_{s,x}^{\varepsilon,v}|^p+\E|X^{\varepsilon}_{s,\xi}|^p+\E|Q^{\varepsilon}_{t,\xi}|^p\big)\d s\nonumber\\&\leq C_{p,T,L_1,L_2}\big(1+|\xi|^p\big)+C_{p,T}\int_{0}^{t}K_1(t,s)\cdot \E|\Lambda_{s,x}^{\varepsilon,v}|^p\d s\nonumber.
	\end{align}

	For the term $I_2$, by Lemma \ref{basic lemma}-(iii), \textbf{(H1)}-\textbf{(H3)}, Lemmas \ref{xxi} and   \ref{x0xi} we have
	\begin{align}
		I_2&\leq C_{p,T,N} \int_{0}^{t}K^2_1(t,s)\cdot \E\Big\|N(s,X^{0}_{s,\xi},\mathcal{L}_{X_{s,\xi}^{\varepsilon}},\Lambda^{\varepsilon,v}_{s,x},h(\varepsilon))\Big\|^p\d s\nonumber\\&\leq C_{p,T,N}\int_{0}^{t}K^2_1(t,s)\cdot \E\big|1+|\Lambda_{s,x}^{\varepsilon,v}|+|X^{0}_{s,\xi}|+\mathbb{W}_2(\mathcal{L}_{X_{s,\xi}^{\varepsilon}},\delta_{0})\big|^p\d s\nonumber\\&\leq C_{p,T,N}\int_{0}^{t}K^2_1(t,s)\cdot\big(1+\E|X_{s,\xi}^{\varepsilon,v}|^p+\E|\Lambda_{s,x}^{\varepsilon,v}|^p+\E|X^{0}_{s,\xi}|^p\big)\d s\nonumber\\&\leq C_{p,T,N,L_2}\big(1+|\xi|^p\big)+C_{p,T,N}\int_{0}^{t}K^2_1(t,s)\cdot \E|\Lambda_{s,x}^{\varepsilon,v}|^p\d s.\nonumber
	\end{align}
	
	For the term $I_3$, applying \ref{basic lemma}-(ii), \textbf{(H1)}-\textbf{(H3)}, Lemmas \ref{xxi} and   \ref{x0xi} yields 
	\begin{align}
		I_3&\leq C_{p,T}\int_{0}^{t}K_2^2(t,s)\cdot \E\Big|G(s,X^{0}_{s,\xi},\mathcal{L}_{X_{s,\xi}^{\varepsilon}},\Lambda^{\varepsilon,v}_{s,x},h(\varepsilon))\Big|^p\d s\nonumber\\&\leq C_{p,T}\int_{0}^{t}K^2_1(t,s)\cdot \E\big|1+|\Lambda_{s,x}^{\varepsilon,v}|+|X^{0}_{s,\xi}|+\mathbb{W}_2(\mathcal{L}_{X_{s,\xi}^{\varepsilon}},\delta_{0})\big|^p\d s\nonumber\\&\leq C_{p,T}\int_{0}^{t}K^2_1(t,s)\cdot\big(1+\E|X_{s,\xi}^{\varepsilon,v}|^p+\E|\Lambda_{s,x}^{\varepsilon,v}|^p+\E|X^{0}_{s,\xi}|^p\big)\d s\nonumber\\&\leq C_{p,T,L_2}\big(1+|\xi|^p\big)+C_{p,T}\int_{0}^{t}K^2_1(t,s)\cdot \E|\Lambda_{s,x}^{\varepsilon,v}|^p\d s.\nonumber 
	\end{align}
	Therefore
	$$\E|\Lambda_{t,x}^{\varepsilon,v}|^p\leq C_{p,T,N,L_1,L_2}\big(1+|\xi|^p\big)+\int_{0}^{t}K^{*}_{1,2}(t,s)\cdot \E |\Lambda_{s,x}^{\varepsilon,v}|^p\d s,$$
	where $K^*_{1,2}(t,s):=C_{p,T}\cdot (K_1(t,s)+K^2_1(t,s))+C_{p,T,N}K_2^2(t,s)$ and the constants $C_{p,T}, C_{p,T,N}$ are independent of $\varepsilon$.
	
	Let $R^{K_{1,2}}$ be defined by \eqref{Resolvent} in terms of $K^*_{1,2}$.  By Remark \ref{condition of Volterra}, Lemma \ref{Gronwall} and \eqref{finite of resolvent}, we obtain that for almost all $t\in [0,T]$,
	\begin{align}
		\E|\Lambda_{t,x}^{\varepsilon,v}|^{p}&\leq C_{p,T,N,L_1,L_2}\big(1+|\xi|^p\big)+\int_{0}^{t}R^{K_{1,2}}(t,s)\cdot C_{p,T,N,L_1,L_2}\big(1+|\xi|^p\big)\d s\nonumber\\&\leq C_{p,T,N,L_1,L_2}(1+|\xi|^{p}).\nonumber
	\end{align}
	The proof  of the lemma is then completed.
\end{proof}
\begin{lem}\label{Lambda-ts}
	Assume that  $\Lambda^{\varepsilon,v}_{t,x}$ satisfies  \eqref{eq of Lambda}  and \textbf{(H1)}-\textbf{(H3)} hold, if $p>2\vee \frac{2}{\gamma}$, $t,s\in [0,T]$, $N>0, \xi\in \mathbb{R}^d$ and $\{v^{\varepsilon}\}_{\varepsilon>0}$ is a family in $\mathcal{A}_N$, in addition, assume
	there exist a constant $C$ such that
	$$M(s,X^{0}_{s,\xi},\mathcal{L}_{X_{s,\xi}^{\varepsilon}},\delta_{X_{s,\xi}^{0}},\Lambda^{\varepsilon,v}_{s,x},h(\varepsilon))\leq C\big(1+|\Lambda_{s,x}^{\varepsilon,v}|+\mathbb{W}_2(\mathcal{L}_{X_{s,\xi}^{\varepsilon}},\delta_{0})+\frac{1}{\sqrt{\varepsilon}h(\varepsilon)}\mathbb{W}_2(\mathcal{L}_{X_{s,\xi}^{\varepsilon}},\delta_{X^0_{s,\xi}})\big),$$
	$$N(s,X^{0}_{s,\xi},\mathcal{L}_{X_{s,\xi}^{\varepsilon}},\Lambda^{\varepsilon,v}_{s,x},h(\varepsilon))\vee G(s,X^{0}_{s,\xi},\mathcal{L}_{X_{s,\xi}^{\varepsilon}},\Lambda^{\varepsilon,v}_{s,x},h(\varepsilon)) \leq C\big(1+|\Lambda_{s,x}^{\varepsilon,v}|+|X^{0}_{s,\xi}|+\mathbb{W}_2(\mathcal{L}_{X_{s,\xi}^{\varepsilon}},\delta_{0})\big).$$
	 Then $\Lambda^{\varepsilon,v^{\varepsilon}}$ admits a version which is H$\mathrm{\ddot{o}}$lder continuous on $[0,T]$ of any order $\alpha<\frac{\gamma}{2}-\frac{1}{p}$, uniformly for all $\varepsilon>0$. Denoting again this version by $\Lambda^{\varepsilon,v^{\varepsilon}}$, one can find a $\varepsilon_0>0$ such that %one has for all $\varepsilon>0$ small enough
	
	$$\sup_{0<\varepsilon<\varepsilon_0} \E\left[\left(\sup_{0\leq s<t\leq T}\frac{\left| \Lambda^{\varepsilon,v^{\varepsilon}}_{t,x}-\Lambda^{\varepsilon,v^{\varepsilon}}_{s,x}\right|}{|t-s|^{\alpha}}\right)^p\right]\leq C_{\alpha, p,T,N,L_2,\xi},$$
	for all $\a\in [0,\frac{\gamma}{2}-\frac{1}{p})$. Moreover, the family of random processes  $\{\Lambda^{\varepsilon,v^{\varepsilon}}\}_{\varepsilon>0}$ is tight in $C_T$.
\end{lem}
\begin{proof}
	For any $s<t$, we have
	\begin{align}
		&\Lambda^{\varepsilon,v^{\varepsilon}}_{t,x}-\Lambda^{\varepsilon,v^{\varepsilon}}_{s,x}\nonumber\\&=\int_{s}^{t}K_1(t,u)M(s,X^{0}_{s,\xi},\mathcal{L}_{X_{s,\xi}^{\varepsilon}},\delta_{X_{s,\xi}^{0}},\Lambda^{\varepsilon,v^{\varepsilon}}_{s,x},h(\varepsilon))\d u+\int_{s}^{t}K_1(t,u)N(s,X^{0}_{s,\xi},\mathcal{L}_{X_{s,\xi}^{\varepsilon}},\Lambda^{\varepsilon,v}_{s,x},h(\varepsilon))v_u\d u\nonumber\\&\quad+\int_{s}^{t}K_2(t,u)G(s,X^{0}_{s,\xi},\mathcal{L}_{X_{s,\xi}^{\varepsilon}},\Lambda^{\varepsilon,v}_{s,x},h(\varepsilon))\d W_u\nonumber\\&\quad+\int_{0}^{s}\Big(K_1(t,u)-K_1(s,u)\Big)\cdot M(s,X^{0}_{s,\xi},\mathcal{L}_{X_{s,\xi}^{\varepsilon}},\delta_{X_{s,\xi}^{0}},\Lambda^{\varepsilon,v^{\varepsilon}}_{s,x},h(\varepsilon))\d u\nonumber\\&\quad+\int_{0}^{s}\Big(K_1(t,u)-K_1(s,u)\Big)\cdot N(s,X^{0}_{s,\xi},\mathcal{L}_{X_{s,\xi}^{\varepsilon}},\Lambda^{\varepsilon,v}_{s,x},h(\varepsilon))v_u\d u\nonumber\\&\quad+\int_{0}^{s}\Big(K_2(t,u)-K_2(s,u)\Big)\cdot  G(s,X^{0}_{s,\xi},\mathcal{L}_{X_{s,\xi}^{\varepsilon}},\Lambda^{\varepsilon,v}_{s,x},h(\varepsilon))\d W_u\nonumber\\&:=J_1+J_2+J_3+J_4+J_5+J_6.\nonumber
	\end{align}
Note that under the assumptions on  $M,N,G$, Lemmas \ref{xxi},  \ref{x0xi},   \ref{esi-Q} and   \ref{estimate of Lambda} give us
$$\sup_{s\in[0,T]}\E|M(s,X^{0}_{s,\xi},\mathcal{L}_{X_{s,\xi}^{\varepsilon}},\delta_{X_{s,\xi}^{0}},\Lambda^{\varepsilon,v^{\varepsilon}}_{s,x},h(\varepsilon))|^p\leq C_{p,T,N,L_1,L_2}\big(1+|\xi|^p\big),$$
$$\sup_{s\in[0,T]}\E\|N(s,X^{0}_{s,\xi},\mathcal{L}_{X_{s,\xi}^{\varepsilon}},\Lambda^{\varepsilon,v}_{s,x},h(\varepsilon))\|^p\leq C_{p,T,N,L_2}\big(1+|\xi|^p\big),$$
$$\sup_{s\in[0,T]}\E\|G(s,X^{0}_{s,\xi},\mathcal{L}_{X_{s,\xi}^{\varepsilon}},\Lambda^{\varepsilon,v}_{s,x},h(\varepsilon))\|^p\leq C_{p,T,N,L_2}\big(1+|\xi|^p\big).$$	
	For the terms $J_1$ and $J_4$, combining Lemma \ref{basic lemma}-(iv), Lemmas \ref{xxi}-\ref{estimate of Lambda} we have
	\begin{align}
		\E|J_1|^p+\E|J_4|^p&\leq C_{p,T}|t-s|^{\gamma p}\sup_{s\in[0,T]}\E|M(s,X^{0}_{s,\xi},\mathcal{L}_{X_{s,\xi}^{\varepsilon}},\delta_{X_{s,\xi}^{0}},\Lambda^{\varepsilon,v^{\varepsilon}}_{s,x},h(\varepsilon))|^p\nonumber\\&\leq C_{p,T,N,L_1,L_2}\big(1+|\xi|^p\big)|t-s|^{\gamma p}.\nonumber
	\end{align}
	For the terms $J_2$ and $J_5$, by Lemma \ref{basic lemma}-(vi), Lemmas \ref{xxi}-\ref{estimate of Lambda}, we have
	\begin{align}
		\E|J_2|^p+\E|J_5|^p&\leq C_{T,N}|t-s|^{\gamma p}\sup_{s\in[0,T]}\E\|N(s,X^{0}_{s,\xi},\mathcal{L}_{X_{s,\xi}^{\varepsilon}},\Lambda^{\varepsilon,v}_{s,x},h(\varepsilon))\|^p\nonumber\\&\leq C_{p,T,N,L_2}\big(1+|\xi|^p\big)|t-s|^{\gamma p}.\nonumber
	\end{align}
	For the terms $J_3$ and $J_6$, by Lemma \ref{basic lemma}-(v), Lemmas \ref{xxi}-\ref{estimate of Lambda}, we have
	\begin{align}
		\E|J_3|^p+\E|J_6|^p&\leq C_{p,T}|t-s|^{\gamma p}\sup_{t\in [0,T]}\E\|G(s,X^{0}_{s,\xi},\mathcal{L}_{X_{s,\xi}^{\varepsilon}},\Lambda^{\varepsilon,v}_{s,x},h(\varepsilon))\|^p\nonumber\\&\leq C_{p,T,N,L_2}\big(1+|\xi|^p\big)|t-s|^{\gamma p}.\nonumber
	\end{align}
	Combining the above estimations, we have
	\begin{align}
		\E|\Lambda_{t,x}^{\varepsilon,v^{\varepsilon}}-\Lambda_{s,x}^{\varepsilon,v^{\varepsilon}}|^p\leq C_{p,T,N,L_1,L_2}\big(1+|\xi|^p\big)|t-s|^{\gamma p}.\nonumber
	\end{align} 
	This gives the bound in the lemma. The existence of a continuous version   now follows from the Kolmogorov's continuity theorem (e.g.   Theorem I. $2.1$ in \cite{RY}). Furthermore, Aldous theorem (\cite[Theorem 16.10]{Bil}) 
	can be used to conclude  that
	the sequence $\{X^{\varepsilon,v^{\varepsilon}}\}_{\varepsilon>0}$ is tight.
\end{proof}
	\section{Large deviation principle for Volterra type McKean-Vlasov SDEs}\label{The large deviation principle for Volterra type McKean-Valsov SDEs}
	
	In this section, we will use the weak convergence approach to obtain the  LDP for Volterra type McKean-Vlasov SDEs. 
	
	It is clear that under the assumptions \textbf{(H1)}-\textbf{(H3)}, there exists a unique strong solution $X^{\varepsilon}_{t,\xi}$ of \eqref{Volterra MV eq} (Theorem $3.1$ of \cite{LG}). As $\varepsilon$ tends to zero, $X^{\varepsilon}_{\cdot,\xi}$ will converge to
	$X^0_{\cdot,\xi}$, where $X^0_{\cdot,\xi}$ satisfies 
	\eqref{Volterra eq-limit eq}.
%	the following determinstic Volterra integral equation:
%	\begin{align}\label{Volterra eq-limit eq}
%		X^0_t=\xi+\int_{0}^{t}K_1(t,s)b(s,X^0_s,\delta_{X^0_s})\d s.
%	\end{align}
	In order to prove Laplace principle for \eqref{Volterra MV eq}, we  take 
	$$E=C_T^d: =C([0,T],\mathbb{R}^d),\quad  \mathcal{G}^{\varepsilon}(W):=X^{\varepsilon}_{\cdot,\xi}\,, $$
	where  $ C([0,T],\mathbb{R}^d)$ denotes the space of continuous functions from $[0,T]$ to $\mathbb{R}^d$ equipped with the supremum norm $\|\phi\|_T:=\sup_{t\in [0,T]}|\phi_t|$. Then by Girsanov's theorem, the shifted version $X^{\varepsilon,v}_{\cdot,\xi}:=\mathcal{G}^{\varepsilon}\big(W+\frac{1}{\sqrt{\varepsilon}}\int_{0}^{\cdot}v_s\d s\big)$ is the unique strong solution of the controlled equation under $\mathbb{P}$:
	\begin{align}\label{controlled Volterra eq}
		X^{\varepsilon,v}_{t,\xi}=\xi+&\int_{0}^{t}K_1(t,s)b(s,X^{\varepsilon,v}_{s,\xi},\mathcal{L}_{X^{\varepsilon}_{s,\xi}})\d s+\int_{0}^{t}K_1(t,s)\sigma(s,X^{\varepsilon,v}_{s,\xi},\mathcal{L}_{X^{\varepsilon}_{s,\xi}})v_s\d s\nonumber\\&+\sqrt{\varepsilon}\int_{0}^{t}K_2(t,s)\sigma(s,X^{\varepsilon,v}_{s,\xi},\mathcal{L}_{X^{\varepsilon}_{s,\xi}})\d W_s.
	\end{align}
	Under some appropriate conditions, we will prove that $X^{\varepsilon,v}_{t,\xi}$ will converge to $\phi_t$ as $\varepsilon\to 0$,  where $\phi_t$ is the unique solution of  the  following deterministic Volterra equation
	\begin{align}\label{controlled determinstic Volterra eq}
		\phi_t=\xi+\int_{0}^{t}K_1(t,s)b(s,\phi_s,\delta_{X^0_{s,\xi}})\d s+\int_{0}^{t}K_1(t,s)\sigma(s,\phi_s,\delta_{X^0_{s,\xi}})v_s\d s.
	\end{align}
	We now state the results of LDP.
	\begin{thm}\label{thm-LDP}
		Under the assumptions \textbf{(H1)}-\textbf{(H3)}, the family $\{X^{\varepsilon}_{\cdot,\xi}\}_{\varepsilon>0}$, the unique solution of \eqref{Volterra MV eq}, satisfies the  Large Deviation  Principle with rate function \eqref{def of I} and speed $\varepsilon^{-1}$, where $\mathcal{G}^0_v$ corresponds to the unique solution of the limiting equation \eqref{controlled determinstic Volterra eq}.
	\end{thm}
	
%	$\bullet$ Verification of Condition $(i)$.
We shall use Lemma \ref{general LDP} to prove this theorem. Note that the third condition in Lemma \ref{general LDP} is automatically satisfied if the limiting equation \eqref{controlled determinstic Volterra eq} has a unique solution. We shall verify the first two conditions in Lemma 
\ref{general LDP}, which will be given by Lemmas \ref{estimate of X v} and \ref{compact level sets-LDP}. But before this we need some preparation. 
 \begin{lem}\label{existence}
 	Assume that \textbf{(H1)}-\textbf{(H3)} hold, then there exists a unique solution $\phi_t$ such that for almost all $t\geq0$, $N>0,v\in\mathcal{A}_N, \xi\in \mathbb{R}^d$, some constant $C_{T,N,L_2}$
 	\begin{align}\label{moment estimate of p}
 		\sup_{t\in [0,T]}|\phi_t|\leq C_{T,N,L_2}(1+|\xi|),
 	\end{align}
 	for almost all $t\in [0,T]$.
 \end{lem}
 Lemma  \ref{existence} implies that the solution of \eqref{controlled determinstic Volterra eq} is unique. By Remark $3.12$ in \cite{JP}, we only need to check Conditions $(i)$ and $(ii)$ in Theorem \ref{general LDP}.

 Let
 $$x:=\xi,\Lambda^{\varepsilon,v}_{t,\xi}:=X^{\varepsilon,v}_{t,\xi},M(s,X^{0}_{s,\xi},\mathcal{L}_{X_{s,\xi}^{\varepsilon}},\delta_{X_{s,\xi}^{0}},\Lambda^{\varepsilon,v}_{s,x},h(\varepsilon)):=b(s,X^{\varepsilon,v}_{s,\xi},\mathcal{L}_{X^{\varepsilon}_{s,\xi}})$$
 and
 $$N(s,X^{0}_{s,\xi},\mathcal{L}_{X_{s,\xi}^{\varepsilon}},\Lambda^{\varepsilon,v}_{s,x},h(\varepsilon)):=\sigma(s,X^{\varepsilon,v}_{s,\xi},\mathcal{L}_{X^{\varepsilon}_{s,\xi}}), G(s,X^{0}_{s,\xi},\mathcal{L}_{X_{s,\xi}^{\varepsilon}},\Lambda^{\varepsilon,v}_{s,x},h(\varepsilon)):=\sqrt{\varepsilon}\sigma(s,X^{0}_{s,\xi},\mathcal{L}_{X_{s,\xi}^{\varepsilon}},\Lambda^{\varepsilon,v}_{s,x},h(\varepsilon)).$$
 Then Lemmas \ref{estimate of Lambda} and  \ref{Lambda-ts} give us that
	\begin{lem}\label{estimate of X v}
		Under the assumptions \textbf{(H1)}-\textbf{(H3)}, it follows that for $p> 2, N>0,v\in\mathcal{A}_N,  \xi\in \mathbb{R}^d$ Hence $\{X^{n,v},v^n\}_{n\geq0}$ converges almost surely in the product
		topology on $C_T^d\times \mathcal{S}_N$,
		and the limit is the $C_T^d\times \mathcal{S}_N$-valued random variable $(\phi, v)$ in a possibly different probability space $(\Omega^0,\mathcal{F}^0,\mathbb{P}^0)$ as $n$ tends to $+\infty$.
		$$\sup_{t\in  [0,T]}\E|X^{\varepsilon,v}_{t,\xi}|^{p}\leq C_{p,T,N,L_2}(1+|\xi|^{p}),$$
		where the constant $C_{p,T,N,L_2}$ is independent of $\varepsilon\in (0, 1],v\in\mathcal{A}_N$ (but it depends on $\mathcal{A}_N$ or $N$).
	\end{lem}
\begin{lem}\label{X vare v t s}
	Assume that \textbf{(H1)}-\textbf{(H3)} hold, if $p>2\vee \frac{2}{\gamma}$, $t,s\in [0,T]$, $N>0, \xi\in \mathbb{R}^d$ and $\{v^{\varepsilon}\}_{\varepsilon>0}$ is a family in $\mathcal{A}_N$, then $X^{\varepsilon,v^{\varepsilon}}$ admits a version which is H$\mathrm{\ddot{o}}$lder continuous on $[0,T]$ of any order $\alpha<\frac{\gamma}{2}-\frac{1}{p}$, uniformly for all $\varepsilon>0$. Denoting again this version by $X^{\varepsilon,v^{\varepsilon}}$, one has for all $\varepsilon>0$ small enough
	
	$$\E\Big[\Big(\sup_{0\leq s<t\leq T}\frac{X^{\varepsilon,v^{\varepsilon}}_{t,\xi}-X^{\varepsilon,v^{\varepsilon}}_{s,\xi}}{|t-s|^{\alpha}}\Big)^p\Big]\leq C_{p,T,N,L_2,\xi},$$
	for all $\a\in [0,\frac{\gamma}{2}-\frac{1}{p})$. Moreover, the family of random processes  $\{X^{\varepsilon,v^{\varepsilon}}_{\cdot,\xi}\}_{\varepsilon>0}$ is tight in $C_T$.
\end{lem}
By a similar  procedure, we can get that the solution $\phi_{t}$ of \eqref{controlled determinstic Volterra eq} also have H$\mathrm{\ddot{o}}$lder continuous path, that is
\begin{lem}\label{small esti}
	Assume that \textbf{(H1)}-\textbf{(H3)} hold, for $t,s\in [0,T]$, $N>0, v\in\mathcal{A}_N, \xi\in \mathbb{R}^d$, then the solution $\phi_{t}$ of \eqref{controlled determinstic Volterra eq} has for all $\varepsilon>0$ small enough
	
	$$|\phi_{t}-\phi_s|\leq C_{T,N,L_2}\big(1+|\xi|\big)|t-s|^{\gamma}.$$ 
\end{lem}	
%	$\bullet$ Verification of Condition $(ii)$.
Since $X^{\varepsilon}_{t,\xi}-X^0_{t,\xi}=Q^{\varepsilon}_{t,\xi}\vert_{h(\varepsilon)=\frac{1}{\sqrt{\varepsilon}}}$, so Lemma \ref{esi-Q} is simplified as
	\begin{lem}\label{xx}
		Under the assumptions \textbf{(H1)}-\textbf{(H3)}, it follows that for $p> 2,\xi\in \mathbb{R}^d$
		$$\sup_{t\in [0,T]}\E|X^{\varepsilon}_{t,\xi}-X^0_{t,\xi}|^{p}\leq \varepsilon^{\frac{p}{2}}C_{p,T,L_1,L_2}\big(1+|\xi|^p\big),$$
		where the constant $C_{p,T,L_1,L_2}$ is independent of $\varepsilon$.
	\end{lem}
	\begin{lem}\label{unique characterised}
		The set $\mathcal{G}^0_v$ from Definition \ref{v} is characterized by
		$$\mathcal{G}^0_v=\Big\{\phi: [0,T]\to\mathbb{R}^d|	\phi_t=\xi+\int_{0}^{t}K_1(t,s)b(s,\phi_s,\delta_{X^0_{s,\xi}})\d s+\int_{0}^{t}K_1(t,s)\sigma(s,\phi_s,\delta_{X^0_{s,\xi}})v_s\d s\Big\}.$$
	\end{lem}
	\begin{proof}
		The proof of this lemma follow the technique                                                                                                                                                                                                                             in Theorem $1$ \cite{CF} or Lemma $3.17$ of \cite{JP}. Before we characterize $\mathcal{G}^0_v$, we first identify $\mathcal{G}^0_{v,N}$. Recall the definition \ref{v} of $\mathcal{G}^0_{v,N}$. For $N\in\mathbb{N}$ and $v\in\mathcal{A}_N$, consider a subsequence $\{\varepsilon_n\}_{n\in\mathbb{N}}\subset[0,\infty)$ with $\lim_{n\to \infty}\varepsilon_{n}=0$ and a sequence $\{v^{\varepsilon_n}\}_{n\in\mathbb{N}}\subset\mathcal{A}_N$ such that $\lim_{n\to \infty}v^{\varepsilon_{n}}=v\ \text{in distribution}$, and assume that $X^{n,v}:=X^{\varepsilon_n,v^{\varepsilon_{n}}}$ converges in distribution to some random variable $\phi$ with values in $C^d_T$. To simplify the representation, we define $v^{n}:=v^{\varepsilon_{n}}, X^{\varepsilon_{n}}:=X^{n}$.
		
		By Skorohod representation theorem (\cite[Theorem  3.15 ]{DW}) we can work with almost sure convergence for the purpose of identifying the limit. 
		 Hence $\{X^{n,v},v^n\}_{n\geq0}$ converges almost surely in the product
		topology on $C_T^d\times \mathcal{S}_N$,
		 and the limit is a $C_T^d\times \mathcal{S}_N$-valued random variable $(\phi, v)$ in a possibly different probability space $(\Omega^0,\mathcal{F}^0,\mathbb{P}^0)$ as $n$ tends to $+\infty$. The convergence of the couple also takes place in distribution. For $t\in [0,T]$, define $\Phi_t:C_T^d\times \mathcal{S}_N\to\mathbb{R}$ as
		$$\Phi_t(\omega,f):=\Big|\omega_t-\xi-\int_{0}^{t}K_1(t,s)\cdot\big(b(s,\omega_s,\delta_{X^0_{s,\xi}})+\sigma(s,\omega_s,\delta_{X^0_{s,\xi}})f_s\big)\d s\Big|\wedge 1.$$
		It is clear that $\Phi_t$ is bounded, and we also can show $\Phi_t$ is continuous. Let $\omega^n\to \omega$ in $C_T^d$ and $f^n\to f$ in $\mathcal{S}_N$ with respect the weak topology. 
		
		For $\omega,\omega^n\in C^d_T,f,f^n\in\mathcal{S}_N$, by Lemma \ref{x0xi} and \textbf{(H1)-(H3)} we have that
		\begin{align}
			&|\Phi_t(\omega,f)-\Phi_t(\omega^n,f^n)|\leq |\omega_t-\omega^n_t|+\int_{0}^{t}K_1(t,s)\cdot \big|b(s,\omega_s,\delta_{X^0_{s,\xi}})-b(s,\omega^n_s,\delta_{X^0_{s,\xi}})\big|\d s\nonumber\\&\quad+\int_{0}^{t}K_1(t,s)\cdot\Big|\big(\sigma(s,\omega_s,\delta_{X^0_{s,\xi}})-\sigma(s,\omega^n_s,\delta_{X^0_{s,\xi}})\big)f^n_s+\big(f_s-f^n_s\big)\sigma(s,\omega_s,\delta_{X^0_{s,\xi}})\Big|\d s\nonumber\\&\leq \sup_{t\in [0,T]}|\omega_t-\omega^n_t|+\Big(\int_{0}^{t}K_1^2(t,s)\d s\Big)^{\frac{1}{2}}\cdot \Big(\int_{0}^{t}\big|b(s,\omega_s,\delta_{X^0_{s,\xi}})-b(s,\omega^n_s,\delta_{X^0_{s,\xi}})\big|^2\d s\Big)^{\frac{1}{2}}\nonumber\\&\quad+\Big(\int_{0}^{t}|f_s^n|^2\d s\Big)^{\frac{1}{2}}\cdot \Big(\int_{0}^{t}K^2_1(t,s)\cdot\big\|\sigma(s,\omega_s,\mathcal{L}_{X^0_{s,\xi}})-\sigma(s,\omega^n_s,\mathcal{L}_{X^0_{s,\xi}})\big\|^2\d s\Big)^{\frac{1}{2}}\nonumber\\&\quad+\sup_{t\in [0,T]}\|\sigma(t,\omega_t,\mathcal{L}_{X^0_{t,\xi}})\|\int_{0}^{t}K_1(t,s)\cdot|f_s-f_s^n|\d s\nonumber\\&\leq \sup_{t\in [0,T]}|\omega_t-\omega^n_t|+C_{T,L_1}\Big(\int_{0}^{t}|\omega_s-\omega_s^n|^2\d s\Big)^{\frac{1}{2}}+C_{N,L_1}\Big(\int_{0}^{t}K_1^2(t,s)\cdot|\omega_s-\omega_s^n|^2\d s\Big)^{\frac{1}{2}}\nonumber\\&\quad+C_{L_2}\sup_{t\in [0,T]}\big(1+|\omega_t|+|X^0_t|)\big)\int_{0}^{t}K_1(t,s)\cdot|f_s-f_s^n|\d s\nonumber\\&\leq \sup_{t\in [0,T]}|\omega_t-\omega^n_t|+C_{T,L_1}\sup_{t\in [0,T]}|\omega_t-\omega^n_t|+C_{N,T,L_1}\sup_{t\in [0,T]}|\omega_t-\omega^n_t|\nonumber\\&\quad+C_{T,L_2}\big(1+\sup_{t\in [0,T]}\omega_t\big)\int_{0}^{t}K_1(t,s)\cdot|f_s-f_s^n|\d s\nonumber\\&\leq C_{N,T,L_1}\sup_{t\in [0,T]}|\omega_t-\omega^n_t|+C_{L_2}\big(1+\sup_{t\in [0,T]}\omega_t\big)\int_{0}^{t}K_1(t,s)\cdot|f_s-f_s^n|\d s.\nonumber
		\end{align}
		Since $f^n$ tends to $f$ weakly in $L^2$ and $\lim_{n\to \infty}|\omega_t-\omega^n_t|=0$, so we get $\Phi_t(f^n,\omega^n)$ tends to $\Phi_t(f,\omega)$ as $n$ tends to $\infty$, which implies $\Phi_t$ is continuous and therefore
		$$\lim_{n\to \infty}\E[\Phi_t(X^{n,v},v^n)]=\E[\Phi_t(\phi,v)].$$
		It remains to  prove that the right hand of above equality is actually equal to zero. By BDG's inequality, Minkowski's inequality, Lemmas \ref{xxi},   \ref{estimate of X v} and   \ref{xx} we have
		\begin{align}
		 \E[\Phi_t(X^{n,v},v^n)]&\leq \E\Big|\int_{0}^{t}K_1(t,s)\cdot\big(b(s,X^{n,v}_s,\mathcal{L}_{X^{n}_s})-b(s,X^{n,v}_s,\delta_{X^0_{s,\xi}})\big)\d s\Big|\nonumber \\ &\quad+\E\Big|\int_{0}^{t}K_1(t,s)\cdot\big(\sigma(s,X^{n,v}_s,\mathcal{L}_{X^{n}_s})-\sigma(s,X^{n,v}_s,\delta_{X^0_{s,\xi}})\big)v^n_s\d s\Big|\nonumber\\&\quad+\sqrt{\varepsilon_n}\E\Big|\int_{0}^{t}K_2(t,s)\cdot\sigma(s,X^{n,v},\mathcal{L}_{X^{n}_s}) \d W_s\Big|\nonumber\\&\leq L_1\E\int_{0}^{t}K_1(t,s)\cdot \mathbb{W}_2(\mathcal{L}_{X^{n}_s},\delta_{X^0_{s,\xi}})\d s\nonumber\\&\quad+\Big(\int_{0}^{t}|v^n_s|^2\d s\Big)^{\frac{1}{2}}\cdot \E\Big(\int_{0}^{t}K^2_1(t,s)\cdot\big\|\sigma(s,X^{n,v}_s,\mathcal{L}_{X^{n}_s})-\sigma(s,X^{n,v}_s,\delta_{X^0_{s,\xi}})\big\|^2\d s\Big)^{\frac{1}{2}}\nonumber\\&\quad+ \sqrt{\varepsilon_n}\E\Big(\int_{0}^{t}K^2_2(t,s)\cdot\|\sigma(s,X^{n,v}_s,\mathcal{L}_{X^n_s})\|^2\d s\Big)^{\frac{1}{2}}\nonumber\\&\leq L_1\int_{0}^{t}K_1(t,s)\cdot\big(\E|X^n_s-X^0_{s,\xi}|^p\big)^{\frac{1}{p}}\d s\nonumber\\&\quad+\sqrt{NL_2}\E\Big(\int_{0}^{t}K^2_1(t,s)\cdot\mathbb{W}_2^2(\mathcal{L}_{X^n_s},\delta_{X^0_{s,\xi}})\d s\Big)^{\frac{1}{2}}\nonumber\\&+\sqrt{\varepsilon_n L_2}\E\Big(\int_{0}^{t}K^2_2(t,s)\cdot \big(1+|X^{n,v}_s|^2+\mathbb{W}_2^2(\mathcal{L}_{X^n_s},\delta_{0})\big)\d s\Big)^{\frac{1}{2}}\nonumber\\&\leq \sqrt{\varepsilon_{n}}C_{T,L_1,L_2,\xi}+\sqrt{NL_2}\Big[\E\Big(\int_{0}^{t}K^2_1(t,s)\cdot \mathbb{W}_2^2(\mathcal{L}_{X^n_s},\delta_{X^0_{s,\xi}})\d s\Big)^{\frac{p}{2}}\Big]^{\frac{1}{p}}\nonumber\\&\quad+\sqrt{\varepsilon_n L_2}\Big[\E\Big(\int_{0}^{t}K^2_1(t,s)\cdot \big(1+|X^{n,v}_s|^2+\mathbb{W}_2^2(\mathcal{L}_{X^n_s},\delta_{0})\big)\d s\Big)^{\frac{p}{2}}\Big]^{\frac{1}{p}}\nonumber\\&\leq \sqrt{\varepsilon_{n}}C_{T,L_1,L_2,\xi}+\sqrt{NL_2}\Big[\Big(\int_{0}^{t}K^2_1(t,s)\cdot \E\big(\mathbb{W}_2^p(\mathcal{L}_{X^n_s},\delta_{X^0_{s,\xi}})\big)^{\frac{2}{p}}\d s\Big)^{\frac{p}{2}}\Big]^{\frac{1}{p}}\nonumber\\&\quad+ \sqrt{\varepsilon_n L_2}\Big[\Big(\int_{0}^{t}K^2_1(t,s)\cdot \E\big(1+|X^{n,v}_s|^p+\mathbb{W}_2^p(\mathcal{L}_{X^n_s},\delta_{0})\big)^{\frac{2}{p}}\d s\Big)^{\frac{p}{2}}\Big]^{\frac{1}{p}}\nonumber\\&\leq \sqrt{\varepsilon_{n}}C_{T,N,L_1,L_2,\xi}+ \sqrt{\varepsilon_n}C_{T,N,L_2,\xi}\nonumber\\&\leq C_{T,N,L_1,L_2,\xi}\sqrt{\varepsilon_n}\to 0\quad \text{as $n\to\infty$}.\nonumber
		\end{align}
		Hence $\E[\Phi_t(\phi,v)]=0$, which implies that $\phi$ satisfies \eqref{controlled determinstic Volterra eq} almost surely, for all $t\in T$. By Lemma \ref{small esti}, we have that $\phi$ has continuous paths, so it satisfies \eqref{controlled determinstic Volterra eq} for all $t\in T$, almost surely, which means $\mathcal{G}^0_{v,N}$ consists of the unique solution of \eqref{controlled determinstic Volterra eq}. Since this definition is independent of $N$, it extends to $\mathcal{G}_v^0$. The proof is complete.
	\end{proof}
	This Lemma shows that the set $\mathcal{G}^0_v$ corresponds to the unique solution $\phi_t$ of \eqref{controlled determinstic Volterra eq}.
	\begin{lem}\label{compact level sets-LDP}
		Assume that \textbf{(H1)}-\textbf{(H3)} hold, the functional $I$ in \eqref{def of I} has compact level sets.
	\end{lem}
	\begin{proof}
		By Lemma \ref{unique characterised}, we have that
		\small\begin{align}
			I(\phi):=\inf\Big\{\frac{1}{2}\int_{0}^{T}|v_s|^2ds: v\in L^2 ,\ \phi_t=\xi+\int_{0}^{t}K_1(t,s)b(s,\phi_s,\delta_{X^0_{s,\xi}})\d s+\int_{0}^{t}K_1(t,s)\sigma(s,\phi_s,\delta_{X^0_{s,\xi}})v_s\d s\Big\}.\nonumber
		\end{align}\normalsize
		We are going to  prove that the functional $I$ defined in \eqref{def of I} has compact level sets, that is for all $N>0$, the level sets
		$$L_N:=\{\phi\in C^d_T: I(\phi)\leq N\}$$
		are compact. Fix $N>0$ and consider an arbitrary sequence $\{\phi^n\}_{n\in\mathbb{N}}\subset L_N$, we will show that there exists a converging subsequence such that the limit belongs to $L_N$.
		
		We first prove that $\{\phi^n\}_{n\in\mathbb{N}}$ is relative compact. According to classical Arzel$\grave{\mathrm{a}}$-Ascoli's Theorem, we only need to show that $\{\phi^n\}_{n\in\mathbb{N}}$ is uniformly bounded and equicontinuous. For all $n\in\mathbb{N}$ and $t\in [0,T]$, there exists $v^n\in L^2$ such that $\frac{1}{2}\int_{0}^{T}|v^n_t|^2\d t\leq N$ and $\phi^n\in \mathcal{G}^0_{v^n}$, which implies that $v^n\in\mathcal{S}_{2N}$ and $\phi^n$ is the solution of following Volterra type McKean-Vlasov integral equation:
		$$\phi^n_t=\xi+\int_{0}^{t}K_1(t,s)b(s,\phi^n_s,\delta_{X^0_{s,\xi}})\d s+\int_{0}^{t}K_1(t,s)\sigma(s,\phi^n_s,\delta_{X^0_{s,\xi}})v^n_s\d s.$$
		By Lemmas \ref{existence} and \ref{small esti}, we obtain $\{\phi^n\}_{n\in\mathbb{N}}$ is relatively compact which ensures that $L_N$ is relatively compact for any $N>0$.
		
		Next we show that $\{\phi^n\}_{n\in\mathbb{N}}$ is closed, that is, let $\{\phi^n\}_{n\in\mathbb{N}}$ be a converging sequence of $L_N$  and denote its limit by $\phi\in C^d_T$, we  want to  prove that the limit $\phi\in L_N$. Since $v^n\in\mathcal{S}_{2N}$ which is a compact space with respect to the weak topology. Hence there exists a subsequence $\{n_k\}_{k\in\mathbb{N}}$ such that $v^{n_k}$ 
		%converges weakly in $\{n_k\}_{k\in\mathbb{N}}$ such that $v{n_k}$ 
		converges weakly in $L^2$ to the limit $v\in\mathcal{S}_{2N}$ and $\lim_{k\to \infty}\phi^{n_k}=\phi$.
		
		Notice that for $v,v^{n_k}\in\mathcal{S}_{2N}$, by \textbf{(H1)-(H3)}, H$\mathrm{\ddot{o}}$lder's inequality, Lemmas \ref{x0xi} and  \ref{existence} we have that
	\begin{align}
			&\Big|\phi_t^{n_k}-\big(\xi+\int_{0}^{t}K_1(t,s)b(s,\phi_s,\delta_{X^0_{s,\xi}})\d s+\int_{0}^{t}K_1(t,s)\sigma(s,\phi_s,\delta_{X^0_{s,\xi}})v_s\d s\big)\Big|\nonumber\\&\leq \Big|\int_{0}^{t}K_1(t,s)\cdot\big(b(s,\phi_s^{n_k},\delta_{X^0_{s,\xi}})-b(s,\phi_s,\delta_{X^0_{s,\xi}})\big)\d s\Big|\nonumber\\&\quad+\Big|\int_{0}^{t}K_1(t,s)\cdot\big(\sigma(s,\phi_s^{n_k},\delta_{X^0_{s,\xi}})v^{n_k}_s-\sigma(s,\phi_s,\delta_{X^0_{s,\xi}})v_s\big)\d s\Big|\nonumber\\&\leq L_1\int_{0}^{t}K_1(t,s)\cdot|\phi^{n_k}_s-\phi_s|\d s+\int_{0}^{t}K_1(t,s)\cdot \big\|\sigma(s,\phi_s^{n_k},\delta_{X^0_{s,\xi}})-\sigma(s,\phi_s,\delta_{X^0_{s,\xi}})\big\|\cdot|v^{n_k}_s|\d s\nonumber\\&\quad+\int_{0}^{t}K_1(t,s)\cdot \big\|\sigma(s,\phi_s,\delta_{X^0_{s,\xi}})\big\|\cdot|v^{n_k}_s-v_s|\d s\nonumber\\&\leq L_1\int_{0}^{t}K_1(t,s)\cdot|\phi^{n_k}_s-\phi_s|\d s+\Big(\int_{0}^{t}|v^{n_k}_s|^2\d s\Big)^{\frac{1}{2}}\cdot\Big(\int_{0}^{t}K_1^2(t,s)\cdot\big\|\sigma(s,\phi_s^{n_k},\delta_{X^0_{s,\xi}})-\sigma(s,\phi_s,\delta_{X^0_{s,\xi}})\big\|^2\d s\Big)^{\frac{1}{2}}\nonumber\\&\quad+C_{L_2}\int_{0}^{t}K_1(t,s)\cdot\big(1+|\phi_s|+\mathbb{W}_2(\delta_{X^0_{s,\xi}},\delta_0)\big)\cdot|v_s^{n_k}-v_s|\d s\nonumber\\&\leq C_{L_1,T}\sup_{t\in [0,T]}|\phi^{n_k}_t-\phi_t|+C_{N,L_1}\Big(\int_{0}^{t}K_1^2(t,s)\cdot|\phi^{n_k}_s-\phi_s|^2\d s\Big)^{\frac{1}{2}}\nonumber\\&\quad+C_{L_2}(1+\sup_{t\in [0,T]}|\phi_t|+\sup_{t\in [0,T]}|X^0_t|)\cdot\int_{0}^{t}K_1(t,s)\cdot|v^{n_k}_s-v_s|\d s\nonumber\\&\leq C_{L_1,T}\sup_{t\in [0,T]}|\phi^{n_k}_t-\phi_t|+C_{N,L_1,T}\sup_{t\in [0,T]}|\phi^{n_k}_t-\phi_t|+C_{T,N,L_2}\int_{0}^{t}K_1(t,s)\cdot|v^{n_k}_s-v_s|\d s\nonumber\\&\leq C_{N,L_1,T}\sup_{t\in [0,T]}|\phi^{n_k}_t-\phi_t|\to 0\quad\text{as $n_k\to\infty$}\quad (\lim_{k\to \infty}\phi^{n_k}=\phi)\nonumber\\&\quad+C_{T,N,L_2}\int_{0}^{t}K_1(t,s)\cdot|v^{n_k}_s-v_s|\d s\to 0\quad\text{as $n_k\to\infty$}\quad (v^{n_k}\ \text{converges weakly in $L^2$ to}\ v).\nonumber
		\end{align}
		So we easily obtain that
		\begin{align}
			\phi_{t}&=\lim_{k\to \infty}\phi^{n_k}\nonumber\\&=\lim_{k\to \infty}\Big(\xi+\int_{0}^{t}K_1(t,s)b(s,\phi^{n_k}_s,\delta_{X^0_{s,\xi}})\d s+\int_{0}^{t}K_1(t,s)\sigma(s,\phi^{n_k}_s,\delta_{X^0_{s,\xi}})v^{n_k}_s\d s\Big)\nonumber\\&=\xi+\int_{0}^{t}K_1(t,s)b(s,\phi_s,\delta_{X^0_{s,\xi}})\d s+\int_{0}^{t}K_1(t,s)\sigma(s,\phi_s,\delta_{X^0_{s,\xi}})v_s\d s\nonumber.
		\end{align}
		Since $v\in\mathcal{S}_{2N}$ so that $\phi\in L_N$, which concludes the proof of the closure and therefore of the compactness of $L_N$.
	\end{proof}
Thus,  we have verified the two conditions in Lemma \ref{general LDP}  (by  Lemma \ref{X vare v t s} and Lemma \ref{compact level sets-LDP}) and   Theorem \ref{thm-LDP}  
	is proved.

	\section{The central limit Theorem for Volterra type McKean-Vlasov SDEs}\label{app to CLT}
	In this section, we will study the asymptotic behavior of 
	$$\frac{X_{t,\xi}^{\varepsilon}-X_{t,\xi}^0}{\sqrt{\varepsilon}}, \quad 0\leq t\le T,$$
	where $X^0$ satisfies the deterministic Volterra integral equation \eqref{Volterra eq-limit eq}. Now we impose some assumptions on the coefficient functions
	that will enable us to study central limit theorem:
	
	\textbf{(H4)}  For any $x,y\in\mathbb{R}^d$, $\mu,\nu\in \mathcal{P}_{2}(\mathbb{R}^d)$ and $t\in [0,T]$
	$$\|\nabla b(t,x,\mu)\|\leq L_3,\quad \|D^L b(t,x,\mu)\|_{T_{\mu,2}}\leq L_3,$$
	where $\nabla b(t,x,\mu)$ denotes the derivative of $b(t,x,\mu)$  with respect to  $x$, $D^L b(t,x,\mu)$ is the Lions derivative with respect to $\mu$, and $L_3>0$ is a positive constant. 
	
	\textbf{(H5)} For $t \in [0,T]$, $x,y\in\mathbb{R}^d$ and $X,Y,\phi\in L^2(\Omega\mapsto\mathbb{R}^d,\mathbb{P})$ 
 \[
\begin{split}
	 &\Big|\E\big\langle D^L b(t,x,\mathcal{L}_X)(X),\phi\big\rangle-\E\big\langle D^L b(t,y,\mathcal{L}_Y)(Y),\phi\big\rangle\Big|\nonumber\\
	 &\qquad \quad \leq L_4\Big(|x-y|+\mathbb{W}_{2}(\mathcal{L}_X,\mathcal{L}_Y)+(\E|X-Y|^2)^{\frac{1}{2}}\Big)\cdot \big(\E|\phi|^2\big)^{\frac{1}{2}}  
\end{split}
	\]
	for some positive constant $L_4>0$. 
	
	\textbf{(H6)} The derivative of $\nabla b(t,\cdot,\mu)$ exists and for any $x,y\in\mathbb{R}^d$, $\mu,\nu\in \mathcal{P}_{2}(\mathbb{R}^d)$ and $t \in [0,T]$
	$$\|\nabla b(t,x,\mu)-\nabla b(t,y,\nu)\|\leq L_5\Big(|x-y|+\mathbb{W}_{2}(\mu,\nu)\Big),$$
	for some positive constant $L_5>0$.
	 
	Define  $Z^{\varepsilon}_{t,\xi}=Q^{\varepsilon}_{t,\xi}\vert_{h(\varepsilon)=1}=\frac{X_{t,\xi}^{\varepsilon}-X_{t,\xi}^0}{\sqrt{\varepsilon}}$, and we can write   %then it is clear $Z^{\varepsilon}_{t,\xi}$ satisfies the following Volterra type McKean-Vlasov SDE:
	
 	\begin{align}\label{Volterra eq-after CLT}
		Z^\varepsilon_{t,\xi}=\int_{0}^{t}K_1(t,s)\cdot\frac{b(s,X^\varepsilon_{s,\xi},\mathcal{L}_{X^\varepsilon_{s,\xi}})-b(s,X^0_{s,\xi},\delta_{X^0_{s,\xi}})}{\sqrt{\varepsilon}}\d s+\int_{0}^{t}K_2(t,s)\sigma(s,X^\varepsilon_{s,\xi},\mathcal{L}_{X^\varepsilon_{s,\xi}})\d W_s,
	\end{align}

	Then we construct another Volterra type McKean-Vlasov SDE of $Z_{t,\xi}$:
	\begin{align}\label{Volterra eq-after CLT of Z}
		Z_{t,\xi}&=\int_{0}^{t}K_1(t,s)\cdot\Big[\nabla_{Z_{s,\xi}}b(s,X^0_{s,\xi},\delta_{X^0_{s,\xi}})+  \E \big\langle D^L b(s,X^0_{s,\xi},\delta_{X^0_{s,\xi}})(X^0_{s,\xi}),Z_{s,\xi}\big\rangle\Big]\d s\nonumber\\&\quad+\int_{0}^{t}K_2(t,s)\sigma(t,X^0_{s,\xi},\delta_{X^0_{s,\xi}})\d W_s.
	\end{align}
%	\footnote{ should not be $\EE$ here?} 
Assume \textbf{(H1)}-\textbf{(H4)}.  There exists a unique solution $Z_{t,\xi}$ of \eqref{Volterra eq-after CLT of Z} (Theorem $3.1$ of \cite{LG}). 
%So, the central limit theorem for \eqref{Volterra MV eq} means that
%	$$Z^{\varepsilon}_{t,\xi}\to Z_{t,\xi} \quad\text{in distribution},$$
%	for $t\in [0,T]$.
%	
	We now state main results in this section.
	\begin{thm}\label{CLT}
		Suppose that the assumptions \textbf{(H1)}-\textbf{(H6)} are satisfied, then it holds that for any $R>0$ and $p>2$ 
		$$\lim_{\varepsilon\rightarrow 0}\E\Big(\sup_{t\in [0,T], |\xi|\leq R}|Z^\varepsilon_{t,\xi}-Z_{t,\xi}|^p\Big)=0.$$
	\end{thm}
%{\red 	\begin{thm}\label{CLT-2}
%		Suppose that the assumptions \textbf{(H1)}-\textbf{(H6)} are satisfied, then it holds that for $p> \frac{1}{\gamma}\vee 2$ 
%		$$\lim_{\varepsilon\rightarrow 0}\E\Big(\sup_{t\in [0,T]} |Z^\varepsilon_{t,\xi}-Z_{t,\xi}|^p\Big)=0.$$
%	\end{thm} }\footnote{delete this? The previous one implies this? or you want to show the convergence in 
%	H\"older norm?}
	We first prepare some preliminary lemmas needed for the proof of the main theorem. 
	\begin{lem}\cite{RW}
		Let $(\Omega,\mathcal{F},(\mathcal{F}_t)_{t\ge0},\mathbb{P})$ be an atomless space, and let $X,Y\in L^2(\Omega\mapsto\mathbb{R}^d,\mathbb{P})$ with $\mathscr{L}_X=\mu$. If either $X$ and $Y$ are bounded and $f$ is $L$-differentiable at $\mu$, or $f\in C_b^1(\mathcal{P}_2(\mathbb{R}^d))$, then
		$$\lim_{\varepsilon\rightarrow 0}\frac{f(\mathcal{L}_{X+\varepsilon Y})-f(\mu)}{\varepsilon}=\E\langle D^Lf(\mu)(X),Y\rangle.$$
		Consequently,
		$$\Big|\lim_{\varepsilon\rightarrow 0}\frac{f(\mathcal{L}_{X+\varepsilon Y})-f(\mu)}{\varepsilon}\Big|=|\E\langle D^Lf(\mu)(X),Y\rangle|\leq \|D^Lf(\mu)\|_{T_{\mu,2}}\sqrt{\E|Y|^2}.$$
	\end{lem}

First we estimate  $Z^{\varepsilon}_{t,\xi}=Q^{\varepsilon}_{t,\xi}\vert_{h(\varepsilon)=1}$.
 Lemma \ref{esi-Q} implies that Lemma \ref{zxi} (i).
\begin{lem}\label{zxi}
\begin{enumerate}
	\item[(i)] Under the assumptions \textbf{(H1)}-\textbf{(H3)}, it follows that for $p> 2,\xi\in \mathbb{R}^d$
	  	$$\sup_{t\in [0,T]}\E|Z^{\varepsilon}_{t,\xi}|^{p}\leq C_{p,T,L_1,L_2}(1+|\xi|^{p})\,. $$ 
   
	\item[(ii)] Under the assumptions \textbf{(H1)}-\textbf{(H6)}, it follows that for $p> 2,\xi,\eta\in \mathbb{R}^d$
  	$$\sup_{t\in [0,T]}\E|Z^{\varepsilon}_{t,\xi}-Z^{\varepsilon}_{t,\eta}|^{p}\leq C_{p,T,L_1,L_2,L_3,L_4,L_5}(1+|\xi|^{p})|\xi-\eta|^p\,. $$ 
  
	\item[(iii)] 	Under the assumptions \textbf{(H1)}-\textbf{(H3)}, it follows that for any $t,t'\in [0,T]$, $p> 2,\xi \in \mathbb{R}^d$
  	$$\sup_{t\in [0,T]}\E|Z^{\varepsilon}_{t,\xi}-Z^{\varepsilon}_{t',\xi}|^{p}\leq C_{p,T,L_1,L_2}(1+|\xi|^{p})|t-t'|^{\gamma p},$$ 
 where all the constants  on the above right hand sides are   independent of $\varepsilon$. 
\end{enumerate}
%	\begin{lem}\label{zxieta}
%		Under the assumptions \textbf{(H1)}-\textbf{(H6)}, it follows that for $p> 2,\xi,\eta\in \mathbb{R}^d$
%		$$\sup_{t\in [0,T]}\E|Z^{\varepsilon}_{t,\xi}-Z^{\varepsilon}_{t,\eta}|^{p}\leq C_{p,T,L_1,L_2,L_3,L_4,L_5}(1+|\xi|^{p})|\xi-\eta|^p,$$
%		where the constant $C_{p,T,L_1,L_2,L_3,L_4,L_5}$ is independent of $\varepsilon$.
%	\end{lem}
%	\begin{lem}\label{zxitt}
%		Under the assumptions \textbf{(H1)}-\textbf{(H3)}, it follows that for any $t,t'\in [0,T]$, $p> 2,\xi \in \mathbb{R}^d$
%		$$\sup_{t\in [0,T]}\E|Z^{\varepsilon}_{t,\xi}-Z^{\varepsilon}_{t',\xi}|^{p}\leq C_{p,T,L_1,L_2}(1+|\xi|^{p})|t-t'|^{\gamma p},$$
%		where the constant $C_{p,T,L_1,L_2}$ is independent of $\varepsilon$.	
	\end{lem}
	\begin{proof} We are going to prove (iii) and the case (ii) is similar. 
		For any $t>t'$ we have
		\begin{align}
			&\E|Z^{\varepsilon}_{t,\xi}-Z^{\varepsilon}_{t',\xi}|^{p}\nonumber\\&=\E\Big|\int_{0}^{t}K_1(t,s)\cdot\frac{b(s,X^\varepsilon_{s,\xi},\mathcal{L}_{X^\varepsilon_{s,\xi}})-b(s,X^0_{s,\xi},\delta_{X^0_{s,\xi}})}{\sqrt{\varepsilon}}\d s+\int_{0}^{t}K_2(t,s)\sigma(s,X^\varepsilon_{s,\xi},\mathcal{L}_{X^\varepsilon_{s,\xi}})\d W_s\nonumber\\&\quad-\int_{0}^{t'}K_1(t',s)\cdot\frac{b(s,X^\varepsilon_{s,\xi},\mathcal{L}_{X^\varepsilon_{s,\xi}})-b(s,X^0_{s,\xi},\delta_{X^0_{s,\xi}})}{\sqrt{\varepsilon}}\d s-\int_{0}^{t'}K_2(t',s)\sigma(s,X^\varepsilon_{s,\xi},\mathcal{L}_{X^\varepsilon_{s,\xi}})\d W_s\Big|^{p}\nonumber\\&=\E\Bigg|\int_{0}^{t'}\big(K_1(t,s)-K_1(t',s)\big)\cdot\frac{b(s,X^\varepsilon_{s,\xi},\mathcal{L}_{X^\varepsilon_{s,\xi}})-b(s,X^0_{s,\xi},\delta_{X^0_{s,\xi}})}{\sqrt{\varepsilon}}\d s\nonumber\\&\quad+\int_{0}^{t'}\big(K_2(t,s)-K_2(t',s)\big)\cdot\sigma(s,X^\varepsilon_{s,\xi},\mathcal{L}_{X^\varepsilon_{s,\xi}})\d W_s+\int_{t'}^{t}K_1(t,s)\cdot\frac{b(s,X^\varepsilon_{s,\xi},\mathcal{L}_{X^\varepsilon_{s,\xi}})-b(s,X^0_{s,\xi},\delta_{X^0_{s,\xi}})}{\sqrt{\varepsilon}}\d s\nonumber\\&\quad+\int_{t'}^{t}K_2(t,s)\sigma(s,X^\varepsilon_{s,\xi},\mathcal{L}_{X^\varepsilon_{s,\xi}})\d W_s\Bigg|^p\nonumber\\&\leq 4^{p-1}\E\Big|\int_{0}^{t'}\big(K_1(t,s)-K_1(t',s)\big)\cdot\frac{b(s,X^\varepsilon_{s,\xi},\mathcal{L}_{X^\varepsilon_{s,\xi}})-b(s,X^0_{s,\xi},\delta_{X^0_{s,\xi}})}{\sqrt{\varepsilon}}\d s\Big|^p\nonumber\\&\quad+4^{p-1}\E\Big|\int_{t'}^{t}K_1(t,s)\cdot\frac{b(s,X^\varepsilon_{s,\xi},\mathcal{L}_{X^\varepsilon_{s,\xi}})-b(s,X^0_{s,\xi},\delta_{X^0_{s,\xi}})}{\sqrt{\varepsilon}}\d s\Big|^p\nonumber\\&\quad+4^{p-1}\E\Big|\int_{0}^{t'}\big(K_2(t,s)-K_2(t',s)\big)\cdot\sigma(s,X^\varepsilon_{s,\xi},\mathcal{L}_{X^\varepsilon_{s,\xi}})\d W_s\Big|^p\nonumber\\&\quad+4^{p-1}\E\Big|\int_{t'}^{t}K_2(t,s)\sigma(s,X^\varepsilon_{s,\xi},\mathcal{L}_{X^\varepsilon_{s,\xi}})\d W_s\Big|^{p}:=\Xi_1+\Xi_2.\nonumber
		\end{align}
		For the term $\Xi_1$, by Lemma \ref{basic lemma} (iv), Lemma \ref{zxi} (i), $\mathbf{(H1)}$-$\mathbf{(H3)}$, we have
		\begin{align}
			\Xi_1&\leq C_{p,T}|t-t'|^{\gamma p}\sup_{t\in[0,T]}\E\Big|\frac{b(t,X^\varepsilon_{t,\xi},\mathcal{L}_{X^\varepsilon_{t,\xi}})-b(t,X^0_{t,\xi},\delta_{X^0_{t,\xi}})}{\sqrt{\varepsilon}}\Big|^p\nonumber\\&\leq C_{p,T,L_1}|t-t'|^{\gamma p}\sup_{t\in[0,T]} \E|Z^{\varepsilon}_{t,\xi}|^p\nonumber\\&\leq C_{p,T,L_1,L_2}(1+|\xi|^p)|t-t'|^{\gamma p}.\nonumber
		\end{align}
		For the term $\Xi_2$, applying \ref{basic lemma} (v), Lemma \ref{xxi}, $\mathbf{(H1)}$-$\mathbf{(H3)}$, we have
		\begin{align}
			\Xi_2&\leq 	C_{p,T}|t'-t|^{\gamma p}\sup_{t\in [0,T]}\E\big\|\sigma(s,X^\varepsilon_{s,\xi},\mathcal{L}_{X^\varepsilon_{s,\xi}})\big\|^p\nonumber\\&\leq C_{p,T,L_2}|t'-t|^{\gamma p}\big(1+\sup_{t\in [0,T]}\E|X^\varepsilon_{t,\xi}|^p\big)\nonumber\\&\leq C_{p,T,L_2}(1+|\xi|^{p})|t'-t|^{\gamma p}.\nonumber
		\end{align}
		Hence, for all $0\leq t<t'\leq T$,
		$$\E|Z^{\varepsilon}_{t,\xi}-Z^{\varepsilon}_{t',\xi}|^p\leq C_{p,T,L_1,L_2}(1+|\xi|^p)|t'-t|^{\gamma p}.$$
	\end{proof}

Next, we are going  to take care of bound for  $Z_{t,\xi}$.
	\begin{lem}\label{zxi0}
		Assume  \textbf{(H1)}-\textbf{(H4)}.  Then  for $p> 2,\xi\in \mathbb{R}^d$
		$$\sup_{t\in [0,T]}\E|Z_{t,\xi}|^{p}\leq C_{p,T,L_2,L_3}(1+|\xi|^{p}),$$
		where the constant $C_{p,T,L_2,L_3}$ is independent of $\varepsilon$.
	\end{lem}
	\begin{proof}
		By \eqref{Volterra eq-after CLT of Z}, BDG's inequality, H$\mathrm{\ddot{o}}$lder's inequality, \textbf{(H1)}, \textbf{(H3)} and \textbf{(H4)} we have
		\begin{align}
			\E|Z_{t,\xi}|^{p}&=\E\Big|\int_{0}^{t}K_1(t,s)\cdot\Big[\nabla_{Z_{s,\xi}}b(s,X^0_{s,\xi},\delta_{X^0_{s,\xi}})+\E\big\langle D^L b(s,X^0_{s,\xi},\delta_{X^0_{s,\xi}})(X^0_{s,\xi}),Z_{s,\xi}\big\rangle\Big]\d s\nonumber\\&\quad+\int_{0}^{t}K_2(t,s)\sigma(t,X^0_{s,\xi},\delta_{X^0_{s,\xi}})\d W_s\Big|^{p}\nonumber\\&\leq 2^{p-1}\E\Big|\int_{0}^{t}K_1(t,s)\cdot\Big[\nabla_{Z_{s,\xi}}b(s,X^0_{s,\xi},\delta_{X^0_{s,\xi}})+\E\big\langle D^L b(s,X^0_{s,\xi},\delta_{X^0_{s,\xi}})(X^0_{s,\xi}),Z_{s,\xi}\big\rangle\Big]\d s\Big|^{p}\nonumber\\&\quad+2^{p-1}\E\Big|\int_{0}^{t}K_2(t,s)\sigma(s,X^0_s,\mathcal{L}_{X^0_s})\d W_s\Big|^{p}\nonumber\\&\leq 2^{p-1}\E\int_{0}^{t}K_1(t,s)\cdot\Big|\nabla_{Z_{s,\xi}}b(s,X^0_{s,\xi},\delta_{X^0_{s,\xi}})\nonumber\\&\quad+\E\big\langle D^L b(s,X^0_{s,\xi},\delta_{X^0_{s,\xi}})(X^0_{s,\xi}),Z_{s,\xi}\big\rangle\Big|^{p}\d s\cdot \Big|\int_{0}^{t}K_1(t,s)\d s\Big|^{p-1}\nonumber\\&\quad+C_{p}\E\Big(\int_{0}^{t}K^2_2(t,s)\cdot\big\|\sigma(s,X^0_{s,\xi},\mathcal{L}_{X^0_{s,\xi}})\big\|^{2}\d s\Big)^{\frac{p}{2}}\nonumber\\&\leq C_{p,L_3}\E\int_{0}^{t}K_1(t,s)\cdot\big(|Z_{s,\xi}|^p+\E|Z_{s,\xi}|^p\big)\d s\cdot \Big|\int_{0}^{t}K_1(t,s)\d s\Big|^{p-1}\nonumber\\&\quad+C_{p}\E\int_{0}^{t}K_2^2(t,s)\cdot\big\|\sigma(s,X^0_{s,\xi},\mathcal{L}_{X^0_{s,\xi}})\big\|^{p}\d s\cdot\Big|\int_{0}^{t}K_2^2(t,s)\d s\Big|^{\frac{p}{2}-1}\nonumber\\&\leq C_{p,T,L_3}\int_{0}^{t}K_1(t,s)\cdot \E|Z_{s,\xi}|^{p}\d s+C_{p,T,L_2}\int_{0}^{t}K^2_2(t,s)\cdot\big(1+\E|X^{0}_{s,\xi}|^{p}\big)\d s\nonumber\\&\leq C_{p,T,L_2}\big(1+|\xi|^{p}\big)+C_{p,T,L_3}\int_{0}^{t}K_{1}(t,s)\cdot \E|Z_{s,\xi}|^{p}\d s\nonumber,
		\end{align}
		where in the last inequality  Lemma \ref{x0xi} is used.
		
		By Remark \ref{condition of Volterra}, Lemma \ref{Gronwall} and \eqref{finite of resolvent}, we obtain that for almost all $t\in [0,T]$,
		\begin{align}
			\E|Z_{t,\xi}|^{2p}&\leq C_{p,T,L_2}\big(1+|\xi|^{p}\big)+\int_{0}^{t}R^{K^{*}_{1}}(t,s)\cdot C_{p,T,L_2}\big(1+|\xi|^{p}\big)\d s\nonumber\\&\leq C_{p,T,L_2,L_3}(1+|\xi|^{p}).\nonumber
		\end{align}
	\end{proof}
	\begin{lem}\label{zxieta0}
		Under the assumptions \textbf{(H1)}-\textbf{(H5)},  we have  that for $p> 2,\xi,\eta\in \mathbb{R}^d$
		$$\sup_{t\in [0,T]}\E|Z_{t,\xi}-Z_{t,\eta}|^{p}\leq C_{p,T,L_2,L_3,L_4,L_5}(1+|\xi|^{p})|\xi-\eta|^p,$$
		where the constant $C_{p,T,L_2,L_3,L_4,L_5}$ is independent of $\varepsilon$.
	\end{lem}
	\begin{proof}
		By \eqref{Volterra eq-after CLT of Z}, BDG's inequality, H$\mathrm{\ddot{o}}$lder's inequality, \textbf{(H1)} and \textbf{(H3)}-\textbf{(H5)} we have
		\begin{align}
			&\E|Z_{t,\xi}-Z_{t,\eta}|^{p}\nonumber\\&=\E\Big|\int_{0}^{t}K_1(t,s)\cdot\Big[\nabla_{Z_{s,\xi}}b(s,X^0_{s,\xi},\delta_{X^0_{s,\xi}})+\E\big\langle D^L b(s,X^0_{s,\xi},\delta_{X^0_{s,\xi}})(X^0_{s,\xi}),Z_{s,\xi}\big\rangle\Big]\d s\nonumber\\&\quad+\int_{0}^{t}K_2(t,s)\sigma(t,X^0_{s,\xi},\delta_{X^0_{s,\xi}})\d W_s-\int_{0}^{t}K_2(t,s)\sigma(t,X^0_{s,\eta},\delta_{X^0_{s,\eta}})\d W_s\nonumber\\&\quad-\int_{0}^{t}K_1(t,s)\cdot\Big[\nabla_{Z_{s,\eta}}b(s,X^0_{s,\eta},\delta_{X^0_{s,\eta}})+\E\big\langle D^L b(s,X^0_{s,\eta},\delta_{X^0_{s,\eta}})(X^0_{s,\eta}),Z_{s,\eta}\big\rangle\Big]\d s\Big|^{p}\nonumber\\&\leq 3^{p-1}\E\Big|\int_{0}^{t}K_1(t,s)\cdot\big(\nabla_{Z_{s,\xi}}b(s,X^0_{s,\xi},\delta_{X^0_{s,\xi}})-\nabla_{Z_{s,\eta}}b(s,X^0_{s,\eta},\delta_{X^0_{s,\eta}})\big)\d s\Big|^p\nonumber\\&\quad+3^{p-1}\E\Big|\int_{0}^{t}K_1(t,s)\cdot\big(\E\big\langle D^L b(s,X^0_{s,\xi},\delta_{X^0_{s,\xi}})(X^0_{s,\xi}),Z_{s,\xi}\big\rangle-\E\big\langle D^L b(s,X^0_{s,\eta},\delta_{X^0_{s,\eta}})(X^0_{s,\eta}),Z_{s,\eta}\big\rangle\big)\d s\Big|^p\nonumber\\&\quad+3^{p-1}\E\Big|\int_{0}^{t}K_2(t,s)\cdot\big(\sigma(t,X^0_{s,\xi},\delta_{X^0_{s,\xi}})-\sigma(t,X^0_{s,\eta},\delta_{X^0_{s,\eta}})\big)\d W_s\Big|^p\nonumber\\&\leq 6^{p-1}\E\Big|\int_{0}^{t}K_1(t,s)\cdot\big(\nabla_{Z_{s,\xi}}b(s,X^0_{s,\xi},\delta_{X^0_{s,\xi}})-\nabla_{Z_{s,\xi}}b(s,X^0_{s,\eta},\delta_{X^0_{s,\eta}})\big)\d s\Big|^p\nonumber\\&\quad+6^{p-1}\E\Big|\int_{0}^{t}K_1(t,s)\cdot\big(\nabla_{Z_{s,\xi}}b(s,X^0_{s,\eta},\delta_{X^0_{s,\eta}})-\nabla_{Z_{s,\eta}}b(s,X^0_{s,\eta},\delta_{X^0_{s,\eta}})\big)\d s\Big|^p\nonumber\\&\quad+6^{p-1}\E\Big|\int_{0}^{t}K_1(t,s)\cdot\big(\E\big\langle D^L b(s,X^0_{s,\xi},\delta_{X^0_{s,\xi}})(X^0_{s,\xi}),Z_{s,\xi}\big\rangle-\E\big\langle D^L b(s,X^0_{s,\eta},\delta_{X^0_{s,\eta}})(X^0_{s,\eta}),Z_{s,\xi}\big\rangle\big)\d s\Big|^p\nonumber\\&\quad+6^{p-1}\E\Big|\int_{0}^{t}K_1(t,s)\cdot\big(\E\big\langle D^L b(s,X^0_{s,\eta},\delta_{X^0_{s,\eta}})(X^0_{s,\eta}),Z_{s,\xi}\big\rangle-\E\big\langle D^L b(s,X^0_{s,\eta},\delta_{X^0_{s,\eta}})(X^0_{s,\eta}),Z_{s,\eta}\big\rangle\big)\d s\Big|^p\nonumber\\&\quad+C_p \E\Big(\int_{0}^{t}K^2_2(t,s)\cdot\big\|\sigma(t,X^0_{s,\xi},\delta_{X^0_{s,\xi}})-\sigma(t,X^0_{s,\eta},\delta_{X^0_{s,\eta}})\big\|^2\d s\Big)^{\frac{p}{2}}\nonumber\\&\leq 6^{p-1}\E\int_{0}^{t}K_1(t,s)\cdot\big|\nabla_{Z_{s,\xi}}b(s,X^0_{s,\xi},\delta_{X^0_{s,\xi}})-\nabla_{Z_{s,\xi}}b(s,X^0_{s,\eta},\delta_{X^0_{s,\eta}})\big|^p\d s\cdot \Big(\int_{0}^{t}K_1(t,s)\d s\Big)^{p-1}\nonumber\\&\quad+6^{p-1}\E\int_{0}^{t}K_1(t,s)\cdot\big|\nabla_{Z_{s,\xi}}b(s,X^0_{s,\eta},\delta_{X^0_{s,\eta}})-\nabla_{Z_{s,\eta}}b(s,X^0_{s,\eta},\delta_{X^0_{s,\eta}})\big|^p\d s\cdot \Big(\int_{0}^{t}K_1(t,s)\d s\Big)^{p-1}\nonumber\\&\quad+6^{p-1}\E\int_{0}^{t}K_1(t,s)\cdot\Big|\E\big\langle D^L b(s,X^0_{s,\xi},\delta_{X^0_{s,\xi}})(X^0_{s,\xi})-D^L b(s,X^0_{s,\eta},\delta_{X^0_{s,\eta}})(X^0_{s,\eta}),Z_{s,\xi}\big\rangle\Big|^p\d s\cdot \Big(\int_{0}^{t}K_1(t,s)\d s\Big)^{p-1}\nonumber\\&\quad+6^{p-1}\E\int_{0}^{t}K_1(t,s)\cdot\Big|\E\big\langle D^L b(s,X^0_{s,\eta},\delta_{X^0_{s,\eta}})(X^0_{s,\eta}),Z_{s,\xi}-Z_{s,\eta}\big\rangle\Big|^p\d s\cdot \Big(\int_{0}^{t}K_1(t,s)\d s\Big)^{p-1}\nonumber\\&\quad+C_p\E\int_{0}^{t}K_2^2(t,s)\cdot\big\|\sigma(t,X^0_{s,\xi},\delta_{X^0_{s,\xi}})-\sigma(t,X^0_{s,\eta},\delta_{X^0_{s,\eta}})\big\|^p\cdot \d s\Big(\int_{0}^{t}K_2^2(t,s)\d s\Big)^{\frac{p}{2}-1}\nonumber\\&\leq C_{p,T,L_5}\E\int_{0}^{t}K_1(t,s)\cdot \big(|X^0_{s,\xi}-X^0_{s,\eta}|^p+\mathbb{W}_2^p(\delta_{X^0_{s,\xi}},\delta_{X^0_{s,\eta}})\big)\cdot |Z_{s,\xi}|^p\d s+C_{p,T,L_3}\int_{0}^{t}K_1(t,s)\cdot \E|Z_{s,\xi}-Z_{s,\eta}|^p\d s\nonumber\\&\quad+C_{p,T,L_4}\E\int_{0}^{t}K_1(t,s)\cdot\big(|X^0_{s,\xi}-X^0_{s,\eta}|^p+\mathbb{W}_2^p(\delta_{X^0_{s,\xi}},\delta_{X^0_{s,\eta}})+(\E|X^0_{s,\xi}-X^0_{s,\eta}|^2)^{\frac{p}{2}}\big)\cdot\big(\E|Z_{s,\xi}|^2\big)^{\frac{p}{2}}\d s\nonumber\\&\quad+C_{p,T,L_3}\int_{0}^{t}K_1(t,s)\cdot \E|Z_{s,\xi}-Z_{s,\eta}|^p\d s+C_{p,T,L_2}\E\int_{0}^{t}K_2^2(t,s)\cdot \big(|X^0_{s,\xi}-X^0_{s,\eta}|^p+\mathbb{W}_2^p(\delta_{X^0_{s,\xi}},\delta_{X^0_{s,\eta}})\big)\d s\nonumber\\&\leq C_{p,T,L_3}\int_{0}^{t}K_1(t,s)\cdot \E|Z_{s,\xi}-Z_{s,\eta}|^p\d s+C_{p,T,L_4}\E\int_{0}^{t}K_1(t,s)\cdot \E|X^0_{s,\xi}-X^0_{s,\eta}|^p\d s\nonumber\\&\quad+C_{p,T,L_2}\int_{0}^{t}K_2^2(t,s)\cdot \E|X^0_{s,\xi}-X^0_{s,\eta}|^p\d s+C_{p,T,L_4}\int_{0}^{t}K_1(t,s)\cdot \E|X^0_{s,\xi}-X^0_{s,\eta}|^p\cdot \E|Z_{s,\xi}|^p\d s\nonumber\\&\quad+C_{p,T,L_5}\int_{0}^{t}K_1(t,s)\cdot \big(\E|X^{0}_{s,\xi}-X^{0}_{s,\eta}|^{2p}\big)^{\frac{1}{2}}\cdot \big(\E|Z_{s,\xi}|^{2p}\big)^{\frac{1}{2}}\d s\nonumber\\&\leq C_{p,T,L_3}\int_{0}^{t}K_1(t,s)\cdot \E|Z_{s,\xi}-Z_{s,\eta}|^p\d s+C_{p,T,L_2,L_3,L_4}(1+|\xi|^p)|\xi-\eta|^p,\nonumber
		\end{align}
		where in the last inequality Lemmas \ref{x0xi} and  \ref{zxi0} are used.
		
		Finally, by Remark \ref{condition of Volterra} and Lemma \ref{Gronwall}, we have
		\begin{align}
			\E|Z_{t,\xi}-Z_{t,\eta}|^{p}&\leq C_{p,T,L_2,L_3,L_4,L_5}(1+|\xi|^p)|\xi-\eta|^p+\int_{0}^{t}R^{K^{*}_{1}}(t,s)\cdot\big(C_{p,T,L_2,L_3,L_4,L_5}(1+|\xi|^p)|\xi-\eta|^p\big)\d s\nonumber\\&\leq C_{p,T,L_2,L_3,L_4,L_5}(1+|\xi|^{p})|\xi-\eta|^p.\nonumber
		\end{align}
		This completes the proof of the lemma. 
	\end{proof}
	\begin{lem}\label{zxitt0}
		Under the assumptions \textbf{(H1)}-\textbf{(H4)},  we have that for any $t,t'\in [0,T]$, $p> 2,\xi\in \mathbb{R}^d$
		$$\sup_{t\in [0,T]}\E|Z_{t,\xi}-Z_{t',\xi}|^{p}\leq C_{p,T,L_2,L_3}(1+|\xi|^{p})|t-t'|^{\gamma p},$$
		where the constant $C_{p,T,L_2,L_3}$ is independent of $\varepsilon$.	
	\end{lem}
	\begin{proof}
		For any $t>t'$ we have
		\begin{align}
			&\E|Z_{t,\xi}-Z_{t',\xi}|^{p}\nonumber\\&=\E\Big|\int_{0}^{t}K_1(t,s)\cdot\Big[\nabla_{Z_{s,\xi}}b(s,X^0_{s,\xi},\delta_{X^0_{s,\xi}})+\E\big\langle D^L b(t,X^0_{t,\xi},\delta_{X^0_{t,\xi}})(X^0_{t,\xi}),Z_{t,\xi}\big\rangle\Big]\d s\nonumber\\&\quad-\int_{0}^{t'}K_1(t',s)\cdot\Big[\nabla_{Z_{s,\xi}}b(s,X^0_{s,\xi},\delta_{X^0_{s,\xi}})+\E\big\langle D^L b(t,X^0_{t,\xi},\delta_{X^0_{t,\xi}})(X^0_{t,\xi}),Z_{t,\xi}\big\rangle\Big]ds\nonumber\\&\quad+\int_{0}^{t}K_2(t,s)\sigma(t,X^0_{s,\xi},\delta_{X^0_{s,\xi}})\d W_s-\int_{0}^{t'}K_2(t',s)\sigma(t,X^0_{s,\xi},\delta_{X^0_{s,\xi}})\d W_s\Big|^{p}\nonumber\\&=\E\Bigg|\int_{0}^{t'}\big(K_1(t,s)-K_1(t',s)\big)\cdot\Big[\nabla_{Z_{s,\xi}}b(s,X^0_{s,\xi},\delta_{X^0_{s,\xi}})+\E\big\langle D^L b(t,X^0_{t,\xi},\delta_{X^0_{t,\xi}})(X^0_{t,\xi}),Z_{t,\xi}\big\rangle\Big]\d s\nonumber\\&\quad+\int_{0}^{t'}\big(K_2(t,s)-K_2(t',s)\big)\cdot\sigma(s,X^0_{s,\xi},\mathcal{L}_{X^0_{s,\xi}})\d W_s+\int_{t'}^{t}K_2(t,s)\sigma(s,X^0_{s,\xi},\mathcal{L}_{X^0_{s,\xi}})\d W_s\nonumber\\&\quad+\int_{t'}^{t}K_1(t,s)\cdot\Big[\nabla_{Z_{s,\xi}}b(s,X^0_{s,\xi},\delta_{X^0_{s,\xi}})+\E\big\langle D^L b(t,X^0_{t,\xi},\delta_{X^0_{t,\xi}})(X^0_{t,\xi}),Z_{t,\xi}\big\rangle\Big]\d s\Bigg|^p\nonumber\\&\leq 4^{p-1}\E\Big|\int_{0}^{t'}\big(K_1(t,s)-K_1(t',s)\big)\cdot\Big[\nabla_{Z_{s,\xi}}b(s,X^0_{s,\xi},\delta_{X^0_{s,\xi}})+\E\big\langle D^L b(t,X^0_{t,\xi},\delta_{X^0_{t,\xi}})(X^0_{t,\xi}),Z_{t,\xi}\big\rangle\Big]\d s\Big|^p\nonumber\\&\quad+4^{p-1}\E\Big|\int_{0}^{t'}\big(K_2(t,s)-K_2(t',s)\big)\cdot\sigma(s,X^0_{s,\xi},\mathcal{L}_{X^0_{s,\xi}})\d W_s\Big|^p\nonumber\\&\quad+4^{p-1}\E\Big|\int_{t'}^{t}K_1(t,s)\cdot\Big[\nabla_{Z_{s,\xi}}b(s,X^0_{s,\xi},\delta_{X^0_{s,\xi}})+\E\big\langle D^L b(t,X^0_{t,\xi},\delta_{X^0_{t,\xi}})(X^0_{t,\xi}),Z_{t,\xi}\big\rangle\Big]\d s\Big|^p\nonumber\\&\quad+4^{p-1}\E\Big|\int_{t'}^{t}K_2(t,s)\sigma(s,X^0_{s,\xi},\mathcal{L}_{X^0_{s,\xi}})\d W_s\Big|^{p}:=\Theta_1+\Theta_2+\Theta_3+\Theta_4.\nonumber
		\end{align}
		For the term $\Theta_1$, by H$\mathrm{\ddot{o}}$lder's inequality, Lemma \ref{zxi0}, $\mathbf{(H1)}$, $\mathbf{(H3)}$ and $\mathbf{(H4)}$, we have
		\begin{align}
			\Theta_1&:=4^{p-1}\E\Big|\int_{0}^{t'}\big(K_1(t,s)-K_1(t',s)\big)\cdot\Big[\nabla_{Z_{s,\xi}}b(s,X^0_{s,\xi},\delta_{X^0_{s,\xi}})+\E\big\langle D^L b(t,X^0_{t,\xi},\delta_{X^0_{t,\xi}})(X^0_{t,\xi}),Z_{t,\xi}\big\rangle\Big]\d s\Big|^p\nonumber\\&\leq C_{p,L_3}\E\int_{0}^{t'}\big|K_1(t,s)-K_1(t',s)\big|\cdot \big(\E|Z_{s,\xi}|^p+|Z_{s,\xi}|^p\big)\d s\cdot \Big(\int_{0}^{t'}\big|K_1(t,s)-K_1(t',s)\big|\d s\Big)^{p-1}\nonumber\\&\leq C_{p,T,L_3}\int_{0}^{t'}\big|K_1(t,s)-K_1(t',s)\big|\cdot \E|Z_{s,\xi}|^p\d s|t-t'|^{\gamma(p-1)}\nonumber\\&\leq C_{p,T,L_2,L_3}(1+|\xi|^p)|t-t'|^{\gamma p}.\nonumber
		\end{align}
		For the term $\Theta_2$, applying BDG's inequality, H$\mathrm{\ddot{o}}$lder's inequality, extended Minkowski's inequality, Lemma \ref{xxi}, $\mathbf{(H1)}$ and $\mathbf{(H3)}$, we have
		\begin{align}
			\Theta_2&:=4^{p-1}\E\Big|\int_{0}^{t'}\big(K_2(t,s)-K_2(t',s)\big)\cdot\sigma(s,X^0_{s,\xi},\mathcal{L}_{X^0_{s,\xi}})\d W_s\Big|^p\nonumber\\&\leq C_p\E\Big(\int_{0}^{t'}\big|K_2(t,s)-K_2(t',s)\big|^2\cdot\big\|\sigma(s,X^0_{s,\xi},\mathcal{L}_{X^0_{s,\xi}})\big\|^2\d s\Big)^{\frac{p}{2}}\nonumber\\&\leq C_{p,L_2}\E\Big(\int_{0}^{t'}\big|K_2(t,s)-K_2(t',s)\big|^2\cdot\Big(1+|X^{0}_{s,\xi}|^2+\mathbb{W}_2^2(\mathcal{L}_{X^{0}_{s,\xi}},\delta_0)\Big)\d s\Big)^{\frac{p}{2}}\nonumber\\&\leq C_{p,L_2}\Big(\int_{0}^{t'}\big|K_2(t,s)-K_2(t',s)\big|^2\cdot \big(\E\big[1+|X^{0}_{s,\xi}|^p+\mathbb{W}_2^p(\mathcal{L}_{X^{0}_{s,\xi}},\delta_0)\big]\big)^{\frac{2}{p}}\d s\Big)^{\frac{p}{2}}\nonumber\\&\leq C_{p,L_2}\Big(\int_{0}^{t'}\big|K_2(t,s)-K_2(t',s)\big|^2\cdot \big(1+(\E|X^{0}_{s,\xi}|^p)^{\frac{2}{p}}\big)\d s\Big)^{\frac{p}{2}}\nonumber\\&\leq C_{p,T,L_2}(1+|\xi|^{p})\Big(\int_{0}^{t}\big|K_2(t,s)-K_2(t',s)\big|^2\d s\Big)^{\frac{p}{2}}\leq C_{p,T,L_2}(1+|\xi|^{p})|t'-t|^{\gamma p}.\nonumber
		\end{align}
		For the term $\Theta_3$, by H$\mathrm{\ddot{o}}$lder's inequality, Lemma \ref{zxi0}, $\mathbf{(H1)}$, $\mathbf{(H3)}$ and $\mathbf{(H4)}$, we have
		\begin{align}
			\Theta_3&:=4^{p-1}E\Big|\int_{t'}^{t}K_1(t,s)\cdot\Big[\nabla_{Z_{s,\xi}}b(s,X^0_{s,\xi},\delta_{X^0_{s,\xi}})+\E\big\langle D^L b(t,X^0_{t,\xi},\delta_{X^0_{t,\xi}})(X^0_{t,\xi}),Z_{t,\xi}\big\rangle\Big]\d s\Big|^p\nonumber\\&\leq C_{p,L_3}\E\int_{t'}^{t}K_1(t,s)\cdot \big(\E|Z_{s,\xi}|^p+|Z_{s,\xi}|^p\big)\d s\cdot \Big(\int_{t}^{t'}K_1(t,s)\d s\Big)^{p-1}\nonumber\\&\leq C_{p,T,L_3}\int_{t'}^{t}K_1(t,s)\cdot  \E|Z_{s,\xi}|^p\d s|t-t'|^{\gamma(p-1)}\nonumber\\&\leq C_{p,T,L_2,L_3}(1+|\xi|^p)|t-t'|^{\gamma p}\nonumber.
		\end{align}
		For the term $\Theta_4$, by BDG's inequality, H$\mathrm{\ddot{o}}$lder's inequality, Lemma \ref{xxi}, $\mathbf{(H1)}$ and $\mathbf{(H3)}$, we have
		\begin{align}
			\Theta_4&:=4^{p-1}\E\Big|\int_{t'}^{t}K_2(t,s)\sigma(s,X^0_{s,\xi},\mathcal{L}_{X^0_{s,\xi}})\d W_s\Big|^{p}\nonumber\\&\leq C_p \E\Big(\int_{t'}^{t}K^2_2(t,s)\big\|\sigma(s,X^0_{s,\xi},\mathcal{L}_{X^0_{s,\xi}})\big\|^2\d s\Big)^{\frac{p}{2}}\nonumber\\&\leq  C_{p}\E\int_{t'}^{t}K_2^2(t,s)\cdot\big\|\sigma(s,X^0_{s,\xi},\mathcal{L}_{X^0_{s,\xi}})\big\|^{p}\d s\cdot \Big(\int_{t'}^{t}K_2^2(t,s)\d s\Big)^{\frac{p}{2}-1}\nonumber\\&\leq C_{p,T}\E\int_{t'}^{t}K^2_2(t,s)\cdot\Big(1+|X^0_{s,\xi}|^p+\mathbb{W}_{2}^p(\mathcal{L}_{X^0_{s,\xi}},\delta_0)\Big)\d s|t-t'|^{\gamma(p-2)}\nonumber\\&\leq C_{p,T,L_2}(1+|\xi|^p)|t-t'|^{\gamma p}\,. \nonumber
		\end{align}
		
		Hence, for all $0\leq t<t'\leq T$,
		$$\E|Z_{t,\xi}-Z_{t',\xi}|^p\leq C_{p,T,L_2,L_3}(1+|\xi|^p)|t'-t|^{\gamma p}.$$
		This proves the lemma. 
	\end{proof}
Now we bound  $Z^{\varepsilon}_{t,\xi}-Z_{t,\xi}$.
	\begin{lem}\label{zzxi}
		Suppose that the assumptions \textbf{(H1)}-\textbf{(H6)} are satisfied, then it holds that for $p>2, \xi\in\mathbb{R}^d$
		$$\E|Z^{\varepsilon}_{t,\xi}-Z_{t,\xi}|^p\leq C_{p,T,L_3,L_4}\varepsilon^{\frac{p}{2}}\big(1+|\xi|^{2p}\big),$$
		where the constant $C_{p,T,L_3,L_4}$ is independent of $\varepsilon$.
	\end{lem}
	\begin{proof}
		\begin{align}
			&\E|Z^{\varepsilon}_{t,\xi}-Z_{t,\xi}|^p\nonumber\\
			&=\E\Bigg|\int_{0}^{t}K_1(t,s)\cdot\frac{b(s,X^\varepsilon_{s,\xi},\mathcal{L}_{X^\varepsilon_{s,\xi}})-b(s,X^0_{s,\xi},\delta_{X^0_{s,\xi}})}{\sqrt{\varepsilon}}\d s+\int_{0}^{t}K_2(t,s)\sigma(s,X^\varepsilon_{s,\xi},\mathcal{L}_{X^\varepsilon_{s,\xi}})\d W_s\nonumber\\
			&\quad-\int_{0}^{t}K_1(t,s)\Big[\nabla_{Z_{s,\xi}}b(s,X^0_{s,\xi},\delta_{X^0_{s,\xi}})+ \E  \big\langle D^L b(s,X^0_{s,\xi},\delta_{X^0_{s,\xi}})(X^0_{s,\xi}),Z_{s,\xi}\big\rangle\Big]\d s-\int_{0}^{t}K_2(t,s)\sigma(t,X^0_{s,\xi},\delta_{X^0_{s,\xi}})\d W_s\Bigg|^p\nonumber\\&\leq 3^{p-1}\E\Big|\int_{0}^{t}K_1(t,s)\cdot \big(\frac{b(s,X^{\varepsilon}_{s,\xi},\mathcal{L}_{X^{\varepsilon}_{s,\xi}})-b(s,X^{0}_{s,\xi},\mathcal{L}_{X^{\varepsilon}_{s,\xi}})}{\sqrt{\varepsilon}}-\nabla_{Z_{s,\xi}}b(s,X^0_{s,\xi},\delta_{X^0_{s,\xi}})\big)\d s\Big|^p\nonumber\\&\quad+3^{p-1}\E\Big|\int_{0}^{t}K_1(t,s)\cdot \big(\frac{b(s,X^{0}_{s,\xi},\mathcal{L}_{X^{\varepsilon}_{s,\xi}})-b(s,X^{0}_{s,\xi},\delta_{X^{0}_{s,\xi}})}{\sqrt{\varepsilon}}-\E\big\langle D^L b(s,X^0_{s,\xi},\delta_{X^0_{s,\xi}})(X^0_{s,\xi}),Z_{s,\xi}\big\rangle\big)\Big|^p\d s\nonumber\\&\quad+3^{p-1}\E\Big|\int_{0}^{t}K_2(t,s)\cdot\big(\sigma(s,X^\varepsilon_{s,\xi},\mathcal{L}_{X^\varepsilon_{s,\xi}})-\sigma(t,X^0_{s,\xi},\delta_{X^0_{s,\xi}})\big)\d W_s\Big|^p\nonumber\\&:=\varPi_1+\varPi_2+\varPi_3.\nonumber
		\end{align}
		For the term $\varPi_1$, by H$\mathrm{\ddot{o}}$lder's inequality, Lemmas \ref{xx},   \ref{zxi}, \textbf{(H1)}-\textbf{(H4)} and \textbf{(H6)} we have
		\begin{align}
			\varPi_1&=3^{p-1}\E\Big|\int_{0}^{t}K_1(t,s)\cdot \big(\frac{b(s,X^{\varepsilon}_{s,\xi},\mathcal{L}_{X^{\varepsilon}_{s,\xi}})-b(s,X^{0}_{s,\xi},\mathcal{L}_{X^{\varepsilon}_{s,\xi}})}{\sqrt{\varepsilon}}-\nabla_{Z_{s,\xi}}b(s,X^0_{s,\xi},\delta_{X^0_{s,\xi}})\big)\d s\Big|^p\nonumber\\&\leq 3^{p-1}\E\int_{0}^{t}K_1(t,s)\cdot\Big|\frac{b(s,X^{\varepsilon}_{s,\xi},\mathcal{L}_{X^{\varepsilon}_{s,\xi}})-b(s,X^{0}_{s,\xi},\mathcal{L}_{X^{\varepsilon}_{s,\xi}})}{\sqrt{\varepsilon}}-\nabla_{Z_{s,\xi}}b(s,X^0_{s,\xi},\delta_{X^0_{s,\xi}})\Big|^p\d s\cdot\Big(\int_{0}^{t}K_1(t,s)\d s\Big)^{p-1}\nonumber\\&\leq C_{p,T}\E\int_{0}^{t}K_1(t,s)\cdot\Big|\int_{0}^{1}\nabla_{Z^{\varepsilon}_{s,\xi}}b(s,M_s^{\xi}(u),\mathcal{L}_{X^{\varepsilon}_{s,\xi}})-\nabla_{Z_{s,\xi}}b(s,X^0_{s,\xi},\mathcal{L}_{X^{\varepsilon}_{s,\xi}})\d u\Big|^p\d s\nonumber\\&\leq C_{p,T}\E\int_{0}^{t}K_1(t,s)\cdot\int_{0}^{1}\Big|\nabla_{Z^{\varepsilon}_{s,\xi}}b(s,M_s^{\xi}(u),\mathcal{L}_{X^{\varepsilon}_{s,\xi}})-\nabla_{Z^{\varepsilon}_{s,\xi}}b(s,X^0_{s,\xi},\mathcal{L}_{X^{\varepsilon}_{s,\xi}})\Big|^p\d u\d s\nonumber\\&\quad+C_{p,T}\E\int_{0}^{t}K_1(t,s)\cdot\int_{0}^{1}\Big|\nabla_{Z^{\varepsilon}_{s,\xi}}b(s,X^{0}_{s,\xi},\mathcal{L}_{X^{\varepsilon}_{s,\xi}})-\nabla_{Z_{s,\xi}}b(s,X^0_{s,\xi},\mathcal{L}_{X^{\varepsilon}_{s,\xi}})\Big|^p\d u\d s\nonumber\\&\leq C_{p,T,L_5}\E\int_{0}^{t}K_1(t,s)\cdot\int_{0}^{1}|M_s^{\xi}(u)-X^0_{s,\xi}|^p\d u\cdot |Z^{\varepsilon}_{s,\xi}|^p\d s+C_{p,T,L_3}\int_{0}^{t}K_1(t,s)\cdot \E|Z^{\varepsilon}_{s,\xi}-Z_{s,\xi}|^p\d s\nonumber\\&\leq C_{p,T,L_3}\int_{0}^{t}K_1(t,s)\cdot \E|Z^{\varepsilon}_{s,\xi}-Z_{s,\xi}|^p\d s+C_{p,T,L_5}\int_{0}^{t}K_1(t,s)\cdot \E\big(|X^{\varepsilon}_{s,\xi}-X^0_{s,\xi}|^p\cdot|Z^{\varepsilon}_{s,\xi}|^p\big)\d s\nonumber\\&\leq C_{p,T,L_3}\int_{0}^{t}K_1(t,s)\cdot \E|Z^{\varepsilon}_{s,\xi}-Z_{s,\xi}|^p\d s+C_{p,T,L_5}\int_{0}^{t}K_1(t,s)\cdot \big(\E|Z^{\varepsilon}_{s,\xi}|^{2p}\big)^{\frac{1}{2}}\cdot  \big(\E|X^{\varepsilon}_{s,\xi}-X^{0}_{s,\xi}|^{2p}\big)^{\frac{1}{2}}\d s \nonumber\\&\leq C_{p,T,L_3}\int_{0}^{t}K_1(t,s)\cdot \E|Z^{\varepsilon}_{s,\xi}-Z_{s,\xi}|^p\d s+C_{p,T,L_1,L_2}\varepsilon^{\frac{p}{2}}\big(1+|\xi|^{2p}\big)\,, \nonumber
		\end{align}
		where $M_s^{\xi}(u)=X^0_{s,\xi}+u(X^{\varepsilon}_{s,\xi}-X^0_{s,\xi})$, $u\in [0,T]$.
		
		For the term $\varPi_2$, by H$\mathrm{\ddot{o}}$lder's inequality, Lemmas \ref{xx},   \ref{zxi}, and \textbf{(H1)}-\textbf{(H5)} we have
	 \begin{align}
			\varPi_2&=3^{p-1}\E\Big|\int_{0}^{t}K_1(t,s)\cdot \Big(\frac{b(s,X^{0}_{s,\xi},\mathcal{L}_{X^{\varepsilon}_{s,\xi}})-b(s,X^{0}_{s,\xi},\delta_{X^{0}_{s,\xi}})}{\sqrt{\varepsilon}}-\E\big\langle D^L b(s,X^0_{s,\xi},\delta_{X^0_{s,\xi}})(X^0_{s,\xi}),Z_{s,\xi}\big\rangle\Big)\d s\Big|^p\nonumber\\&\leq 3^{p-1}\E\Big|\int_{0}^{t}K_1(t,s)\cdot \Big(\int_{0}^{1}\big(\E\big\langle D^Lb(s,X^{0}_{s,\xi},\mathcal{L}_{R^{\xi}_s(r)})(R^{\xi}_s(r)),Z^{\varepsilon}_{s,\xi}\big\rangle\nonumber\\&\quad-\E\big\langle D^L b(s,X^0_{s,\xi},\delta_{X^0_{s,\xi}})(X^0_{s,\xi}),Z_{s,\xi}\big\rangle\big)\d r\d s\Big|^p\nonumber\\&\leq 6^{p-1}\E\Big|\int_{0}^{t}K_1(t,s)\cdot \Big(\int_{0}^{1}\big(\E\big\langle D^Lb(s,X^{0}_{s,\xi},\mathcal{L}_{R^{\xi}_s(r)})(R^{\xi}_s(r)),Z^{\varepsilon}_{s,\xi}\big\rangle\nonumber\\&\quad-\E\big\langle D^L b(s,X^0_{s,\xi},\delta_{X^0_{s,\xi}})(X^0_{s,\xi}),Z^{\varepsilon}_{s,\xi}\big\rangle\Big)\d r\d s\Big|^p+6^{p-1}\E\Big|\int_{0}^{t}K_1(t,s)\nonumber\\&\quad \cdot\Big(\int_{0}^{1}\big(\E\big\langle D^Lb(s,X^{0}_{s,\xi},\delta_{X^0_{s,\xi}})(X^0_{s,\xi}),Z^{\varepsilon}_{s,\xi}\big\rangle-\E\big\langle D^L b(s,X^0_{s,\xi},\delta_{X^0_{s,\xi}})(X^0_{s,\xi}),Z_{s,\xi}\big\rangle\Big)\d r\d s\Big|^p\nonumber\\&\leq 6^{p-1}\E\int_{0}^{t}K_1(t,s)\cdot\Big|\int_{0}^{1}\E\big\langle D^Lb(s,X^{0}_{s,\xi},\mathcal{L}_{R^{\xi}_s(r)})(R^{\xi}_s(r)),Z^{\varepsilon}_{s,\xi}\big\rangle-\E\big\langle D^L b(s,X^0_{s,\xi},\delta_{X^0_{s,\xi}})(X^0_{s,\xi}),Z^{\varepsilon}_{s,\xi}\big\rangle \d r\Big|^p\d s\nonumber\\&\cdot \Big|\int_{0}^{t}K_1(t,s)\d s\Big|^{p-1}+6^{p-1}\E\int_{0}^{t}K_1(t,s)\cdot\Big|\int_{0}^{1}\E\big\langle D^Lb(s,X^{0}_{s,\xi},\delta_{X^0_{s,\xi}})(X^0_{s,\xi}),Z^{\varepsilon}_{s,\xi}\big\rangle\nonumber\\&\quad-\E\big\langle D^L b(s,X^0_{s,\xi},\delta_{X^0_{s,\xi}})(X^0_{s,\xi}),Z_{s,\xi}\big\rangle \d r\Big|^p\d s\cdot \Big|\int_{0}^{t}K_1(t,s)\d s\Big|^{p-1}\nonumber\\&\leq C_{p,T,L_4}\E\int_{0}^{t}K_1(t,s)\cdot\int_{0}^{1}\Big|\mathbb{W}_2(\mathcal{L}_{R^{\xi}_s(r)},\delta_{X^0_{s,\xi}})+\big(\E|R^{\xi}_s(r)-X^0_{s,\xi}|^2\big)^{\frac{1}{2}}\Big|^p\d r\cdot \big(\E|Z^{\varepsilon}_{s,\xi}|^2\big)^{\frac{p}{2}}\d s\nonumber\\&\quad+C_{p,T,L_3}\int_{0}^{t}K_1(t,s)\cdot \E|Z^{\varepsilon}_{s,\xi}-Z_{s,\xi}|^p\d s\nonumber\\&\leq C_{p,T,L_4}\int_{0}^{t}K_1(t,s)\cdot \E|X^{\varepsilon}_{s,\xi}-X^0_{s,\xi}|^p\cdot \E|Z^{\varepsilon}_{s,\xi}|^p\d s+C_{p,T,L_3}\int_{0}^{t}K_1(t,s)\cdot \E|Z^{\varepsilon}_{s,\xi}-Z_{s,\xi}|^p\d s\nonumber\\&\leq C_{p,T,L_3}\int_{0}^{t}K_1(t,s)\cdot \E|Z^{\varepsilon}_{s,\xi}-Z_{s,\xi}|^p\d s+C_{p,T,L_1,L_2,L_4}\varepsilon^{\frac{p}{2}}\big(1+|\xi|^{2p}\big)\,, \nonumber
		\end{align}
		where $R^{\xi}_s(r):=X^0_{s,\xi}+r(X^{\varepsilon}_{s,\xi}-X^0_{s,\xi}), t\in [0,T]$.
		
		For the term $\varPi_3$, by BDG's inequality, H$\mathrm{\ddot{o}}$lder's inequality, Lemma \ref{xx} and $\mathbf{(H1)}$-$\mathbf{(H3)}$, we have
		\begin{align}
			\varPi_3&=3^{p-1}\E\Big|\int_{0}^{t}K_2(t,s)\cdot\big(\sigma(s,X^\varepsilon_{s,\xi},\mathcal{L}_{X^\varepsilon_{s,\xi}})-\sigma(t,X^0_{s,\xi},\delta_{X^0_{s,\xi}})\big)\d W_s\Big|^p\nonumber\\&\leq C_p \E\Big(\int_{0}^{t}K^2_2(t,s)\cdot\|\sigma(s,X^\varepsilon_{s,\xi},\mathcal{L}_{X^\varepsilon_{s,\xi}})-\sigma(t,X^0_{s,\xi},\delta_{X^0_{s,\xi}})\|^2\d s\Big)^{\frac{p}{2}}\nonumber\\&\leq C_p \E\int_{0}^{t}K^2_2(t,s)\cdot\|\sigma(s,X^\varepsilon_{s,\xi},\mathcal{L}_{X^\varepsilon_{s,\xi}})-\sigma(t,X^0_{s,\xi},\delta_{X^0_{s,\xi}})\|^p\d s\cdot \Big(\int_{0}^{t}K_2^2(t,s)\d s\Big)^{\frac{p}{2}-1}\nonumber\\&\leq C_{p,T,L_2}\int_{0}^{t}K_2^2(t,s)\cdot \E|X^\varepsilon_{s,\xi}-X^0_{s,\xi}|^p\d s\nonumber\\&\leq C_{p,T,L_1,L_2}\varepsilon^{\frac{p}{2}}\big(1+|\xi|^p\big)\leq C_{p,T,L_1,L_2}\varepsilon^{\frac{p}{2}}\big(1+|\xi|^{2p}\big).\nonumber
		\end{align}
		So by Remark \ref{condition of Volterra} and Lemma \ref{Gronwall}, we have
		\begin{align}
			\E|Z^{\varepsilon}_{t,\xi}-Z_{t,\xi}|^p&\leq C_{p,T,L_3}\int_{0}^{t}K_1(t,s)\cdot \E|Z^{\varepsilon}_{s,\xi}-Z_{s,\xi}|^p\d s+C_{p,T,L_1,L_2,L_4}\varepsilon^{\frac{p}{2}}\big(1+|\xi|^{2p}\big)\nonumber\\&\leq C_{p,T,L_1,L_2,L_3,L_4}\varepsilon^{\frac{p}{2}}\big(1+|\xi|^{2p}\big)\nonumber.
		\end{align}
	\end{proof}
We are now ready to prove Theorem \ref{CLT}. 
	\begin{proof}[Proof of Theorem \ref{CLT}] 
		 We define 
		\begin{equation}
			Z(r,t,\xi)=\begin{cases}
				Z_{t,\xi}, \ r=0,\nonumber\\
				Z^{\varepsilon_{n}}_{t,\xi}+\frac{r-\varepsilon_n T}{\varepsilon_{n+1}}\big(Z^{\varepsilon_{n+1}}_{t,\xi}-Z^{\varepsilon_{n}}_{t,\xi}\big), \ \varepsilon_{n+1}T<r\leq \varepsilon_n T,\ n\in\mathbb{N}. \nonumber
			\end{cases}
		\end{equation}
		Fix $R>0$, by Lemmas \ref{zxi}[(ii)], \ref{zxi}[(iii)], \ref{zxieta0}, \ref{zxitt0}, \ref{zzxi}, there exists a constant $\beta:=\beta(\gamma)>0$ such that for all $p>2$, $t,t'\in [0,T]$, $\xi,\eta\in D_R:=\{\xi\in\mathbb{R}^d,|\xi|\leq R\}$, we have
		$$\E\Big|Z(r,t,\xi)-Z(r',t',\eta)\Big|^p\leq C_{p,T,R,L_3,L_4}\Big(|r-r'|^{\beta p}+|t-t'|^{\beta p}+|\xi-\eta|^{\beta p}\Big).$$
		Thus, by Kolmogorov’s continuity criterion, there is a $p$-order integrable random variable $A(\omega)$ such that
		$$\sup_{t\in [0,T],|\xi|\leq R}\Big|Z(r,t,\xi)-Z(r',t,\xi)\Big|\leq A(\omega)|r-r'|^{\lambda},\quad a.s.,$$
		where $\lambda\in (0,\beta-\frac{d+2}{p})$. In particular, taking $r=0,r'=\varepsilon_{n}T$ we have
		$$\E\Big(\sup_{t\in [0,T],|\xi|\leq R}\big|Z^{\varepsilon_{n}}_{t,\xi}-Z_{t,\xi}\big|\Big)^p\leq |\varepsilon_{n}T|^{p\lambda}\E|A(\omega)|^p,$$
		which yields the desired convergence.
	\end{proof}
	\section{The moderate deviation principle for Volterra type McKean-Vlasov SDEs}\label{MDP-1}
	In this section, we shall  study moderate deviation principle for our Volterra type McKean-Vlasov SDEs.
	
	For $\varepsilon>0$, define 
	\begin{equation}
	 Y^{\varepsilon}_{\cdot}:=\frac{X^{\varepsilon}_{\cdot,\xi}-X^{0}_{\cdot,\xi}}{\sqrt{\varepsilon}h(\varepsilon)}=\frac{\mathcal{G}^{\varepsilon}(W)-X^{0}_{\cdot,\xi}}{\sqrt{\varepsilon}h(\varepsilon)}:=\mathcal{J}^{\varepsilon}(W)\,.
	\end{equation}
Assume  that  $h(\varepsilon)$ satisfies
	$$h(\varepsilon)\to \infty, \ \sqrt{\varepsilon}h(\varepsilon)\to 0 \ \text{as}\ \varepsilon\to 0.$$
		This $\mathcal{J}^{\varepsilon}: C_T^m\to C_T^d$ is a Borel measurable map for each $\varepsilon>0$, and the
  $Y^{\varepsilon}_{\cdot}$ satisfies the following Volterra type  McKean-Vlasov SDEs:
	\begin{align}\label{Volterra eq Y of MDP}
		Y^\varepsilon_{t}&=\int_{0}^{t}K_1(t,s)\cdot\frac{b(s,X^0_{s,\xi}+\sqrt{\varepsilon}h(\varepsilon)Y^\varepsilon_{s,\xi},\mathcal{L}_{X^{\varepsilon}_{s,\xi}})-b(s,X^0_{s,\xi},\delta_{X^0_{s,\xi}})}{\sqrt{\varepsilon}h(\varepsilon)}\d s\nonumber\\&\quad+\frac{1}{h(\varepsilon)}\int_{0}^{t}K_2(t,s)\sigma(s,X^0_{s,\xi}+\sqrt{\varepsilon}h(\varepsilon)Y^\varepsilon_{s},\mathcal{L}_{X^{\varepsilon}_{s,\xi}})\d W_s\,. 
	\end{align}
	It is clear that the existence and unique of $Y^\varepsilon_{\cdot}$ hold by Theorem $3.1$ of \cite{LG}, then the moderate deviation principle for $\{X^{\varepsilon}\}_{\varepsilon>0}$ is equivalent to the large deviation principle for $\{Y^{\varepsilon}\}_{\varepsilon>0}$.
	
	Applying Girsanov's theorem,  the shifted version $Y^{\varepsilon,v}_{\cdot}:=\mathcal{J}^{\varepsilon}\big(W+h(\varepsilon)\int_{0}^{\cdot}v_sds\big)$ is   the unique strong solution of the controlled equation under $\mathbb{P}$:
 	\begin{align}\label{controlled Volterra eq-MDP}
		Y^{\varepsilon,v}_{t}&=\int_{0}^{t}K_1(t,s)\cdot\frac{b(s,X^0_{s,\xi}+\sqrt{\varepsilon}h(\varepsilon)Y^{\varepsilon,v}_{s},\mathcal{L}_{X^{\varepsilon}_{s,\xi}})-b(s,X^0_{s,\xi},\delta_{X^0_{s,\xi}})}{\sqrt{\varepsilon}h(\varepsilon)}\d s\nonumber\\&\quad+\int_{0}^{t}K_2(t,s)\sigma(s,X^0_{s,\xi}+\sqrt{\varepsilon}h(\varepsilon)Y^{\varepsilon,v}_{s},\mathcal{L}_{X^{\varepsilon}_{s,\xi}})v_s\d s\\&\quad+\frac{1}{h(\varepsilon)}\int_{0}^{t}K_2(t,s)\sigma(s,X^0_{s,\xi}+\sqrt{\varepsilon}h(\varepsilon)Y^{\varepsilon,v}_{s},\mathcal{L}_{X^{\varepsilon}_{s,\xi}})\d W_s.\nonumber
	\end{align}
	For all $v\in\mathcal{A}$, we define $\mathcal{J}^0_v$ to be the solution of the limiting Volterra integral equation
	\begin{align}\label{controlled determinstic Volterra eq-MDP}
		\psi_t=\int_{0}^{t}K_1(t,s)\nabla_{\psi_s}b(s,X^0_{s,\xi},\delta_{X^0_{s,\xi}})\d s+\int_{0}^{t}K_2(t,s)\sigma(s,X^0_{s,\xi},\delta_{X^0_{s,\xi}})v_s\d s.
	\end{align}
	To obtain the MDP, in addition to \textbf{(H1)}-\textbf{(H3)} and \textbf{(H6)}, we also need the following assumption. 
	
	\textbf{(H7)} There exists a positive constant $L_6$ such that for $t\in [0,T]$
	$$\nabla b(t,0,\delta_0)\leq L_6.$$

	Define   the new rate function as
	\begin{align}\label{def of G}
		\Lambda(\psi):=\inf\Big\{\frac{1}{2}\int_{0}^{T}|v_s|^2\d s: v\in L^2 \ \text{such that}\ \psi\in\mathcal{J}^0_v\Big\}.
	\end{align}
	We now state our  results of MDP.
	\begin{thm}\label{MDP}
		Under \textbf{(H1)}-\textbf{(H3)} and \textbf{(H6)}-\textbf{(H7)}, the family $\{X^{\varepsilon}\}_{\varepsilon>0}$,  the unique solution of \eqref{Volterra MV eq}, satisfies the  Moderate Deviations Principle with rate function \eqref{def of G} and speed $h^2(\varepsilon)$, where $\mathcal{J}^0_v$ corresponds to the unique solution of the limiting equation \eqref{controlled determinstic Volterra eq-MDP}.
	\end{thm} 
	
 We are going to use Lemma \ref{general LDP}  and will then verify the two conditions there.
 
\noindent {\bf Step  1}\ Verification of the Condition $(i)$ of Lemma \ref{general LDP}. 
%	\footnote{Which  theorem you use to prove the above theorem? 
%	Condition	(i) of what?} 
		\begin{lem}\label{existence-mid}
		Assume that \textbf{(H1)}-\textbf{(H3)} and \textbf{(H6)}-\textbf{(H7)} hold.  Then there exists a unique solution $\psi_t$ such that for almost all $t\geq0$, $N>0,v\in\mathcal{A}_N, \xi\in \mathbb{R}^d$, some constant $C_{T,N,L_2,L_5,L_6}$
		\begin{align}\label{moment estimate of p-MDP}
			\sup_{t\in [0,T]}|\psi_t|\leq C_{T,N,L_2,L_5,L_6}(1+|\xi|),
		\end{align}
		for almost all $t\in [0,T]$.
	\end{lem}
	From Lemma \ref{esi-Q}, we easily get the following estimate associated with $Y^{\varepsilon}:=Q_t^{\varepsilon}=\frac{X^{\varepsilon}-X^{0}}{\sqrt{\varepsilon}h(\varepsilon)}$.
	\begin{lem}\label{yy}
		Under the assumptions \textbf{(H1)}-\textbf{(H3)}, it follows that for $p> 2,\xi\in \mathbb{R}^d$
		$$\sup_{t\in [0,T]}\E|Y^{\varepsilon}_{t}|^{p}\leq \frac{1}{|h(\varepsilon)|^p}C_{p,T,L_1,L_2}\big(1+|\xi|^p\big),$$
		where the constant $C_{p,T,L_1,L_2}$ is independent of $\varepsilon$.
	\end{lem}
Let
\begin{align}
&x:=0,\Lambda^{\varepsilon,v}_{t,\xi}:=Y^{\varepsilon,v}_{t},N(s,X^{0}_{s,\xi},\mathcal{L}_{X_{s,\xi}^{\varepsilon}},\Lambda^{\varepsilon,v}_{s,x},h(\varepsilon)):=\sigma(s,X^0_{s,\xi}+\sqrt{\varepsilon}h(\varepsilon)Y^{\varepsilon,v}_{s},\mathcal{L}_{X^{\varepsilon}_{s,\xi}}),\nonumber\\&M(s,X^{0}_{s,\xi},\mathcal{L}_{X_{s,\xi}^{\varepsilon}},\delta_{X_{s,\xi}^{0}},\Lambda^{\varepsilon,v}_{s,x},h(\varepsilon)):= \frac{b(s,X^0_{s,\xi}+\sqrt{\varepsilon}h(\varepsilon)Y^{\varepsilon,v}_{s},\mathcal{L}_{X^{\varepsilon}_{s,\xi}})-b(s,X^0_{s,\xi},\delta_{X^0_{s,\xi}})}{\sqrt{\varepsilon}h(\varepsilon)},\nonumber\\	&G(s,X^{0}_{s,\xi},\mathcal{L}_{X_{s,\xi}^{\varepsilon}},\Lambda^{\varepsilon,v}_{s,x},h(\varepsilon)):=\frac{1}{h(\varepsilon)}\sigma(s,X^0_{s,\xi}+\sqrt{\varepsilon}h(\varepsilon)Y^{\varepsilon,v}_{s},\mathcal{L}_{X^{\varepsilon}_{s,\xi}}).\nonumber
\end{align}
Then Lemma \ref{estimate of Lambda} and Lemma \ref{Lambda-ts} give us that
	\begin{lem}\label{estimate of Y v}
		Under the assumptions \textbf{(H1)}-\textbf{(H3)},   for $p> 2, N>0,v\in\mathcal{A}_N, \xi\in \mathbb{R}^d$
		$$\sup_{t\in [0,T]}\E|Y^{\varepsilon,v}_{t}|^{p}\leq C_{p,T,N,L_1,L_2}(1+|\xi|^{p}),$$
		where the constant $C_{p,T,N,L_1,L_2}$ is independent of $\varepsilon,v$.
	\end{lem}
\begin{lem}\label{Y vare v t s}
	Assume that \textbf{(H1)}-\textbf{(H3)} hold. If $p>2\vee \frac{2}{\gamma}$, $t,s\in [0,T]$, $N>0, \xi\in \mathbb{R}^d$ and $\{v^{\varepsilon}\}_{\varepsilon>0}$ is a family in $\mathcal{A}_N$, then $Y^{\varepsilon,v^{\varepsilon}}$ admits a version which is H$\mathrm{\ddot{o}}$lder continuous on $[0,T]$ of any order $\alpha<\frac{\gamma}{2}-\frac{1}{p}$, uniformly for all $\varepsilon>0$. Denoting again this version by $Y^{\varepsilon,v^{\varepsilon}}$, one has for all $\varepsilon>0$ small enough
	
	$$\E\Big[\Big(\sup_{0\leq s<t\leq T}\frac{Y^{\varepsilon,v^{\varepsilon}}_{t}-Y^{\varepsilon,v^{\varepsilon}}_s}{|t-s|^{\alpha}}\Big)^p\Big]\leq C_{p,T,N,L_1,L_2,\xi},$$
	for all $\a\in [0,\frac{\gamma}{2}-\frac{1}{p})$. Moreover, the family of random variables $\{Y^{\varepsilon,v^{\varepsilon}}\}_{\varepsilon>0}$ is tight in $C_T$.
\end{lem}
By almost same procedures, we can get that the solution $\psi_{t}$ of \eqref{controlled determinstic Volterra eq-MDP} also have H$\mathrm{\ddot{o}}$lder continuous path, that is
\begin{lem}\label{11}
	Assume that \textbf{(H1)}-\textbf{(H3)} and \textbf{(H6)} hold, for $t,s\in [0,T]$, $N>0, v\in\mathcal{A}_N, \xi\in \mathbb{R}^d$, then the solution $\psi_{t}$ of \eqref{controlled determinstic Volterra eq-MDP} has for all $\varepsilon>0$ small enough
	$$|\psi_{t}-\psi_s|\leq C_{T,N,L_2}\big(1+|\xi|^p\big)|t-s|^{\gamma}.$$
\end{lem}
This verifies  Condition (i) of Lemma \ref{general LDP}.

\noindent {\bf Step  2}\  Verification of the Condition $(ii)$ of Lemma \ref{general LDP}.
	\begin{lem}\label{unique characterised-MDP}
		The set $\mathcal{J}^0_v$ from Definition \ref{v} ($\mathcal{G}$ is replaced by $\mathcal{J}$ and $\varepsilon_{n}=\frac{1}{h(\varepsilon_n)}$) is characterized by
		$$\mathcal{J}^0_v=\Big\{\psi: [0,T]\to\mathbb{R}^d|\quad	\psi_t=\int_{0}^{t}K_1(t,s)\nabla_{\psi_s}b(s,X^0_{s,\xi},\delta_{X^0_{s,\xi}})\d s+\int_{0}^{t}K_1(t,s)\sigma(s,X^0_{s,\xi},\delta_{X^0_{s,\xi}})v_s\d s\Big\}.$$
	\end{lem}
	\begin{proof}
		Recall the definition \ref{v} of $\mathcal{J}^0_{v,N}$. For $N\in\mathbb{N}$ and $v\in\mathcal{A}_N$, consider a subsequence $\{\varepsilon_n\}_{n\in\mathbb{N}}\subset[0,\infty)$ with $\lim_{n\to \infty}\varepsilon_{n}=0$ and a sequence $\{v^{\varepsilon_n}\}_{n\in\mathbb{N}}\subset\mathcal{A}_N$ such that $\lim_{n\to \infty}v^{\varepsilon_{n}}=v\ \text{in distribution}$, and assume that $Y^{n,v}:=Y^{\varepsilon_n,v^{\varepsilon_{n}}}$ converges in distribution to some random variable $\psi$ with values in $C^d_T$. To simplify the representation, we denote  $v^{n}:=v^{\varepsilon_{n}}$.
		
		By Skorohod representation theorem we can work with almost sure convergence for the purpose of identifying the limit. Hence, we may assume that $\{Y^{n,v},v^n\}_{n\geq0}$ converges almost surely in the product
		topology on $C_T^d\times \mathcal{S}_N$, and the limit is a  $C_T^d\times \mathcal{S}_N$-valued random variable $(\psi, v)$ in a possibly different probability space $(\Omega^0,\mathcal{F}^0,\mathbb{P}^0)$ as $n$ tends to $+\infty$.
		% {\red The convergence of the couple also takes place in distribution. }   
		For $t\in [0,T]$, define $\Psi_t:C_T^d\times \mathcal{S}_N\to\mathbb{R}$ as
		$$\Psi_t(\omega,f):=\Big|\omega_t-\int_{0}^{t}K_1(t,s)\nabla_{\omega_s}b(s,X^0_{s,\xi},\delta_{X^0_{s,\xi}})\d s-\int_{0}^{t}K_2(t,s)\sigma(s,X^0_{s,\xi},\delta_{X^0_{s,\xi}})f_s\d s\Big|\wedge 1.$$
		It is clear that $\Psi_t$ is bounded, and we  are going to   show that $\Psi_t$ is continuous. To this end, let $\omega^n\to \omega$ in $C_T^d$ and $f^n\to f$ in $\mathcal{S}_N$ with respect the weak topology. 
		
		For $\omega,\omega^n\in C^d_T,f,f^n\in\mathcal{S}_N$, by \textbf{(H1)-(H3)}, \textbf{(H6)}-\textbf{(H7)} and Lemma \ref{x0xi} we have that
		\begin{align}
			&|\Psi_t(\omega,f)-\Psi_t(\omega^n,f^n)|\nonumber\\&\leq |\omega_t-\omega^n_t|+\int_{0}^{t}K_1(t,s)\cdot \big\|\nabla_{\omega_s}b(s,X^0_{s,\xi},\delta_{X^0_{s,\xi}})-\nabla_{\omega^n_s}b(s,X^0_{s,\xi},\delta_{X^0_{s,\xi}})\big\|\d s\nonumber\\&\quad+\int_{0}^{t}K_2(t,s)\cdot\Big|\sigma(s,X^0_{s,\xi},\delta_{X^0_{s,\xi}})f_s-\sigma(s,X^0_{s,\xi},\delta_{X^0_{s,\xi}})f^n_s\Big|\d s\nonumber\\&\leq \sup_{t\in [0,T]}|\omega_t-\omega^n_t|+\int_{0}^{t}K_1(t,s)\cdot |\omega_s-\omega^n_s|\cdot\big\|\nabla b(s,X^0_{s,\xi},\delta_{X^0_{s,\xi}})-\nabla b(s,0,\delta_{0})\big\|\d s\nonumber\\&\quad+\int_{0}^{t}K_1(t,s)\cdot |\omega_s-\omega^n_s|\cdot\big\|\nabla b(s,0,\delta_{0})\big\|\d s+\int_{0}^{t}K_2(t,s)\cdot |f_s-f^n_s|\cdot\big\|\sigma(s,X^0_{s,\xi},\delta_{X^0_{s,\xi}})\big\|\d s\nonumber\\&\leq \sup_{t\in [0,T]}|\omega_t-\omega^n_t|\cdot\Big[1+L_5\int_{0}^{t}K_1(t,s)\cdot \big(|X^0_{s,\xi}|+\mathbb{W}_2(\delta_{X^0_{s,\xi}},\delta_0\big)\d s\Big]\nonumber\\&\quad+C_{T,L_6}\sup_{t\in [0,T]}|\omega_t-\omega^n_t|+C_{L_2}\int_{0}^{t}K_2(t,s)\cdot |f_s-f^n_s|\cdot\big(1+|X^0_{s,\xi}|+\mathbb{W}_2(\delta_{X^0_{s,\xi}},\delta_{0})\big)\d s\nonumber\\&\leq C_{T,L_2,L_5,L_6}\sup_{t\in [0,T]}|\omega_t-\omega^n_t|+C_{L_2}\int_{0}^{t}K_2(t,s)\cdot |f_s-f^n_s|\d s\nonumber.
		\end{align}
		Since $f^n$ tends to $f$ weakly in $L^2$ and $\lim_{n\to \infty}|\omega_t-\omega^n_t|=0$,   we see that  $\Psi_t(f^n,\omega^n)$ tends to $\Psi_t(f,\omega)$ as $n$ tends to $\infty$, which implies $\Psi_t$ is continuous and therefore
		$$\lim_{n\to \infty}\E[\Psi_t(Y^{n,v},v^n)]=\E[\Psi_t(\psi,v)].$$
	Now it suffices for us  to prove that the right hand of above equality is actually equal to zero. 
		
		Since
		\begin{align}
			\E[\Psi_t(Y^{n,v},v^n)]&=\E\Big[\Big|\int_{0}^{t}K_1(t,s)\cdot\frac{b(s,X^0_{s,\xi}+\sqrt{\varepsilon_n}h(\varepsilon_n)Y^{n,v}_{s},\mathcal{L}_{X^{n}_s})-b(s,X^0_{s,\xi},\delta_{X^0_{s,\xi}})}{\sqrt{\varepsilon_n}h(\varepsilon_n)}\d s\nonumber\\&\quad+\int_{0}^{t}K_2(t,s)\cdot\sigma(s,X^0_{s,\xi}+\sqrt{\varepsilon_n}h(\varepsilon_n)Y^{n,v}_{s},\mathcal{L}_{X^{n}_s})v^n_s\d s\nonumber\\&\quad+\frac{1}{h(\varepsilon_n)}\int_{0}^{t}K_2(t,s)\cdot\sigma(s,X^0_{s,\xi}+\sqrt{\varepsilon_n}h(\varepsilon_n)Y^{n,v}_{s},\mathcal{L}_{X^{n}_s})\d W_s\nonumber\\&\quad-\int_{0}^{t}K_1(t,s)\nabla_{Y^n_s}b(s,X^0_{s,\xi},\delta_{X^0_{s,\xi}})\d s-\int_{0}^{t}K_2(t,s)\sigma(s,X^0_{s,\xi},\delta_{X^0_{s,\xi}})v^n_s\d s\Big|\wedge 1\Big]\nonumber\\&\leq\mathcal{V}_1+\mathcal{V}_2+\mathcal{V}_3,\nonumber
		\end{align}
		where
		\begin{align}
			\mathcal{V}_1&:=\E\Big|\int_{0}^{t}K_1(t,s)\cdot\Big[\frac{b(s,X^0_{s,\xi}+\sqrt{\varepsilon_n}h(\varepsilon_n)Y^{n,v}_{s},\mathcal{L}_{X^{n}_s})-b(s,X^0_{s,\xi},\delta_{X^0_{s,\xi}})}{\sqrt{\varepsilon_n}h(\varepsilon_n)}-\nabla_{Y^{n,v}_s}b(s,X^0_{s,\xi},\delta_{X^0_{s,\xi}})\Big]\d s\Big|,\nonumber\\\mathcal{V}_2&:=\E\Big|\int_{0}^{t}K_2(t,s)\cdot\Big[\sigma(s,X^0_{s,\xi}+\sqrt{\varepsilon_n}h(\varepsilon_n)Y^{n,v}_{s},\mathcal{L}_{X^{n}_s})-\sigma(s,X^0_{s,\xi},\delta_{X^0_{s,\xi}})\Big]\cdot v^n_s\d s\Big|,\nonumber\\\mathcal{V}_3&:=\E\Big|\frac{1}{h(\varepsilon_n)}\int_{0}^{t}K_2(t,s)\cdot\sigma(s,X^0_{s,\xi}+\sqrt{\varepsilon_n}h(\varepsilon_n)Y^{n,v}_{s},\mathcal{L}_{X^{n}_s})\d W_s\Big|\nonumber.
		\end{align}
		
		For the term $\mathcal{V}_1$, by Lemmas \ref{yy},  \ref{estimate of Y v}, H$\mathrm{\ddot{o}}$lder's inequality, \textbf{(H1)}-\textbf{(H3)} and \textbf{(H6)}, we have
		\begin{align}
			\mathcal{V}_1&\leq \E\int_{0}^{t}K_1(t,s)\cdot\Big|\frac{b(s,X^0_{s,\xi}+\sqrt{\varepsilon_n}h(\varepsilon_n)Y^{n,v}_{s},\mathcal{L}_{X^{n}_s})-b(s,X^0_{s,\xi},\delta_{X^0_{s,\xi}})}{\sqrt{\varepsilon_n}h(\varepsilon_n)}-\nabla_{Y^{n,v}_s}b(s,X^0_{s,\xi},\delta_{X^0_{s,\xi}})\Big|\d s\nonumber\\&\leq \E \int_{0}^{t}K_1(t,s)\cdot\Big|\frac{b(s,X^0_{s,\xi}+\sqrt{\varepsilon_n}h(\varepsilon_n)Y^{n,v}_{s},\mathcal{L}_{X^{n}_s})-b(s,X^0_{s,\xi}+\sqrt{\varepsilon_n}h(\varepsilon_n)Y^{n,v}_{s},\mathcal{L}_{X^0_{s,\xi}})}{\sqrt{\varepsilon_n}h(\varepsilon_n)}\Big|\d s\nonumber\\&\quad+\E\int_{0}^{t}K_1(t,s)\cdot\Big|\frac{b(s,X^0_{s,\xi}+\sqrt{\varepsilon_n}h(\varepsilon_n)Y^{n,v}_{s},\mathcal{L}_{X^0_{s,\xi}})-b(s,X^0_{s,\xi},\delta_{X^0_{s,\xi}})}{\sqrt{\varepsilon_n}h(\varepsilon_n)}-\nabla_{Y^{n,v}_s}b(s,X^0_{s,\xi},\delta_{X^0_{s,\xi}})\Big|\d s\nonumber\\&\leq L_1 \E\int_{0}^{t}K_1(t,s)\cdot\Big|\frac{\mathbb{W}_2(\mathcal{L}_{X^{n}_s},\delta_{X^0_{s,\xi}})}{\sqrt{\varepsilon_{n}}h(\varepsilon_{n})}\Big|\d s\nonumber \\&\quad+\E\int_{0}^{t}K_1(t,s)\cdot\Big|\int_{0}^{1}\nabla_{Y^{n,v}_s}b(s,P^n_s(u),\mathcal{L}_{X^0_{s,\xi}})-\nabla_{Y^{n,v}_s}b(s,X_{s,\xi}^0,\delta_{X^0_{s,\xi}})\d u\Big|\d s\nonumber\\&\leq L_1\int_{0}^{t}K_1(t,s)\cdot \big(\E|Y^n_s|^p\big)^{\frac{1}{p}}\d s+L_5\sqrt{\varepsilon_n}h(\varepsilon_n)\int_{0}^{t}K_1(t,s)\cdot \big(\E|Y^{n,v}_s|^p\big)^{\frac{1}{p}}\d s\nonumber\\&\leq C_{T,L_1,L_2,\xi}\frac{1}{h(\varepsilon_{n})}+C_{T,L_1,L_2,L_5,\xi}\sqrt{\varepsilon_{n}}h(\varepsilon_{n}),\nonumber
		\end{align}
		where $P^{n}_s(u):=P^{\varepsilon_{n}}_s(u):=X^0_{s,\xi}+u\sqrt{\varepsilon_n}h(\varepsilon_n)Y^{n,v}_{s}$,$u\in [0,T]$.
		
		For the term $\mathcal{V}_2$, by Lemma \ref{xx}, Lemma \ref{estimate of Y v}, H$\mathrm{\ddot{o}}$lder's inequality, Minkowski's inequality and \textbf{(H1)}-\textbf{(H3)}, we have
		\begin{align}
			\mathcal{V}_2&\leq \sqrt{N}\E \Big(\int_{0}^{t}K^2_2(t,s)\cdot\big\|\sigma(s,X^0_{s,\xi}+\sqrt{\varepsilon_n}h(\varepsilon_n)Y^{n,v}_{s},\mathcal{L}_{X^{n}_s})-\sigma(s,X^0_{s,\xi},\delta_{X^0_{s,\xi}})\big\|^2\d s\Big)^{\frac{1}{2}}\nonumber\\&\leq \sqrt{NL_1}\Big[\E \Big(\int_{0}^{t}K^2_2(t,s)\cdot\big(\varepsilon_{n}|h(\varepsilon_{n})|^2|Y^{n,v}_s|^2+\mathbb{W}_2^2(\mathcal{L}_{X^n_s},\delta_{X^0_{s,\xi}})\big)\d s\Big)^{\frac{p}{2}}\Big]^{\frac{1}{p}}\nonumber\\&\leq \sqrt{NL_1} \Big(\int_{0}^{t}K^2_2(t,s)\cdot \big[\E\big(|\sqrt{\varepsilon_{n}}h(\varepsilon_{n})|^p|Y^{n,v}_s|^p+\mathbb{W}^p_2(\mathcal{L}_{X^n_s},\delta_{X^0_{s,\xi}})\big)\big]^{\frac{2}{p}}\d s\Big)^{\frac{1}{2}}\nonumber\\&\leq \sqrt{NL_1} \Big(\int_{0}^{t}K^2_2(t,s)\cdot \big[\sqrt{\varepsilon_{n}}h(\varepsilon_{n})|^p\E|Y^{n,v}_s|^p+\E|X^n_s-X^0_{s,\xi}|^p\big]^{\frac{2}{p}}\d s\Big)^{\frac{1}{2}}\nonumber\\&\leq C_{T,N,L_1,L_2,\xi}\sqrt{\varepsilon_{n}}h(\varepsilon_{n})+C_{T,N,L_1,L_2,\xi}\sqrt{\varepsilon_{n}}.\nonumber
		\end{align}
		
		For the term $\mathcal{V}_3$, by  Lemmas \ref{xxi},  \ref{x0xi},  and  \ref{estimate of Y v}, H$\mathrm{\ddot{o}}$lder's inequality, Minkowski's inequality and \textbf{(H1)}-\textbf{(H3)}, we have
		\begin{align}
			\mathcal{V}_3&\leq \frac{1}{h(\varepsilon_n)} \E\Big(\int_{0}^{t}K^2_2(t,s)\cdot\big\|\sigma(s,X^0_{s,\xi}+\sqrt{\varepsilon_n}h(\varepsilon_n)Y^{n,v}_{s},\mathcal{L}_{X^{n}_s})\big\|^2\d s\Big)^{\frac{1}{2}}\nonumber\\&\leq \frac{C_{L_1}}{h(\varepsilon_n)}\Big[\E\Big(\int_{0}^{t}K^2_2(t,s)\cdot\big(1+|X^0_{s,\xi}|^2+\varepsilon_n h(\varepsilon_{n})|Y^{n,v}_s|^2+\mathbb{W}_2^2(\mathcal{L}_{X^{n}_s},\delta_{0})\big)\d s\Big)^{\frac{p}{2}}\Big]^{\frac{1}{p}}\nonumber\\&\leq \frac{C_{L_1}}{h(\varepsilon_n)}\Big(\int_{0}^{t}K^2_2(t,s)\cdot\big(1+|X^0_{s,\xi}|^p+|\sqrt{\varepsilon_{n}}h(\varepsilon_{n})|^p\E|Y^{n,v}_s|^p+\E\mathbb{W}_2^p(\mathcal{L}_{X^{n}_s},\delta_{0})\big)^{\frac{2}{p}}\d s\Big)^{\frac{1}{2}}\nonumber\\&\leq C_{T,L_1,L_2,\xi}\frac{1}{h(\varepsilon_{n})}+C_{T,N,L_1,L_2,\xi}\sqrt{\varepsilon_{n}}h(\varepsilon_{n}).\nonumber
		\end{align}
		Therefore
		$$\E[\Psi_t(Y^{n,v},v^n)]\leq C_{T,N,L_1,L_2,\xi}\frac{1}{h(\varepsilon_{n})}+C_{T,N,L_1,L_2,\xi}\sqrt{\varepsilon_{n}}h(\varepsilon_{n})\to 0\quad \text{as $n\to\infty$}.$$
		Hence $\E[\Psi_t(\psi,v)]=0$, which implies that $\psi$ satisfies \eqref{controlled determinstic Volterra eq-MDP} almost surely, for all $t\in T$. By Lemma \ref{11}, we have that $\psi$ has continuous paths, so it satisfies \eqref{controlled determinstic Volterra eq-MDP} for all $t\in T$, almost surely,.  Thus,  $\mathcal{J}^0_{v,N}$ consists of   unique solution of \eqref{controlled determinstic Volterra eq}. Since this definition is independent of $N$, it extends to $\mathcal{J}_v^0$. The proof is complete.
	\end{proof}
	This Lemma shows that the set $\mathcal{J}^0_v$ corresponds to the unique solution $\psi_t$ of \eqref{controlled determinstic Volterra eq-MDP}.
	\begin{lem}\label{compact level sets-MDP}
		Assume that \textbf{(H1)}-\textbf{(H3)} and \textbf{(H6)}-\textbf{(H7)} hold, the functional $\Lambda$ in \eqref{def of G} has compact level sets.
	\end{lem}
	\begin{proof}
		By Lemma \ref{unique characterised-MDP}, we have that
		\small\begin{align}
			\Lambda(\psi):=\inf\Big\{\frac{1}{2}\int_{0}^{T}|v_s|^2\d s: v\in L^2 ,\ \psi_t=\int_{0}^{t}K_1(t,s)\nabla_{\psi_s}b(s,X^0_{s,\xi},\delta_{X^0_{s,\xi}})\d s+\int_{0}^{t}K_1(t,s)\sigma(s,X^0_{s,\xi},\delta_{X^0_{s,\xi}})v_s\d s\Big\}.\nonumber
		\end{align}\normalsize
		To prove the functional $\Lambda$ defined in \eqref{def of G} has compact level sets, that is for all $N>0$, the level sets
		$$L_N:=\{\psi\in C^d_T: I(\phi)\leq N\}$$
		are compact. Fix $N>0$ and consider an arbitrary sequence $\{\psi^n\}_{n\in\mathbb{N}}\subset L_N$, we will show that there exists a converging subsequence such that the limit belongs to $L_N$.
		
		We first prove that $\{\psi^n\}_{n\in\mathbb{N}}$ is relatively  compact. According to classical Arzel$\grave{\mathrm{a}}$-Ascoli's Theorem, we only need to show that $\{\psi^n\}_{n\in\mathbb{N}}$ is uniformly bounded and equicontinuous. For all $n\in\mathbb{N}$ and $t\in [0,T]$, there exists $v^n\in L^2$ such that $\frac{1}{2}\int_{0}^{T}|v^n_t|^2dt\leq N$ and $\psi^n\in \mathcal{G}^0_{v^n}$, which implies that $v^n\in\mathcal{S}_{2N}$ and $\psi^n$ is the solution of following Volterra type McKean-Vlasov integral equation:
		$$\psi^n_t=\int_{0}^{t}K_1(t,s)\nabla_{\psi^n_s}b(s,X^0_{s,\xi},\delta_{X^0_{s,\xi}})\d s+\int_{0}^{t}K_1(t,s)\sigma(s,X^0_{s,\xi},\delta_{X^0_{s,\xi}})v^n_s\d s.$$
		By Lemmas \ref{existence-mid} and \ref{11}, we obtain $\{\psi^n\}_{n\in\mathbb{N}}$ is relatively compact which ensures that $L_N$ is relatively compact for any $N>0$.
		
		Next we show that $\{\psi^n\}_{n\in\mathbb{N}}$ is closed. Namely, we want to prove that if  $\{\psi^n\}_{n\in\mathbb{N}}$ be a converging sequence of $L_N$  with  limit  $\psi\in C^d_T$,  then    $\phi\in L_N$. Since $v^{n}\in\mathcal{S}_{2N}$ which is a compact space with respect to the weak topology. Hence there exists a subsequence $\{n_k\}_{k\in\mathbb{N}}$ such that $v^{n_k}$ converges weakly in $\{n_k\}_{k\in\mathbb{N}}$ such that $v^{n_k}$ converges weakly in $L^2$ to the limit $v\in\mathcal{S}_{2N}$ and $\lim_{k\to \infty}\psi^{n_k}=\psi$.
		
		Notice that for $v,v^{n_k}\in\mathcal{S}_{2N}$, by \textbf{(H1)-(H3)}, H$\mathrm{\ddot{o}}$lder's inequality and Lemma \ref{x0xi} we have that
		\begin{align}
			&\Big|\phi_t^{n_k}-\big(\int_{0}^{t}K_1(t,s)\nabla_{\psi_s}b(s,X^0_{s,\xi},\delta_{X^0_{s,\xi}})\d s+\int_{0}^{t}K_1(t,s)\sigma(s,X^0_{s,\xi},\delta_{X^0_{s,\xi}})v_s\d s\big)\Big|\nonumber\\&\leq \Big|\int_{0}^{t}K_1(t,s)\cdot\big(\nabla_{\psi^{n_k}_s}b(s,X^0_{s,\xi},\delta_{X^0_{s,\xi}})-\nabla_{\psi_s}b(s,X^0_{s,\xi},\delta_{X^0_{s,\xi}})\big)\d s\Big|\nonumber\\&\quad+\Big|\int_{0}^{t}K_1(t,s)\cdot \sigma(s,X^0_{s,\xi},\delta_{X^0_{s,\xi}})\cdot\big(v^{n_k}-v\big)\d s\Big|\nonumber\\&\leq \int_{0}^{t}K_1(t,s)\cdot\big\|\nabla b(s,X^0_{s,\xi},\delta_{X^0_{s,\xi}})-\nabla b(s,0,\delta_{0})+\nabla b(s,0,\delta_{0})\big\|\cdot|\psi^{n_k}_s-\psi_s|\d s\nonumber\\&\quad+\int_{0}^{t}K_1(t,s)\cdot \big\|\sigma(s,X^0_{s,\xi},\delta_{X^0_{s,\xi}})\big\|\cdot|v^{n_k}-v|\d s\nonumber\\&\leq L_6\int_{0}^{t}K_1(t,s)\cdot\big(|X^0_{s,\xi}|+\mathbb{W}_2(\delta_{X^0_{s,\xi}},\delta_{0})\big)\cdot|\psi^{n_k}_s-\psi_s|\d s+L_5\int_{0}^{t}K_1(t,s)\cdot |\psi^{n_k}_s-\psi_s|\d s\nonumber\\&\quad+\sqrt{L_2}\int_{0}^{t}K_1(t,s)\cdot\big(1+|X^0_{s,\xi}|^2+\mathbb{W}_2^2(\delta_{X^0_{s,\xi}},\delta_{0})\big)^{\frac{1}{2}}\cdot |v^{n_k}-v|\d s\nonumber\\&\leq C_{T,L_2,L_5,L_6,\xi}\sup_{t\in [0,T]}|\psi^{n_k}_t-\psi_t|\quad\text{as $n_k\to\infty$}\quad (\lim_{k\to \infty}\phi^{n_k}=\phi)\nonumber\\&\quad+C_{T,L_2,\xi}\int_{0}^{t}K_1(t,s)\cdot|v^{n_k}_s-v_s|\d s\to 0\quad\text{as $n_k\to\infty$}\quad (v^{n_k}\text{converges weakly in $L^2$ to}\ v).\nonumber
		\end{align}
		So we obtain  easily   that
		\begin{align}
			\psi_{t}&=\lim_{k\to \infty}\psi^{n_k}\nonumber\\&=\lim_{k\to \infty}\Big(\int_{0}^{t}K_1(t,s)\nabla_{\psi^n_s}b(s,X^0_{s,\xi},\delta_{X^0_{s,\xi}})\d s+\int_{0}^{t}K_1(t,s)\sigma(s,X^0_{s,\xi},\delta_{X^0_{s,\xi}})v^n_s\d s\Big)\nonumber\\&=\int_{0}^{t}K_1(t,s)\nabla_{\psi_s}b(s,X^0_{s,\xi},\delta_{X^0_{s,\xi}})\d s+\int_{0}^{t}K_1(t,s)\sigma(s,X^0_{s,\xi},\delta_{X^0_{s,\xi}})v_s\d s\nonumber.
		\end{align}
		Since $v\in\mathcal{S}_{2N}$  we see  $\psi\in L_N$, which concludes the proof of the closeness  and therefore of the compactness of $L_N$.
	\end{proof}
Combining Lemma \ref{Y vare v t s} and Lemma \ref{compact level sets-MDP}, Theorem \ref{MDP} is a disect result of Lemma \ref{general LDP}.
\appendix
\section{Appendix}\label{Appendix}
\subsection{Proof of Lemma  \ref{existence}}
\begin{proof} 
	We utilize the successive approximation scheme. Define recursively $(\phi^n)_{n\geq 1}$ as follows: $\phi^1_t:=\xi$, $t\in [0,T]$ and for $n=0, 1, 2, \cdots$ %each $n\in\mathbb{Z}_+$,
	\begin{align}\label{Picard}
		\phi^{n+1}_t:=\xi+\int_{0}^{t}K_1(t,s)b(s,\phi_s^n,\delta_{X^0_s})\d s+\int_{0}^{t}K_1(t,s)\sigma(s,\phi_s^n,\delta_{X^0_s})v_s\d s, t\in [0,T].
	\end{align}
	By H$\mathrm{\ddot{o}}$lder's inequality, Lemma \ref{x0xi}, \textbf{(H1)} and \textbf{(H3)}  we have
	\begin{align}\label{estimate of X_n}
		|\phi_t^{n+1}|^2&\leq 3|\xi|^2	+3\Big|\int_{0}^{t}K_1(t,s)b(s,\phi_s^n,\delta_{X^0_{s,\xi}})\d s\Big|^2+3\Big|\int_{0}^{t}K_1(t,s)\sigma(s,\phi_s^n,\delta_{X^0_{s,\xi}})v_s\d s\Big|^2\nonumber\\&\leq 3|\xi|^2 +3\int_{0}^{t}K_1(t,s)\cdot\big|b(s,\phi_s^n,\delta_{X^0_{s,\xi}})\big|^2\d s\cdot\Big|\int_{0}^{t}K_1(t,s)\d s\Big|\nonumber\\&\quad+3N\int_{0}^{t}K_1^2(t,s)\cdot\big\|\sigma(s,\phi_s^n,\delta_{X^0_{s,\xi}})\big\|^2\d s\nonumber\\&\leq 3|\xi|^2+C_{T,N,L_2}+C_{T,L_2}\int_{0}^{t}K_1(t,s)\cdot|\phi_s^{n}|^2\d s+C_{p,T,N,L_2}\int_{0}^{t}K^2_1(t,s)\cdot|\phi_s^{n}|^2\d s\nonumber\\&\quad+C_{T,N,L_2}\int_{0}^{t}|X^0_{s,\xi}|^2\d s\nonumber\\&\leq C_{T,N,L_2}(1+|\xi|^2)+\int_{0}^{t}\big(C_{T,L_2}K_1(t,s)+C_{T,N,L_2}K^2_1(t,s)\big)\cdot|\phi^n_s|^2\d s.
	\end{align}
	This  means 
	$$f_m(t)\leq C_{T,N,L_2}(1+|\xi|^2)+\int_{0}^{t}K_{1,1}(t,s)\cdot f_m(s)\d s,$$
	where 	$$f_m(t):=\sup_{n=1,\cdots,m}|\phi_s^n|^2,$$
	%then it is clear that $f_m<\infty$ for all $m\in \mathbb{N}$ and we also easily get
and  $K_{1,1}:=C_{T,L_2}K_1(t,s)+C_{T,N,L_2}K_1^2(t,s)$ %and the constants $C_{T,L_2}$ and $C_{T,N,L_2}$ are 
is  independent of $m$.
	
	Let $R^{K_{1,1}}$ be defined by \eqref{Resolvent} in terms of $K_{1,2}$.  By Remark \ref{condition of Volterra}, Lemma \ref{Gronwall}, \eqref{definition of resolvent}, \eqref{finite of resolvent} and \textbf{(H1)}, we obtain that for almost all $t\in [0,T]$,
	\begin{align}\label{bounded}
		\sup_{n\in\mathbb{N}}|\phi^n_t|^2&=\lim_{m\to\infty}f_m(t)\leq C_{p,T,N,L_2}\Big(|\xi|^2+1+\int_{0}^{t}R^{K_{1,1}}(t,s)\cdot (|\xi|^2+1)\d s\Big)\nonumber\\&\leq C_{T,N,L_2}(1+|\xi|^2).
	\end{align}
	On the other hand, we set
	$$Z_{n,m}(t):=\phi^n_t-\phi^m_t$$
	and
	$$f(t):=\limsup_{n,m\to\infty}|Z_{n,m}(t)|^2.$$
	First,    by \eqref{bounded} we see
	\[
	f(t)\le \sup_{n,m }|Z_{n,m}(t)|^2 
	\le 2 \sup_{n  }|\phi^n_t|^p <\infty\,. 
	\] 
	By H$\mathrm{\ddot{o}}$lder's inequality, \textbf{(H1)} and \textbf{(H2)} we have
	\begin{align}
		|Z_{n+1,m+1}(t)|^2&=|\phi^{n+1}_t-\phi^{m+1}_t|^2\nonumber\\&\leq 2\Big|\int_{0}^{t}|b(t,s,\phi^{n}_s,\delta_{X^0_{s,\xi}})-b(t,s,\phi^{m}_s,\delta_{X^0_{s,\xi}})|\d s\Big|^2\nonumber\\&\quad+2\Big|\int_{0}^{t}\big(\sigma(t,s,\phi^{n}_s,\delta_{X^0_{s,\xi}})-\sigma(t,s,\phi^{m}_s,\delta_{X^0_{s,\xi}})\big)v_s\d s\Big|^2\nonumber\\&\leq C_{L_1}\Big|\int_{0}^{t}K_1(t,s)\cdot\big|\phi^{n}_s-\phi^{m}_s|\d s\Big|^2\nonumber\\&\quad+2N\Big|\int_{0}^{t}K_1^2(t,s)\cdot\big\|\sigma(t,s,\phi^{n}_s,\delta_{X^0_{s,\xi}})-\sigma(t,s,\phi^{m}_s,\delta_{X^0_{s,\xi}})\big\|^2\d s\Big|\nonumber\\&\leq  C_{L_1}\int_{0}^{t}K_1(t,s)\cdot \big|\phi^{n}_s-\phi^{m}_s\big|^2\d s\cdot\int_{0}^{t}K_1(t,s)\d s\nonumber\\&\quad+C_{N,L_1}\int_{0}^{t}K^2_1(t,s)\cdot \big|\phi^{n}_s-\phi^{m}_s\big|^2\d s\nonumber\\&\leq C_{T,L_1}\int_{0}^{t}K_1(t,s)\cdot |\phi_s^n-\phi_s^m|^2\d s+C_{T,N,L_1}\int_{0}^{t}K^2_1(t,s)\cdot |\phi_s^n-\phi_s^m|^2\d s\nonumber\\&\leq C_{T,N,L_1}\int_{0}^{t}\Big(K_1(t,s)+K_1^2(t,s)\Big)\cdot |Z_{n,m}(t)|^2 \d s.\nonumber
	\end{align}
	By \eqref{bounded}, \textbf{(H1)} and applying Fatou's lemma, we have
	$$f(t)\leq C_{T,N,L_1}\int_{0}^{t}\Big(K_1(t,s)+K_1^2(t,s)\Big)\cdot f(s) \d s.$$
	
	By Remark \ref{condition of Volterra} and Lemma \ref{Gronwall}, we get for almost all $t\in [0,T]$,
	$$f(t)=\limsup_{n,m\to\infty}|Z_{n,m}(t)|^p=0.$$
	Hence, there exists a $\phi_t$ such that for almost all $t\in [0,T]$,
	$$\lim_{m\to\infty}|\phi^n_t-\phi_t|^2=0.$$
	Now taking limits in  \eqref{Picard} gives the existence. The uniqueness follows from a similar calculation  and Lyapunov inequality. % gives desired estimate.
\end{proof}
\subsection{Proof of Lemma \ref{zxi}[(ii)]}
	\begin{proof}
	By BDG's inequality, H$\mathrm{\ddot{o}}$lder's inequality and \textbf{(H1)}-\textbf{(H6)} we have
	\begin{align}
		&\E|Z^{\varepsilon}_{t,\xi}-Z^{\varepsilon}_{t,\eta}|^{p}\nonumber\\&=\E\Big|\int_{0}^{t}K_1(t,s)\cdot\frac{b(s,X^\varepsilon_{s,\xi},\mathcal{L}_{X^\varepsilon_{s,\xi}})-b(s,X^0_{s,\xi},\delta_{X^0_{s,\xi}})}{\sqrt{\varepsilon}}\d s+\int_{0}^{t}K_2(t,s)\sigma(s,X^\varepsilon_{s,\xi},\mathcal{L}_{X^\varepsilon_{s,\xi}})\d W_s\nonumber\\&\quad-\int_{0}^{t}K_1(t,s)\cdot\frac{b(s,X^\varepsilon_{s,\eta},\mathcal{L}_{X^\varepsilon_{s,\eta}})-b(s,X^0_{s,\eta},\delta_{X^0_{s,\eta}})}{\sqrt{\varepsilon}}\d s-\int_{0}^{t}K_2(t,s)\sigma(s,X^\varepsilon_{s,\eta},\mathcal{L}_{X^\varepsilon_{s,\eta}})\d W_s\Big|^{p}\nonumber\\&\leq 2^{p-1}\E\Big|\int_{0}^{t}K_1(t,s)\cdot\Big(\frac{b(s,X^\varepsilon_{s,\xi},\mathcal{L}_{X^\varepsilon_{s,\xi}})-b(s,X^0_{s,\xi},\delta_{X^0_{s,\xi}})}{\sqrt{\varepsilon}}-\frac{b(s,X^\varepsilon_{s,\eta},\mathcal{L}_{X^\varepsilon_{s,\eta}})-b(s,X^0_{s,\eta},\delta_{X^0_{s,\eta}})}{\sqrt{\varepsilon}}\Big)\d s\Big|^{p}\nonumber\\&\quad+2^{p-1}\E\Big|\int_{0}^{t}K_2(t,s)\cdot\big(\sigma(s,X^\varepsilon_{s,\xi},\mathcal{L}_{X^\varepsilon_{s,\xi}})-\sigma(s,X^\varepsilon_{s,\eta},\mathcal{L}_{X^\varepsilon_{s,\eta}}) \big)\d W_s\Big|^{p}\nonumber\\&\leq 2^{p-1}\E\Bigg|\int_{0}^{t}K_1(t,s)\cdot\Big(\frac{b(s,X^\varepsilon_{s,\xi},\mathcal{L}_{X^\varepsilon_{s,\xi}})-b(s,X^{\varepsilon}_{s,\xi},\delta_{X^0_{s,\xi}})}{\sqrt{\varepsilon}}+\frac{b(s,X^{\varepsilon}_{s,\xi},\delta_{X^0_{s,\xi}})-b(s,X^0_{s,\xi},\delta_{X^0_{s,\xi}})}{\sqrt{\varepsilon}}\Big)\d s\nonumber\\&\quad-\int_{0}^{t}K_1(t,s)\cdot\Big(\frac{b(s,X^\varepsilon_{s,\eta},\mathcal{L}_{X^\varepsilon_{s,\eta}})-b(s,X^{\varepsilon}_{s,\eta},\delta_{X^0_{s,\eta}})}{\sqrt{\varepsilon}}+\frac{b(s,X^{\varepsilon}_{s,\eta},\delta_{X^0_{s,\eta}})-b(s,X^0_{s,\eta},\delta_{X^0_{s,\eta}})}{\sqrt{\varepsilon}}\Big)\d s\Bigg|^{p}\nonumber\\&\quad+C_{p}\E\Big(\int_{0}^{t}K^2_2(t,s)\cdot\big\|\sigma(s,X^\varepsilon_{s,\xi},\mathcal{L}_{X^\varepsilon_{s,\xi}})-\sigma(s,X^\varepsilon_{s,\eta},\mathcal{L}_{X^\varepsilon_{s,\eta}})\big\|^{2}\d s\Big)^{\frac{p}{2}}\nonumber\\&\leq 2^{p-1}\E\Bigg|\int_{0}^{t}K_1(t,s)\cdot\Big(\int_{0}^{1}\E\big\langle D^Lb(s,X^{\varepsilon}_{s,\xi},\mathcal{L}_{R^{\xi}_s(r)})(R^{\xi}_s(r), Z^{\varepsilon}_{s,\xi}\big\rangle \d r\nonumber\\&\quad+\int_{0}^{1}\nabla_{Z^{\varepsilon}_{s,\xi}}b(s,X^0_{s,\xi}+u(X^{\varepsilon}_{s,\xi}-X^0_{s,\xi}),\delta_{X^0_{t,\xi}})\d u\Big)\d s\nonumber\\&\quad-\int_{0}^{t}K_1(t,s)\cdot\Big(\int_{0}^{1}\E\big\langle D^Lb(s,X^{\varepsilon}_{s,\eta},\mathcal{L}_{R^{\eta}_s(r)})(R^{\eta}_s(r),Z^{\varepsilon}_{s,\eta}\big\rangle \d r\nonumber\\&\quad+\int_{0}^{1}\nabla_{Z^{\varepsilon}_{s,\eta}}b(s,X^0_{s,\eta}+u(X^{\varepsilon}_{s,\eta}-X^0_{s,\eta}),\delta_{X^0_{t,\eta}})\d u\Big)\d s\Bigg|^{p}\nonumber\\&\quad+C_{p}\E\int_{0}^{t}K_2^2(t,s)\cdot\big\|\sigma(s,X^\varepsilon_{s,\xi},\mathcal{L}_{X^\varepsilon_{s,\xi}})-\sigma(s,X^\varepsilon_{s,\eta},\mathcal{L}_{X^\varepsilon_{s,\eta}})\big\|^{p}\d s\cdot\Big|\int_{0}^{t}K_2^2(t,s)\d s\Big|^{\frac{p}{2}-1}\nonumber\\&\leq2^{p-1}\E\Bigg|\int_{0}^{t}K_1(t,s)\cdot\Big(\int_{0}^{1}\E\big\langle D^Lb(s,X^{\varepsilon}_{s,\xi},\mathcal{L}_{R^{\xi}_s(r)})(R^{\xi}_s(r)), Z^{\varepsilon}_{s,\xi}\big\rangle-\big\langle D^Lb(s,X^{\varepsilon}_{s,\eta},\mathcal{L}_{R^{\eta}_s(r)})(R^{\eta}_s(r)), Z^{\varepsilon}_{s,\eta}\big\rangle \d r\nonumber\\&\quad+\int_{0}^{1}\big(\nabla_{Z^{\varepsilon}_{s,\xi}}b(s,M^{\xi}_s(u),\delta_{X^0_{t,\xi}})-\nabla_{Z^{\varepsilon}_{s,\eta}}b(s,M^{\eta}_s(u),\delta_{X^0_{t,\eta}})\big)\d u\Big)\d s\Bigg|^{p}\nonumber\\&\quad+C_{p,T,L_1}\int_{0}^{t}K_2^2(t,s)\cdot \E|X^{\varepsilon}_{s,\xi}-X^{\varepsilon}_{s,\eta}|^{p}\d s\nonumber\\&\leq2^{p-1}\E\int_{0}^{t}K_1(t,s)\cdot\Bigg|\int_{0}^{1}\E\big(\big\langle D^Lb(s,X^{\varepsilon}_{s,\xi},\mathcal{L}_{R^{\xi}_s(r)})(R^{\xi}_s(r),Z^{\varepsilon}_{s,\xi}\big\rangle-\big\langle D^Lb(s,X^{\varepsilon}_{s,\eta},\mathcal{L}_{R^{\eta}_s(r)})(R^{\eta}_s(r)),Z^{\varepsilon}_{s,\eta}\big\rangle\big)\d r\nonumber\\&\quad+\int_{0}^{1}\big(\nabla_{Z^{\varepsilon}_{s,\xi}}b(s,M^{\xi}_s(u),\delta_{X^0_{t,\xi}})-\nabla_{Z^{\varepsilon}_{s,\eta}}b(s,M^{\eta}_s(u),\delta_{X^0_{t,\eta}})\big)\d u\Bigg|^{p}\d s\cdot\Big|\int_{0}^{t}K_1(t,s)\d s\Big|^{p-1}\nonumber\\&\quad+C_{p,T,L_1}\int_{0}^{t}K_2^2(t,s)\cdot \E|X^{\varepsilon}_{s,\xi}-X^{\varepsilon}_{s,\eta}|^{p}\d s\nonumber\\&\leq C_{p,T}\E\int_{0}^{t}K_1(t,s)\cdot\int_{0}^{1}\Big|\big\langle D^Lb(s,X^{\varepsilon}_{s,\xi},\mathcal{L}_{R^{\xi}_s(r)})(R^{\xi}_s(r),Z^{\varepsilon}_{s,\xi}\big\rangle-\big\langle D^Lb(s,X^{\varepsilon}_{s,\eta},\mathcal{L}_{R^{\eta}_s(r)})(R^{\eta}_s(r)),Z^{\varepsilon}_{s,\eta}\big\rangle\Big|^{p}\d r\d s\nonumber\\&\quad+C_{p,T}\E\int_{0}^{t}K_1(t,s)\cdot\int_{0}^{1}\Big|\nabla_{Z^{\varepsilon}_{s,\xi}}b(s,M^{\xi}_s(u),\delta_{X^0_{t,\xi}})-\nabla_{Z^{\varepsilon}_{s,\eta}}b(s,M^{\eta}_s(u),\delta_{X^0_{t,\eta}})\Big|^{p}\d u\d s\nonumber\\&\quad+C_{p,T,L_1}\int_{0}^{t}K_2^2(t,s)\cdot \E|X^{\varepsilon}_{s,\xi}-X^{\varepsilon}_{s,\eta}|^{p}\d s\nonumber\\&\leq C_{p,T}\E\int_{0}^{t}K_1(t,s)\cdot\Big|\int_{0}^{1}\E\big(\big\langle D^Lb(s,X^{\varepsilon}_{s,\xi},\mathcal{L}_{R^{\xi}_s(r)})(R^{\xi}_s(r),Z^{\varepsilon}_{s,\xi}\big\rangle-\big\langle D^Lb(s,X^{\varepsilon}_{s,\eta},\mathcal{L}_{R^{\eta}_s(r)})(R^{\eta}_s(r)),Z^{\varepsilon}_{s,\xi}\big\rangle\big)\Big|^p\d r\d s\nonumber\\&\quad+C_{p,T}E\int_{0}^{t}K_1(t,s)\cdot\Big|\int_{0}^{1}\E\big(\big\langle D^Lb(s,X^{\varepsilon}_{s,\eta},\mathcal{L}_{R^{\eta}_s(r)})(R^{\eta}_s(r),Z^{\varepsilon}_{s,\xi}\big\rangle-\big\langle D^Lb(s,X^{\varepsilon}_{s,\eta},\mathcal{L}_{R^{\eta}_s(r)})(R^{\eta}_s(r)),Z^{\varepsilon}_{s,\eta}\big\rangle\big)\Big|^p\d r\d s\nonumber\\&\quad+C_{p,T}\E\int_{0}^{t}K_1(t,s)\cdot\int_{0}^{1}\Big|\nabla_{Z^{\varepsilon}_{s,\xi}}b(s,M^{\xi}_s(u),\delta_{X^0_{t,\xi}})-\nabla_{Z^{\varepsilon}_{s,\xi}}b(s,M^{\eta}_s(u),\delta_{X^0_{t,\eta}})\Big|^{p}\d u\d s\nonumber\\&\quad+C_{p,T}\E\int_{0}^{t}K_1(t,s)\cdot\int_{0}^{1}\Big|\nabla_{Z^{\varepsilon}_{s,\xi}}b(s,M^{\eta}_s(u),\delta_{X^0_{t,\eta}})-\nabla_{Z^{\varepsilon}_{s,\eta}}b(s,M^{\eta}_s(u),\delta_{X^0_{t,\eta}})\Big|^{p}\d u\d s\nonumber\\&\quad+C_{p,T,L_1}\int_{0}^{t}K_2^2(t,s)\cdot \E|X^{\varepsilon}_{s,\xi}-X^{\varepsilon}_{s,\eta}|^{p}\d s\nonumber\\&\leq C_{p,T,L_4}\E\int_{0}^{t}K_1(t,s)\cdot\int_{0}^{1}\Big||\mathcal{L}_{R_s^{\xi}(r)}-\mathcal{L}_{R_s^{\eta}(r)}|+\mathbb{W}_2(\mathcal{L}_{R_s^{\xi}(r)},\mathcal{L}_{R_s^{\eta}(r)})+\big(\E|R_s^{\xi}(r)-R^{\eta}_{s}(r)|^2\big)^{\frac{1}{2}}\Big|^{p}\d r\cdot \E|Z^{\varepsilon}_{s,\xi}|^p\d s\nonumber\\&\quad+C_{p,T,L_3}\E\int_{0}^{t}K_1(t,s)\cdot|Z^{\varepsilon}_{s,\xi}-Z^{\varepsilon}_{s,\eta}|^p\d s+C_{p,T,L_5}\E\int_{0}^{t}K_1(t,s)\cdot\int_{0}^{1}|M^{\xi}_s(u)-M^{\eta}_s(u)|^{p}\d u\cdot |Z^{\varepsilon}_{s,\xi}|^p\d s\nonumber\\&\quad+C_{p,T,L_3}\int_{0}^{t}K_1(t,s)\cdot \E|Z^{\varepsilon}_{s,\xi}-Z^{\varepsilon}_{s,\eta}|^p\d s+C_{p,T,L_1}\int_{0}^{t}K_2^2(t,s)\cdot \E|X^{\varepsilon}_{s,\xi}-X^{\varepsilon}_{s,\eta}|^{p}\d s\nonumber\\&\leq C_{p,T,L_4}\int_{0}^{t}K_1(t,s)\cdot \E\Big(|X^{0}_{s,\xi}-X^{0}_{s,\eta}|^p+|X^{\varepsilon}_{s,\xi}-X^{\varepsilon}_{s,\eta}|^p\Big)\cdot \E|Z^{\varepsilon}_{s,\xi}|^p\d s\nonumber\\&\quad+C_{p,T,L_3}\int_{0}^{t}K_1(t,s)\cdot \E|Z^{\varepsilon}_{s,\xi}-Z^{\varepsilon}_{s,\eta}|^p\d s\nonumber\\&\quad+C_{p,T,L_5}\E\int_{0}^{t}K_1(t,s)\cdot\Big(|X^{0}_{s,\xi}-X^{0}_{s,\eta}|^p+|X^{\varepsilon}_{s,\xi}-X^{\varepsilon}_{s,\eta}|^p\Big)\cdot|Z^{\varepsilon}_{s,\xi}|^p\d s\nonumber\\ &\quad+C_{p,T,L_3}\int_{0}^{t}K_1(t,s)\cdot \E|Z^{\varepsilon}_{s,\xi}-Z^{\varepsilon}_{s,\eta}|^p\d s+C_{p,T,L_1}\int_{0}^{t}K_2^2(t,s)\cdot \E|X^{\varepsilon}_{s,\xi}-X^{\varepsilon}_{s,\eta}|^{p}\d s\nonumber\\&\leq C_{p,T,L_4}\int_{0}^{t}K_1(t,s)\cdot \E|X^{0}_{s,\xi}-X^{0}_{s,\eta}|^p\cdot \E|Z^{\varepsilon}_{s,\xi}|^p\d s+C_{p,T,L_4}\int_{0}^{t}K_1(t,s)\cdot \E|X^{\varepsilon}_{s,\xi}-X^{\varepsilon}_{s,\eta}|^p\cdot \E|Z^{\varepsilon}_{s,\xi}|^p\d s\nonumber\\&\quad+C_{p,T,L_5}\int_{0}^{t}K_1(t,s)\cdot \E\Big(\big(|X^{0}_{s,\xi}-X^{0}_{s,\eta}|^p+|X^{\varepsilon}_{s,\xi}-X^{\varepsilon}_{s,\eta}|^p\big)\cdot |Z^{\varepsilon}_{s,\xi}|^p\Big)\d s\nonumber\\&\quad+C_{p,T,L_3}\int_{0}^{t}K_1(t,s)\cdot \E|Z^{\varepsilon}_{s,\xi}-Z^{\varepsilon}_{s,\eta}|^p\d s+C_{p,T,L_1}\int_{0}^{t}K_2^2(t,s)\cdot \E|X^{\varepsilon}_{s,\xi}-X^{\varepsilon}_{s,\eta}|^{p}\d s\nonumber\\&\leq C_{p,T,L_4}\int_{0}^{t}K_1(t,s)\cdot \E|X^{0}_{s,\xi}-X^{0}_{s,\eta}|^p\cdot \E|Z^{\varepsilon}_{s,\xi}|^p\d s+C_{p,T,L_4}\int_{0}^{t}K_1(t,s)\cdot \E|X^{\varepsilon}_{s,\xi}-X^{\varepsilon}_{s,\eta}|^p\cdot \E|Z^{\varepsilon}_{s,\xi}|^p\d s\nonumber\\&\quad+C_{p,T,L_5}\int_{0}^{t}K_1(t,s)\cdot\Big(\E|X^{0}_{s,\xi}-X^{0}_{s,\eta}|^{2p}\Big)^{\frac{1}{2}}\cdot\Big(\E|Z^{\varepsilon}_{s,\xi}|^{2p}\Big)^{\frac{1}{2}}\d s\nonumber\\&\quad+C_{p,T,L_5}\int_{0}^{t}K_1(t,s)\cdot\Big(\E|X^{\varepsilon}_{s,\xi}-X^{\varepsilon}_{s,\eta}|^{2p}\Big)^{\frac{1}{2}}\cdot\Big(\E|Z^{\varepsilon}_{s,\xi}|^{2p}\Big)^{\frac{1}{2}}\d s\nonumber\\&\quad+C_{p,T,L_3}\int_{0}^{t}K_1(t,s)\cdot \E|Z^{\varepsilon}_{s,\xi}-Z^{\varepsilon}_{s,\eta}|^p\d s+C_{p,T,L_1}\int_{0}^{t}K_2^2(t,s)\cdot \E|X^{\varepsilon}_{s,\xi}-X^{\varepsilon}_{s,\eta}|^{p}\d s\nonumber\\&\leq C_{p,T,L_1,L_2,L_4,L_5}(1+|\xi|^p)|\xi-\eta|^p+C_{p,T,L_3}\int_{0}^{t}K_1(t,s)\cdot \E|Z^{\varepsilon}_{s,\xi}-Z^{\varepsilon}_{s,\eta}|^p\d s,\nonumber
	\end{align}
	where $R^{\xi}_s(r):=X^0_{s,\xi}+r(X^{\varepsilon}_{s,\xi}-X^0_{s,\xi}),R^{\eta}_s(r):=X^0_{s,\eta}+r(X^{\varepsilon}_{s,\eta}-X^0_{s,\eta}),M_s^{\xi}(u):=X^0_{s,\xi}+u(X^{\varepsilon}_{s,\xi}-X^0_{s,\xi}),M^{\eta}_s(u):=X^0_{s,\eta}+u(X^{\varepsilon}_{s,\eta}-X^0_{s,\eta})\ r,u\in [0,T]$, and in the last inequality Lemma \ref{xxi}, Lemma \ref{x0xi} and Lemma \ref{zxi} are used.
	
	Finally, let  $R^{K^{*}_1}$ be defined by \eqref{Resolvent} with $K_1^{*}:=C_{p,T,L_3}K_1(t,s)$.  From  Remark \ref{condition of Volterra} and Lemma \ref{Gronwall}, it follows that
	\begin{align}
		\E|Z^{\varepsilon}_{t,\xi}-Z^{\varepsilon}_{t,\eta}|^{p}&\leq C_{p,T,L_1,L_2,L_4,L_5}(1+|\xi|^p)|\xi-\eta|^p\nonumber\\&+\int_{0}^{t}R^{K^{*}_{1}}(t,s)\cdot\big(C_{p,T,L_1,L_2,L_4,L_5}(1+|\xi|^p)|\xi-\eta|^p\big)\d s\nonumber\\&\leq C_{p,T,L_1,L_2,L_3,L_4,L_5}(1+|\xi|^{p})|\xi-\eta|^p.\nonumber
	\end{align}
\end{proof}
\subsection{Proof of Theorem \ref{existence-mid}}
\begin{proof} 
	We utilize the successive approximation scheme. Define recursively $(\psi^n)_{n\geq 1}$ as follows: $\psi^1_t:=0$, $t\in [0,T]$ and for each $n\in\mathbb{N}_+$,
	\begin{align}\label{Picard-mdp}
		\psi^{n+1}_t:=\int_{0}^{t}K_1(t,s)\nabla_{\psi^n_s}b(s,X^0_{s,\xi},\delta_{X^0_{s,\xi}})\d s+\int_{0}^{t}K_2(t,s)\sigma(s,X^0_{s,\xi},\delta_{X^0_{s,\xi}})v_s\d s, t\in [0,T].
	\end{align}
	By H$\mathrm{\ddot{o}}$lder's inequality, \textbf{(H1)}, \textbf{(H3)}, \textbf{(H6)}-\textbf{(H7)} and Lemma \ref{x0xi} we have
	\begin{align}\label{estimate of psi_n}
		|\psi_t^{n+1}|^2&\leq 	2\Big|\int_{0}^{t}K_1(t,s)\nabla_{\psi^n_s}b(s,X^0_{s,\xi},\delta_{X^0_{s,
		\xi}})\d s\Big|^2+2\Big|\int_{0}^{t}K_2(t,s)\sigma(s,X^0_{s,\xi},\delta_{X^0_{s,\xi}})v_s\d s\Big|^2\nonumber\\&\leq 2\int_{0}^{t}K_1(t,s)\big\|\nabla_{\psi^n_s}b(s,X^0_{s,\xi},\delta_{X^0_{s,\xi}})\big\|^2\d s\cdot\Big|\int_{0}^{t}K_1(t,s)\d s\Big|\nonumber\\&\quad+2N\int_{0}^{t}K_1^2(t,s)\big\|\sigma(s,X_{s,\xi}^0,\delta_{X^0_{s,\xi}})\big\|^2\d s\nonumber\\&\leq C_{T}\int_{0}^{t}K_1(t,s)\cdot\big\|\nabla_{\psi^n_s}b(s,X^0_{s,\xi},\delta_{X^0_{s,\xi}})-\nabla_{\psi^n_s}b(s,0,\delta_{0})+\nabla_{\psi^n_s}b(s,0,\delta_{0})\big\|^2\d s\nonumber\\&\quad+C_{T,N}\int_{0}^{t}K_2^2(t,s)\cdot\big\|\sigma(s,X_{s,\xi}^0,\delta_{X^0_{s,\xi}})\big\|^2\d s\nonumber\\&\leq C_{T,L_5}\int_{0}^{t}K_1(t,s)\cdot|X_{s,\xi}^0|^2\cdot |\psi^n_s|^2\d s+C_{T,L_6                                                                                                                                                                                                                                                                                                                                                                                                                                                                                                                                                                                   }\int_{0}^{t}K_1(t,s)\cdot |\psi^n_s|^2\d s\nonumber\\&\quad+C_{T,N,L_2}\int_{0}^{t}K_2^2(t,s)\cdot\big(1+|X_{s,\xi}^0|^2\big)\d s\nonumber\\&\leq C_{T,L_2,L_5,L_6}\big(1+|\xi|^2\big)\int_{0}^{t}K_1(t,s)\cdot|\psi^n_s|^2\d s+C_{T,N,L_2}\big(1+|\xi|^2\big)\nonumber\\&:=C_{T,N,L_2}\big(1+|\xi|^2\big)+\int_{0}^{t}\tilde{K}_1(t,s)\cdot|\psi^n_s|^2\d s,
	\end{align}
	where $\tilde{K}_1(t,s):=C_{T,L_2,L_5,L_6}\big(1+|\xi|^2\big)K_1(t,s)$. If we define
	$$g_m(t):=\sup_{n=1,\cdots,m}|\psi_s^n|^2,$$
	then  
%	it is clear that $g_m<\infty$ for all $m\in \mathbb{N}$ and we also easily get
	$$g_m(t)\leq C_{T,N,L_2}(1+|\xi|^2)+\int_{0}^{t}\tilde{K}_1(t,s)\cdot g_m(s)\d s,$$
	where the constant $C_{T,N,L_2}$ is independent of $m$.
	
	Let $R^{\tilde{K}_1}$ be defined by \eqref{Resolvent} in terms of $\tilde{K}_{1}$. Note that by \eqref{definition of R_n}, \eqref{definition of resolvent}, \eqref{finite of resolvent} and \textbf{(H1)}, when $n=1$ we have
	\begin{align}
		\int_{0}^{T}R_1^{\tilde{K}_{1}}(t,s)g_m(s)\d s=\int_{0}^{T}\tilde{K}_{1}(t,s)g_m(s)\d s<\infty,\nonumber
	\end{align}
	and
	\begin{align}
		&\int_{0}^{t}R^{\tilde{K}_{1}}(t,s)\cdot |\xi|^2\d s\nonumber\\&=\int_{0}^{t}\tilde{K}_{1}(t,s)\cdot |\xi|^2\d s+\int_{0}^{t}\Big(\int_{s}^{t}R^{\tilde{K}_{1}}(t,u)\tilde{K}_{1}(u,s)\d u\Big)\cdot |\xi|^2\d s\nonumber\\&=\int_{0}^{t}\tilde{K}_{1}(t,s)\d s\cdot |\xi|^2+\int_{0}^{t}R^{\tilde{K}_{1}}(t,u)\Big(\int_{0}^{u}\tilde{K}_{1}(u,s)ds\Big)\d u\cdot |\xi|^2<\infty.\nonumber
	\end{align}
	Hence, by Lemma \ref{Gronwall}, \eqref{definition of resolvent}, \eqref{finite of resolvent} and \textbf{(H1)}, we obtain that for almost all $t\in [0,T]$,
	\begin{align}\label{bounded-mdp}
		\sup_{n\in\mathbb{N}}|\psi^n_t|^2&=\lim_{m\to\infty}g_m(t)\leq C_{T,N,L_2}\big(1+|\xi|^2\big)+\int_{0}^{t}R^{\tilde{K}_{1}}(t,s)\cdot (1+|\xi|^2)\d s\nonumber\\&\leq C_{T,N,L_2,L_5,L_6}(1+|\xi|^2).
	\end{align}
Denote 
	$$W_{n,m}(t):=\psi^n_t-\psi^m_t$$
	and
	$$g(t):=\limsup_{n,m\to\infty}|W_{n,m}(t)|^2.$$
	By H$\mathrm{\ddot{o}}$lder's inequality, \textbf{(H1)} and \textbf{(H6)}-\textbf{(H7)} we have
	\begin{align}
		|W_{n+1,m+1}(t)|^2&=|\psi^{n+1}_t-\psi^{m+1}_t|^2\nonumber\\&\leq \Big|\int_{0}^{t}K_1(t,s)\cdot\big\|\nabla_{\psi^n_s}b(s,X^0_{s,\xi},\delta_{X^0_{s,\xi}})-\nabla_{\psi^m_s}b(s,X^0_{s,\xi},\delta_{X^0_{s,\xi}})\big\|\d s\Big|^2\nonumber\\&\leq \int_{0}^{t}K_1(t,s)\cdot\big\|\nabla_{\psi^n_s}b(s,X^0_{s,\xi},\delta_{X^0_{s,\xi}})-\nabla_{\psi^m_s}b(s,X^0_{s,\xi},\delta_{X^0_{s,\xi}})\big\|^2\d s\cdot\int_{0}^{t}K_1(t,s)\d s\nonumber\\&\leq C_{T}\int_{0}^{t}K_1(t,s)\cdot|\psi_s^n-\psi_s^m|^2\cdot\big\|\nabla b(s,X^0_{s,\xi},\delta_{X^0_{s,\xi}})-\nabla b(s,0,\delta_{0})+\nabla b(s,0,\delta_{0})\big\|^2\d s\nonumber\\&\leq C_{T,L_2,L_5,L_6}\big(1+|\xi|^2\big)\int_{0}^{t}K_1(t,s)\cdot W_{n,m}(s)\d s.\nonumber
	\end{align}
	By \eqref{bounded-mdp}, \textbf{(H1)} and applying Fatou's lemma, we have
	$$g(t)\leq C_{T,N,L_2,L_5,L_6}\int_{0}^{t}K_1(t,s)\cdot g(s) \d s.$$
	By using Lemma \ref{Gronwall}, we get for almost all $t\in [0,T]$,
	$$g(t)=\limsup_{n,m\to\infty}|W_{n,m}(t)|^2=0.$$
	Hence, there exists a $\phi_t$ such that for almost all $t\in [0,T]$,
	$$\lim_{m\to\infty}|\psi^n_t-\psi_t|^2=0.$$
	Now taking limit in  \eqref{Picard-mdp} gives the existence. The uniqueness follows from a similar calculation. Lyapunov inequality gives desired estimate.
\end{proof}

	\section*{Conflicts of interests}
	The authors declare no conflict of interests.

		{\bf Acknowledgement} S. Liu was supported by China Scholarship Council. Y. Hu was supported by the NSERC discovery fund and a Centennial  fund of University of Alberta. H. Gao was supported in part by the NSFC Grant Nos. 12171084, the Jiangsu Provincial Scientific Research Center of Applied Mathematics under Grant No. BK20233002 and  the fundamental Research
		Funds for the Central Universities No. RF1028623037.
	%  	\bibliographystyle{abbrvnat}
	
	% \bibliographystyle{accountm}
	%	\bibliographystyle{abbrv}
	% \bibliographystyle{nar}
	%\bibliographystyle{plain}
 	\bibliographystyle{amsplain}
	
	\bibliography{ref}
	
\end{document}